\documentclass[11pt, letter]{amsart}
\usepackage{style/package,style/command,style/bibstyle}

\begin{document}


\title[Entropy density and large deviation principles]
{
    Entropy density and large deviation principles\\ without upper semi-continuity of entropy
}

\author{Zhiqiang~Li \and Xianghui~Shi}

\thanks{Z.~Li and X.~Shi were partially supported by NSFC Nos.~12101017, 12090010, 12090015, and BJNSF No.~1214021.}
    
\address{Zhiqiang~Li, School of Mathematical Sciences \& Beijing International Center for Mathematical Research, Peking University, Beijing 100871, China}
\email{zli@math.pku.edu.cn}

\address{Xianghui~Shi, Beijing International Center for Mathematical Research, Peking University, Beijing 100871, China}
\email{xhshi@pku.edu.cn}


\subjclass[2020]{Primary: 37F10; Secondary: 37A50, 37A35, 37D35, 37F15, 37B99, 57M12}

\keywords{expanding Thurston map, postcritically-finite map, large deviation principles, entropy density, upper semi-continuity, subsystems, thermodynamic formalism.} 


\begin{abstract}
Expanding Thurston maps were introduced by M.~Bonk and D.~Meyer with motivation from complex dynamics and Cannon's conjecture from geometric group theory via Sullivan's dictionary. 
In this paper, we show that the entropy map of an expanding Thurston map is upper semi-continuous if and only if the map has no periodic critical points. 
For all expanding Thurston maps, even in the presence of periodic critical points, we show that ergodic measures are entropy-dense and establish level-2 large deviation principles for the distribution of Birkhoff averages, periodic points, and iterated preimages.
It follows that iterated preimages and periodic points are equidistributed with respect to the unique equilibrium state for an expanding Thurston map and a potential that is \holder continuous with respect to a visual metric on $S^2$.
In particular, our results answer two questions in \cite{li2015weak}. 

The main technical tools in this paper are called subsystems of expanding Thurston maps, inspired by a translation of the notion of subgroups via Sullivan's dictionary.
\end{abstract}


\maketitle
\tableofcontents


\section{Introduction}
\label{sec:Introduction}



A Thurston map is a (non-homeomorphic) branched covering map on a topological $2$-sphere $S^{2}$ such that each of its critical points has a finite orbit (postcritically-finite).
The most important examples are given by postcritically-finite rational maps on the Riemann sphere $\ccx$.
While Thurston maps are purely topological objects, a deep theorem due to W.~P.~Thurston characterizes Thurston maps that are, in a suitable sense, described in the language of topology and combinatorics, equivalent to postcritically-finite rational maps (see \cite{douady1993proof}). 
This suggests that for the relevant rational maps, an explicit analytic expression is not so important, but rather a geometric-combinatorial description. 
This inspires one to investigate the most essential dynamical and geometric properties of postcritically-finite rational maps in the setting of Thurston maps instead, with the conformality or smoothness assumptions removed.

In the early 1980s, D.~P.~Sullivan introduced a ``dictionary'' that is now known as \emph{Sullivan's dictionary}, which connects two branches of conformal dynamics, iterations of rational maps and actions of Kleinian groups.
Under Sullivan's dictionary, the counterpart to Thurston's theorem in geometric group theory is Cannon's Conjecture \cite{cannon1994combinatorial}. 
An equivalent formulation of Cannon's Conjecture, viewed from a quasisymmetric uniformization perspective (\cite[Conjecture~5.2]{bonk2006quasiconformal}), predicts that if the boundary at infinity $\partial_{\infty} G$ of a Gromov hyperbolic group $G$ is homeomorphic to $S^2$, then $\partial_{\infty} G$ equipped with a visual metric is quasisymmetrically equivalent to $\ccx$. 

Inspired by Sullivan's dictionary and their interest in Cannon's Conjecture, M.~Bonk and D.~Meyer \cite{bonk2010expanding,bonk2017expanding}, as well as P.~Ha{\"i}ssinsky and K.~M.~Pilgrim \cite{haissinsky2009coarse}, studied a subclass of Thurston maps, called \emph{expanding Thurston maps}, by imposing some additional condition of expansion (see Definition~\ref{def:expanding_Thurston_maps}).
In particular, a postcritically-finite rational map on $\ccx$ is expanding if and only if its Julia set is equal to $\ccx$.
For an expanding Thurston map on $S^{2}$, we can equip $S^2$ with a natural class of metrics $d$, called \emph{visual metrics}, that are quasisymmetrically equivalent to each other and are constructed in a similar way as the visual metrics on the boundary $\partial_{\infty} G$ of a Gromov hyperbolic group $G$ (see \cite[Chapter~8]{bonk2017expanding} for details, and see \cite{haissinsky2009coarse} for a related construction). 
In the language above, the following theorem was obtained in \cite{bonk2010expanding,bonk2017expanding,haissinsky2009coarse}, which can be seen as an analog of Cannon's conjecture for expanding Thurston maps.

\begin{theorem*}[M.~Bonk \& D.~Meyer \cite{bonk2010expanding,bonk2017expanding}; P.~Ha{\"i}ssinky \& K.~M.~Pilgrim \cite{haissinsky2009coarse}]
    Let $f \colon S^2 \mapping S^2$ be an expanding Thurston map and $d$ be a visual metric for $f$. 
    Then $f$ is topologically conjugate to a rational map if and only $(S^2, d)$ is quasisymmetrically equivalent to $\ccx$. 
\end{theorem*}

In this paper we study the dynamics and properties of expanding Thurston maps from the point of view of ergodic theory.
Ergodic theory for expanding Thurston maps has been investigated in \cite{li2018equilibrium,li2015weak,li2017ergodic} by the first-named author of the current paper.
In \cite{li2018equilibrium}, for expanding Thurston maps and \holder continuous potentials (with respect to a visual metric), the first-named author of the current paper works out the thermodynamic formalism and investigates the existence, uniqueness, and other properties of equilibrium states, with respect to which iterated preimages are equidistributed.
In \cite{li2015weak}, for expanding Thurston maps without periodic critical points, by proving that the entropy map is upper semi-continuous and applying a general framework devised by Y.~Kifer \cite{kifer1990large}, the first-named author of the current paper establishes level-$2$ large deviation principles for iterated preimages and periodic points with respect to equilibrium states and obtains the corresponding equidistribution results.

However, for expanding Thurston maps with a periodic critical point, upper semi-continuity of the entropy map, level-2 large deviation principles, and equidistribution of periodic points with respect to equilibrium states remained open.

In the present paper, for any expanding Thurston map, even in the presence of periodic critical points, we prove entropy density of ergodic measures, establish level-$2$ large deviation principles for the distribution of Birkhoff averages, periodic points, and iterated preimages, and conclude that periodic points and iterated preimages are equidistributed with respect to the unique equilibrium state for a potential that is \holder continuous with respect to a visual metric on $S^2$.
In particular, we answer the two questions posed in \cite[p.~523]{li2015weak} by the first-named author of the current paper.
More precisely, by constructing suitable subsystems and applying the thermodynamic formalism for subsystem developed in a series of papers \cite{shi2023thermodynamic,shi2024uniqueness}, we show that the entropy map of an expanding Thurston map $f \colon S^2 \mapping S^2$ is not upper semi-continuous when $f$ has at least one periodic critical point. 
This result gives a negative answer to Question~1 posed in \cite[p.~523]{li2015weak}.
Moreover, it suggests that the method used there to prove large deviation principles does not apply to expanding Thurston maps with at least one periodic critical point. 
In order to answer Question~2 posed in \cite[p.~523]{li2015weak} positively, i.e., obtain the equidistribution results even in the presence of periodic critical points, we show that the equilibrium state is the unique minimizer of the rate function and then apply the level-$2$ large deviation principles.

The main technical tools that facilitated the new discoveries in this paper are called \emph{subsystems} of expanding Thurston maps (see Subsection~\ref{sub:Subsystems of expanding Thurston maps}), introduced and investigated in \cite{shi2023thermodynamic,shi2024uniqueness}.
The notion of subsystems is inspired by a translation of the notion of subgroups from geometric group theory via Sullivan's dictionary.
We remark that subsystems are not only useful tools for studying ergodic theory of expanding Thurston maps, but they also have geometric significance in themselves.
Note that under Sullivan's dictionary, an expanding Thurston map corresponds to a Gromov hyperbolic group whose boundary at infinity is $S^2$.
In this sense, a subsystem corresponds to a Gromov hyperbolic group whose boundary at infinity is a subset of $S^{2}$.
In particular, for Gromov hyperbolic groups whose boundary at infinity is a \sierpinski carpet, there is an analog of Cannon's conjecture---the Kapovich--Kleiner conjecture. 
It predicts that these groups arise from some standard situation in hyperbolic geometry.
Similar to Cannon's conjecture, one can reformulate the Kapovich--Kleiner conjecture in an equivalent way as a question related to quasisymmetric uniformization.
For subsystems, it is easy to find examples where the tile maximal invariant set is homeomorphic to the standard \sierpinski carpet (see Subsection~\ref{sub:Subsystems of expanding Thurston maps} for examples of subsystems).
In this case, an analog of the Kapovich--Kleiner conjecture for subsystems is established in \cite{bonk2024dynamical}.

\subsection{Main results}%
\label{sub:Main results}

Our results consist of three parts.
We first show that the entropy map of an expanding Thurston map $f \colon S^2 \mapping S^2$ is upper semi-continuous if and only if $f$ has no periodic critical points.
Then for every expanding Thurston map, we prove that ergodic measures are entropy-dense, i.e., any invariant Borel probability measure can be approximated in the weak$^{*}$-topology by ergodic ones with similar entropy.
Finally, for every expanding Thurston, we establish level-$2$ large deviation principles for the distribution of Birkhoff averages, periodic points, and iterated preimages, and conclude that periodic points and iterated preimages are equidistributed with respect to the unique equilibrium state for a potential that is \holder continuous with respect to a visual metric on $S^2$.

We now state our results more precisely.

\subsubsection*{Upper semi-continuity of entropy}%
\label{ssub:Upper semi-continuity of entropy}

The entropy map of a continuous map $T \colon X \mapping X$ defined on a metric space $(X, d)$ is the map $\mu \mapsto h_{\mu}(T)$ which is defined on the space of $T$-invariant Borel probability measures $\invmea[X][T]$, where $h_{\mu}(T)$ is the measure-theoretic entropy of $T$ for $\mu$ (see Subsection~\ref{sub:thermodynamic formalism} and Definition~\ref{def:entropy map}).

Our first result is about upper semi-continuity of the entropy map for expanding Thurston maps.

\begin{theorem}    \label{thm:upper semi-continuous iff no periodic critical points}
    Let $f \colon S^2 \mapping S^2$ be an expanding Thurston map. 
    Then the entropy map of $f$ is upper semi-continuous if and only if $f$ has no periodic critical points.
\end{theorem}

We remark that for expanding Thurston maps without periodic critical points, the upper semi-continuity of the entropy map has been established in \cite[Corollary~1.3]{li2015weak} by proving a stronger property called asymptotic $h$-expansiveness (see \cite{misiurewicz1976topological}).
In the present paper, we complete the ``only if'' part in Theorem~\ref{thm:upper semi-continuous iff no periodic critical points} through concrete constructions (see Theorem~\ref{thm:not upper semi-continuous with periodic critical points}).
These constructions show that the entropy map is not upper semi-continuous even when restricted to the set of ergodic measures.
Moreover, we estimate the defects in semi-continuity (see Theorem~\ref{thm:not upper semi-continuous with periodic critical points}).

The continuity properties of the entropy map have been studied for a long time.
A classical result is that for an expansive homeomorphism defined on a compact metric space the entropy map is upper semi-continuous (see for example, \cite[Theorem~8.2]{walters1982introduction}).
M.~Lyubich \cite[Corollary~1]{lyubich1983entropy} showed that for rational maps on the Riemann sphere the entropy map is upper semi-continuous.
Another fundamental result is that for a $C^{\infty}$ map defined on a smooth compact manifold the entropy map is upper semi-continuous (see \cite[Theorem~4.1]{newhouse1989continuity} and \cite{yomdin1987volume}).
While for $C^{r}$ diffeomorphisms with finite $r$, upper semi-continuity of the entropy map may fail (for examples in dimension four see \cite{misiurewicz1973diffeomorphism} and for examples in dimension two see \cite{buzzi2014surface}). 
In this setting J.~Buzzi, S.~Crovisier, and O.~Sarig estimated the discontinuities of the entropy map in terms of Lyapunov exponents (see \cite{buzzi2022continuity}).
In the non-compact setting, for transitive countable Markov shift, the entropy map is upper semi-continuous if the shift map has finite topological entropy (see \cite[Theorem~8.1]{iommi2022escape}).
Otherwise the entropy map may not be upper semi-continuous (see \cite[p.~774]{jenkinson2005zero}).

\subsubsection*{Entropy density of ergodic measures}%
\label{ssub:Entropy density of ergodic measures}

Let $T \colon X \mapping X$ be a continuous map on a compact metric space $(X, d)$ and $\invmea[X][T]$ be the set of $T$-invariant Borel probability measures on $X$.
We say that a subset $\mathcal{N} \subseteq \invmea[X][T]$ is \emph{entropy-dense} in $\invmea[X][T]$ if, for each $\mu \in \invmea[X][T]$, there exists a sequence $\{\mu_{n}\}_{n \in \n}$ in $\mathcal{N}$ such that $\{\mu_{n}\}_{n \in \n}$ converges to $\mu$ in the weak$^{*}$-topology and $h_{\mu_{n}}(T) \to h_{\mu}(T)$ as $n \to +\infty$.

We show that the set of ergodic measures is entropy-dense for expanding Thurston maps.

\begin{theorem}    \label{thm:entropy dense}
    For an expanding Thurston map $f \colon S^2 \mapping S^2$, the set of ergodic $f$-invariant measures is entropy-dense in $\invmea$.
\end{theorem}

Entropy density of ergodic measures guarantees that one can approximate non-ergodic measures with ergodic ones with similar entropy and similar expectations.
Such a property plays an important role in large deviation theory, which was used in \cite{follmer1988large}, and has been studied in various settings such as \cite{eizenberg1994large} ($\z^{d}$ subshifts of finite type), \cite{pfister2005large} (uniformly hyperbolic systems and $\beta$-shifts), \cite{yamamoto2009weaker} (ergodic group automorphisms), and \cite{takahasi2019large,takahasi2023level} (countable Markov Shifts). 
In addition, it has applications in the multifractal analysis (see \cite{iommi2015multifractal}).

\subsubsection*{Level-2 large deviation principles}%
\label{ssub:Level-2 large deviation principles}

Let $\{ \xi_{n} \}_{n \in \n}$ be a sequence of Borel probability measures on a topological space $\mathcal{X}$.
We say that $\{ \xi_{n} \}_{n \in \n}$ satisfies a \emph{large deviation principle} in $\mathcal{X}$ if there exists a lower semi-continuous function $I \colon \mathcal{X} \mapping [0, +\infty]$ such that 
\[
    \liminf_{n \to +\infty} \frac{1}{n} \log \xi_{n}(\mathcal{G}) \geqslant - \inf_{\mathcal{G}} I  \quad \text{for all open } \mathcal{G} \subseteq \mathcal{X},
\]
and
\[
    \limsup_{n \to +\infty} \frac{1}{n} \log \xi_{n}(\mathcal{K}) \leqslant - \inf_{\mathcal{K}} I  \quad \text{for all closed } \mathcal{K} \subseteq \mathcal{X},
\]
where $\log 0 = -\infty$ and $\inf \emptyset = +\infty$ by convention.
Such a function $I$ is called a \defn{rate function}, and we call $x \in \mathcal{X}$ a \emph{minimizer} if $I(x) = 0$ holds.
See Subsection~\ref{sub:Level-2 large deviation principles} for background information and further properties.

For expanding Thurston maps, we establish a level-2 large deviation principle for the distribution of Birkhoff averages, periodic points, and iterated preimages.

\begin{theorem}    \label{thm:level-2 large deviation principle}
    Let $f \colon S^2 \mapping S^2$ be an expanding Thurston map and $d$ be a visual metric on $S^2$ for $f$.
    Let $\potential$ be a real-valued \holder continuous function on $S^2$ with respect to the metric $d$. 
    Let $\mu_{\potential}$ be the unique equilibrium state for $f$ and $\potential$.
    Let $\probsphere$ denote the space of Borel probability measures on $S^2$ equipped with the weak$^{*}$-topology.

    For each $n \in \n$, let $V_{n} \colon S^{2} \mapping \probsphere$ be the continuous function defined by 
    \begin{equation}    \label{eq:def:delta measure for orbit}
        \deltameasure{x} \define \frac{1}{n} \sum_{i = 0}^{n - 1} \delta_{f^{i}(x)},
    \end{equation}
    and denote $S_{n}\potential(x) \define \sum_{i = 0}^{n - 1} \potential(f^{i}(x))$ for each $x \in S^{2}$.
    Fix an arbitrary sequence $\sequen{w_n}$ of real-valued functions on $S^2$ satisfying $w_{n}(x) \in \bigl[ 1, \deg_{f^{n}}(x) \bigr]$ for each $n \in \n$ and each $x \in S^{2}$.
    For each $n \in \n$, we consider the following Borel probability measures on $\probsphere$.

    \smallskip

    {\bf Birkhoff averages.} $\birkhoffmeasure \define (V_{n})_{*}(\mu_{\potential})$ (i.e., $\birkhoffmeasure$ is the push-forward of $\mu_{\potential}$ by $V_{n} \colon S^{2} \mapping \probmea{S^{2}}$).

    \smallskip

    {\bf Periodic points.} With $\periodorbit \define \{ p \in S^{2} \describe f^{n}(p) = p \}$, put
    \begin{equation}    \label{eq:def:Periodic points distribution} 
        \Omega_{n} \define \sum_{p \in \periodorbit} \frac{ w_{n}(p) \myexp{ S_{n}\potential(p) } }{ \sum_{p' \in \periodorbit} w_{n}(p') \myexp{ S_{n}\potential(p') } }  \delta_{\deltameasure{p}}.
    \end{equation}

    {\bf Iterated preimages.} Given a sequence $\{ x_{j} \}_{j \in \n}$ of points in $S^{2}$, put
    \begin{equation}    \label{eq:def:Iterated preimages distribution} 
        \Omega_{n}(x_{n}) \define \sum_{y \in f^{-n}(x_{n})} \frac{ w_{n}(y) \myexp{ S_{n}\potential(y) } }{\sum_{y' \in f^{-n}(x_{n})} w_{n}(y') \myexp{ S_{n}\potential(y') } }  \delta_{\deltameasure{y}}.
    \end{equation}
    Then each of the sequences $\{ \birkhoffmeasure \}_{n \in \n}$, $\{ \Omega_{n} \}_{n \in \n}$, and $\{ \Omega_{n}(x_{n}) \}_{n \in \n}$ 
    satisfies a large deviation principle in $\probsphere$ with the rate function $\ratefun \colon \probsphere \mapping [0, +\infty]$ given by
    \begin{equation}    \label{eq:def:rate function}
        \ratefun(\mu) \define - \inf_{\mathcal{G} \ni \mu} \sup_{\mathcal{G}} \freeenergy,
    \end{equation}
    where the infimum is taken over all open sets $\mathcal{G} \subseteq \probsphere$ containing $\mu$, and $\freeenergy \colon \probsphere \mapping [-\infty, 0]$ is defined by
    \begin{equation}    \label{eq:def:free energy} 
        \freeenergy(\mu) \define 
        \begin{cases}
            h_{\mu}(f) + \int\! \phi \,\mathrm{d}\mu - P(f, \potential) & \mbox{if } \mu \in \mathcal{M}(S^2, f); \\
            -\infty & \mbox{if } \mu \in \probsphere \setminus \mathcal{M}(S^2, f).
        \end{cases}
    \end{equation}
    Moreover, $\mu_{\potential}$ is the unique minimizer of the rate function $\ratefun$, and each of the sequences $\{ \birkhoffmeasure \}_{n \in \n}$, $\{ \Omega_{n} \}_{n \in \n}$, and $\{ \Omega_{n}(x_{n}) \}_{n \in \n}$ converges to $\delta_{\mu_{\potential}}$ in the weak$^{*}$ topology.
    Furthermore, for each convex open subset $\mathcal{G}$ of $\probsphere$ containing some invariant measure, we have $\inf_{\mathcal{G}} \ratefun = \inf_{\overline{\mathcal{G}}} \ratefun$,
    \begin{equation}    \label{eq:equalities for rate function}
        \lim_{n \to +\infty} \frac{1}{n} \log \birkhoffmeasure(\mathcal{G})
        = \lim_{n \to +\infty} \frac{1}{n} \log \Omega_{n}(\mathcal{G}) 
        = \lim_{n \to +\infty} \frac{1}{n} \log \Omega_{n}(x_{n})(\mathcal{G})
        = -\inf_{\mathcal{G}} \ratefun,
    \end{equation}
    and \eqref{eq:equalities for rate function} remains true with $\mathcal{G}$ replaced by its closure $\overline{\mathcal{G}}$.
\end{theorem}

\begin{remark}    \label{rem:rate function lower semi-continuous regularization}
    In Theorem~\ref{thm:level-2 large deviation principle}, one can check that $\sup_{\mathcal{G}} \freeenergy = \sup_{\mathcal{G}} (- \ratefun)$ for all open $\mathcal{G} \subseteq \probsphere$, and $\ratefun$ is convex and lower semi-continuous.
    We call $- \ratefun$ the upper semi-continuous regularization of $\freeenergy$.
    Note that if $f$ has at least one periodic critical point, then the rate function $\ratefun$ is not equal to $- \freeenergy$, because the entropy map of $f$ is not upper semi-continuous on $\mathcal{M}(S^{2}, f)$ (see Theorem~\ref{thm:upper semi-continuous iff no periodic critical points}).
    Indeed, it follows from Theorem~\ref{thm:upper semi-continuous iff no periodic critical points} that $\ratefun = -\freeenergy$ if and only if $f$ has no periodic critical points.
\end{remark}

The following two corollaries are immediate consequences of Theorem~\ref{thm:level-2 large deviation principle}.
See Subsection~\ref{sub:Proof of large deviation principles} for the proof.

\begin{corollary} \label{coro:level-1 large deviation principle}
	Let $\psi \colon S^2 \mapping \real$ be a continuous function, and let $\widehat{\psi} \colon \probsphere \mapping \real$ be defined by $\widehat{\psi}(\mu) \define \int \! \psi \,\mathrm{d}\mu$.
	With the assumptions and notations of Theorem~\ref{thm:level-2 large deviation principle}, each of the sequences $\sequen[\big]{\widehat{\psi}_{*}(\birkhoffmeasure)}$, $\sequen[\big]{\widehat{\psi}_{*}(\Omega_{n})}$, and $\sequen[\big]{\widehat{\psi}_{*}(\Omega_{n}(x_{n})}$ satisfies a large deviation principle in $\real$ with the rate function
	\begin{equation}    \label{eq:coro:level-1 large deviation principle:rate function}
		x \in \real \mapsto \inf \set[\bigg]{ \ratefun(\mu) \describe \mu \in \probsphere, \, \int \! \psi \,\mathrm{d}\mu = x }.
	\end{equation}
	Furthermore, if $c_{\psi} < d_{\psi}$, where $c_{\psi} \define \min\set[\big]{\int \! \psi \,\mathrm{d}\nu \describe \nu \in \invmea }$ and $d_{\psi} \define \max\set[\big]{\int \! \psi \,\mathrm{d}\nu \describe \nu \in \invmea }$, then for each interval $J \subseteq \real$ intersecting $(c_{\psi}, d_{\psi})$,
	\begin{equation}    \label{eq:coro:level-1 large deviation principle:on closed interval}
		\begin{split}
			& - \inf \set[\bigg]{ \ratefun(\mu) \describe \mu \in \probsphere, \, \int \! \psi \,\mathrm{d}\mu \in J }  \\
			&\qquad = \lim_{n \to +\infty} \frac{1}{n} \log \mu_{\potential} \parentheses[\bigg]{ \set[\bigg]{ x \in S^{2} \describe \frac{1}{n}S_n\psi(x) \in J } }  \\
			&\qquad = \lim_{n \to +\infty} \frac{1}{n} \log \parentheses[\Bigg]{ \frac{ \sum_{p \in \periodorbit, \, \frac{1}{n}S_n\psi(p) \in J }  w_{n}(p) \myexp{ S_{n}\potential(p) } }{ \sum_{p' \in \periodorbit} w_{n}(p') \myexp{ S_{n}\potential(p') } } }  \\
			&\qquad = \lim_{n \to +\infty} \frac{1}{n} \log \parentheses[\Bigg]{ \frac{ \sum_{y \in f^{-n}(x_{n}), \, \frac{1}{n} S_n\psi(y) \in J }  w_{n}(y) \myexp{ S_{n}\potential(y) } }{ \sum_{y' \in f^{-n}(x_{n})} w_{n}(y') \myexp{ S_{n}\potential(y') } } }.
		\end{split}
	\end{equation}
\end{corollary}

\begin{corollary} \label{coro:measure-theoretic pressure infimum on local basis}
    With the assumptions and notations of Theorem~\ref{thm:level-2 large deviation principle}, for each $\mu \in \invmea$ and each convex local basis $G_{\mu}$ of $\probsphere$ at $\mu$, we have
    \begin{equation}    \label{eq:coro:measure-theoretic pressure infimum on local basis}
        \begin{split}
            - \ratefun(\mu)
            &= \inf_{\mathcal{G} \in G_{\mu}} \biggl\{ \lim_{n \to +\infty} \frac{1}{n} \log \mu_{\potential}(\set{x \in S^{2} \describe \deltameasure[n]{x} \in \mathcal{G}}) \biggr\}   \\
            &= \inf_{\mathcal{G} \in G_{\mu}} \biggl\{ \lim_{n \to +\infty} \frac{1}{n} \log \sum_{ p \in \periodorbit, \deltameasure[n]{p} \in \mathcal{G} } w_{n}(p) \myexp{S_{n}\potential(p)} \biggr\} -P(f, \potential)   \\
            &= \inf_{\mathcal{G} \in G_{\mu}} \biggl\{ \lim_{n \to +\infty} \frac{1}{n} \log \sum_{ y \in f^{-n}(x_{n}), \deltameasure[n]{y} \in \mathcal{G} } w_{n}(y) \myexp{S_{n}\potential(y)} \biggr\} -P(f, \potential).
        \end{split}
    \end{equation}
\end{corollary}

The following equidistribution results follow from the corresponding level-$2$ large large deviation principles and the uniqueness of the minimizer of the rate function.

\begin{theorem} \label{thm:equidistribution results}
    Let $f \colon S^2 \mapping S^2$ be an expanding Thurston map and $d$ be a visual metric on $S^2$ for $f$.
    Let $\potential$ be a real-valued \holder continuous function on $S^2$ with respect to the metric $d$.
    Let $\mu_{\potential}$ be the unique equilibrium state for $f$ and $\potential$.
    Fix an arbitrary sequence $\sequen{w_n}$ of real-valued functions on $S^2$ satisfying $w_{n}(x) \in \bigl[ 1, \deg_{f^{n}}(x) \bigr]$ for each $n \in \n$ and each $x \in S^{2}$.
    For each $n \in \n$, denote $S_{n}\potential(x) \define \sum_{i = 0}^{n - 1} \potential(f^{i}(x))$ for each $x \in S^{2}$, and consider the following Borel probability measures on $S^2$.

    \smallskip

    {\bf Periodic points.} With $\periodorbit \define \{ p \in S^{2} \describe f^{n}(p) = p \}$, put\[
        \mu_{n} \define \sum_{p \in \periodorbit} \frac{ w_{n}(p) \myexp{ S_{n}\potential(p) } }{\sum_{p' \in \periodorbit} w_{n}(p') \myexp{ S_{n}\potential(p') } } \frac{1}{n} \sum_{i = 0}^{n - 1} \delta_{f^{i}(p)}.
    \]

    {\bf Iterated preimages.} Given a sequence $\{ x_{j} \}_{j \in \n}$ of points in $S^{2}$, put\[
        \nu_{n} \define \sum_{y \in f^{-n}(x_{n})} \frac{ w_{n}(y) \myexp{ S_{n}\potential(y) } }{\sum_{z \in f^{-n}(x_{n})} w_{n}(z) \myexp{ S_{n}\potential(z) } } \frac{1}{n} \sum_{i = 0}^{n - 1} \delta_{f^{i}(y)}.
    \]
    Then each of the sequences $\sequen{\mu_{n}}$ and $\sequen{\nu_{n}}$ converges to $\mu_{\potential}$ in the weak$^{*}$-topology.
\end{theorem}

See Subsection~\ref{sub:Equidistribution with respect to the equilibrium state} for the proof of Theorem~\ref{thm:equidistribution results}.

\begin{remark}\label{rem:simple form for periodic points}
    Since $S_{n}\potential(f^{i}(p)) = S_{n}\potential(p)$ for each $i \in \n$ if $p \in \periodorbit$, we get\[
        \mu_{n} = \sum_{p \in \periodorbit} \frac{ \frac{S_{n}w_{n}(p)}{n}  \myexp{ S_{n}\potential(p) } }{\sum_{p' \in \periodorbit} w_{n}(p') \myexp{ S_{n}\potential(p') } }  \delta_{p}
    \]
    for each $n \in \n$.
    In particular, when $w_{n}(\cdot) \equiv 1$, we have
    \[
        \mu_{n} = \sum_{p \in \periodorbit} \frac{\myexp{S_{n}\potential(p)}}{\sum_{p' \in \periodorbit} \myexp{S_{n}\potential(p')}} \delta_{p};
    \]
    when $w_{n}(\cdot) \equiv \deg_{f^{n}}(\cdot)$, since $\deg_{f^{n}}(f^{i}(p)) = \deg_{f^{n}}(p)$ for each $i \in \n$ if $p \in \periodorbit$, we have
    \[
        \mu_{n} = \sum_{p \in \periodorbit} \frac{ \deg_{f^{n}}(p) \myexp{ S_{n}\potential(p) } }{\sum_{p' \in \periodorbit} \deg_{f^{n}}(p') \myexp{ S_{n}\potential(p') } }  \delta_{p}.
    \]
\end{remark}

We remark that the novelty of Theorem~\ref{thm:level-2 large deviation principle}, Corollary~\ref{coro:measure-theoretic pressure infimum on local basis}, and Theorem~\ref{thm:equidistribution results}, which generalize the corresponding results in \cite{li2015weak}, lies in their application to \emph{all} expanding Thurston maps, including those with periodic critical points.

\subsection{Strategy and organization}%
\label{sub:Strategy and organization}

We now discuss the strategies of proofs of our main results and describe the organization of the paper.

To prove Theorem~\ref{thm:upper semi-continuous iff no periodic critical points}, by \cite[Corollary~1.3]{li2015weak}, it suffices to study expanding Thurston maps with periodic critical points and show that their entropy maps are not upper semi-continuous.
The main point here is to find a sequence of invariant measures that converges in the weak$^{*}$-topology and has an entropy drop at the limit.
Our strategy is to construct a suitable sequence of subsystems and then apply the main results in \cite{shi2023thermodynamic,shi2024uniqueness} (see for example, Theorem~\ref{thm:existence uniqueness and properties of equilibrium state}) to verify that the sequence of measures of maximal entropy associated with the sequence of subsystems satisfies our desired properties.
The construction of such sequence of subsystems is based on the key observation that the local degree at a periodic critical point increases exponentially under iteration.
Hence we can find subsystems around periodic critical points such that entropies of the associated measures of maximal entropy have a uniform positive lower bound.

The main idea behind the proof of Theorem~\ref{thm:entropy dense} is to find an ergodic measure that is close to a given invariant measure in terms of both topology and entropy.
To achieve this, we construct a suitable subsystem and use the corresponding measure of maximal entropy to approximate the given invariant measure.
The construction of the subsystem poses the main difficulty.
Our strategy involves using pairs (as described in Subsection~\ref{sub:Thurston_maps}) instead of tiles to build a strongly primitive subsystem.
While the set of pairs is not a generator (unlike the set of tiles, as shown in Lemma~\ref{lem:generator}), we first approximate ergodic measures with a finite collection of tiles in a specific sense (Lemma~\ref{lem:approximates ergodic measures by tiles}). 
We then construct pairs from these tiles. 
Furthermore, to obtain a primitive subsystem, we add one suitable pair contained in the interior of the corresponding $0$-tile for each color. 
The existence of such pairs is guaranteed by Lemma~\ref{lem:pair in the interior}.
Finally, with these preparations we are able to prove the entropy density of ergodic measures.

The proof of Theorem~\ref{thm:level-2 large deviation principle} is more involved and we divide the proof into four parts: the uniqueness of the minimizer, characterizations of topological pressures, the large deviation lower bound, and the large deviation upper bound.
Here the characterizations of topological pressures are used in the proof of large deviation lower and upper bounds, and the uniqueness of the minimizer is used to prove the convergence of distributions and equidistribution results, which is necessary since in our setting the entropy map may not be upper semi-continuous.
It should be noted that one cannot derive the uniqueness of the minimizer directly from the uniqueness of the equilibrium states (recall Remark~\ref{rem:rate function lower semi-continuous regularization}).

To prove the uniqueness of the minimizer, we use a semiconjugacy of an expanding Thurston map to a shift map, and show that the uniqueness of the minimizer follows from ergodic properties of the shift map and the uniqueness of the equilibrium state. 
Here a key property of such a semiconjugacy is that even in the presence of periodic critical points, the entropy at the equilibrium state does not drop under the projection from the symbolic space (see Proposition~\ref{prop:semiconjugacy with full shift}~\ref{item:prop:semiconjugacy with full shift:no entropy drop at equilibrium state}).
The existence and properties of such a semiconjugacy are proved in \cite{das2021thermodynamic}.

The characterization of topological pressures in terms of periodic points for expanding Thurston maps has been established in \cite[Propositions~6.8]{li2015weak} (see Theorem~\ref{prop:characterization of pressure weighted periodic points}), while for iterated preimages such characterization was only obtained for expanding Thurston maps without periodic critical points (compare Theorem~\ref{prop:characterization of pressure iterated preimages count by degree} with Theorem~\ref{prop:characterization of pressure weighted periodic points}).
By carefully analyzing the combinatorics of critical points, we establish the characterization of topological pressures in terms of iterated preimages for all expanding Thurston maps (see Theorem~\ref{prop:characterization of pressure iterated preimages arbitrary weight}).

In the proof of large deviation lower and upper bounds we take an indirect approach.
We first consider an expanding Thurston map that has an invariant Jordan curve containing the postcritical set.
In such situations the associated cell decompositions have nice compatibility properties, enabling the application of the results in ergodic theory of subsystems established in \cite{shi2023thermodynamic,shi2024uniqueness}.
However, such an invariant Jordan curve may not exist (see for example, \cite[Example~15.11]{bonk2017expanding}). 
Our strategy is to first establish large deviation bounds for sufficiently high iterates of an expanding Thurston map, where such an invariant Jordan curve does exist (see Lemma~\ref{lem:invariant_Jordan_curve}), and then prove for the original expanding Thurston map.

An important property for the proof of the large deviation lower bound for all open sets is \emph{entropy approachability of ergodic measures} (as defined in Definition~\ref{def:weak entropy density}), which guarantees that any invariant measure can be approximated by an ergodic measure with entropy sufficiently large.
This property and is weaker than the entropy density of ergodic measures and allows for the simplification of the proof of the lower bound to the case where the measure being considered is ergodic \cite{eizenberg1994large,follmer1988large,young1990large,takahasi2019large}.
Then we obtain estimates for ergodic measures by using the approximations (Lemma~\ref{lem:approximates ergodic measures by tiles}) in the proof of the entropy density established in Section~\ref{sec:Entropy density}.

The main idea for the large deviation upper bound is to construct measures whose measure-theoretic pressures provide desired upper bounds for the deviations.
Our strategy is to construct certain strongly primitive subsystems (see Definition~\ref{def:primitivity of subsystem}) and apply the results in \cite{shi2023thermodynamic,shi2024uniqueness} to produce such measures (see Proposition~\ref{prop:construction of suitable invariant measures for vector continuous function}).
Similar strategies are used in one-dimensional real dynamics \cite{chung2019large} and countable Markov shift \cite{takahasi2019large}. 
Constructions and estimates for the upper bound are much harder than those for the lower bound and are necessarily involved due to the presence of critical points and the lack of upper semi-continuity of entropy.

Finally, to prove Theorem~\ref{thm:equidistribution results}, we use the property that the equilibrium state is the unique minimizer of the rate function and then apply the results from large deviation principles.

\smallskip

We now give a brief description of the structure of this paper.

In Section~\ref{sec:Notation}, we fix some notation that will be used throughout the paper.
In Section~\ref{sec:Preliminaries}, we first review some notions from ergodic theory and dynamical systems and go over some key concepts and results on Thurston maps. Then we summarize some concepts and results on subsystems of expanding Thurston maps.
In Section~\ref{sec:The Assumptions}, we state the assumptions on some of the objects in this paper, which we will repeatedly refer to later as the \emph{Assumptions in Section~\ref{sec:The Assumptions}}. 

In Section~\ref{sec:Upper semi-continuity}, we investigate the upper semi-continuity of the entropy map of an expanding Thurston map and prove Theorem~\ref{thm:upper semi-continuous iff no periodic critical points}.

In Section~\ref{sec:Entropy density}, we show that ergodic measures are entropy-dense for expanding Thurston maps and prove Theorem~\ref{thm:entropy dense}.

Section~\ref{sec:Large deviation principles} is devoted to the study of large deviation principles and equidistribution results for Birkhoff averages, periodic points, and iterated preimages of expanding Thurston maps. 
In Subsection~\ref{sub:Level-2 large deviation principles}, we give a brief review of level-2 large deviation principles.
In Subsection~\ref{sub:Uniqueness of the minimizer}, we show that the equilibrium state is the unique minimizer of the rate function.
In Subsection~\ref{sub:Characterizations of topological pressures}, we establish characterizations of the topological pressure in terms of periodic points and iterated preimages. 
In Subsection~\ref{sub:Large deviation lower bound}, we prove the large deviations lower bound for all open sets.
In Subsection~\ref{sub:Large deviation upper bound}, we prove the large deviations upper bound for all closed sets.
In Subsection~\ref{sub:Proof of large deviation principles}, by combining the previous results together, we finish the proof of Theorem~\ref{thm:level-2 large deviation principle} and its corollaries.
In Subsection~\ref{sub:Equidistribution with respect to the equilibrium state}, we show that periodic points and iterated preimages are equidistributed with respect to the unique equilibrium state for an expanding Thurston map and a \holder continuous function, which proves Theorem~\ref{thm:equidistribution results}.

\smallskip

\textbf{Acknowledgments.}
The first-named author wants to thank Peter Ha{\"i}ssinky and Juan Rivera-Letelier for discussions on the upper semi-continuity property of the entropy map.


\section{Notation}
\label{sec:Notation}
Let $\cx$ be the complex plane and $\ccx$ be the Riemann sphere. 
Let $S^2$ denote an oriented topological $2$-sphere.
We use $\n$ to denote the set of integers greater than or equal to $1$ and write $\n_0 \define \{0\} \cup \n$. 
The symbol log denotes the logarithm to the base $e$. 
For $x \in \real$, we define $\lfloor x \rfloor$ as the greatest integer $\leqslant x$, and $\lceil x \rceil$ the smallest integer $\geqslant x$.
The cardinality of a set $A$ is denoted by $\card{A}$.


Let $g \colon X \mapping Y$ be a map between two sets $X$ and $Y$. We denote the restriction of $g$ to a subset $Z$ of $X$ by $g|_{Z}$.

Consider a map $f \colon X \mapping X$ on a set $X$. 
The inverse map of $f$ is denoted by $f^{-1}$. 
We write $f^n$ for the $n$-th iterate of $f$, and $f^{-n}\define (f^n)^{-1}$, for each $n \in \n$. 
We set $f^0 \define \id{X}$, the identity map on $X$. 
For a real-valued function $\varphi \colon X \mapping \real$, we write
\begin{equation}    \label{eq:def:Birkhoff average}
    S_n \varphi(x) = S^f_n \varphi(x) \define \sum_{j=0}^{n-1} \varphi \bigl( f^j(x) \bigr)
\end{equation}
for each $x\in X$ and each $n\in \n_0$. 
We omit the superscript $f$ when the map $f$ is clear from the context. Note that when $n = 0$, by definition we always have $S_0 \varphi = 0$.

Let $(X,d)$ be a metric space. For each subset $Y \subseteq X$, we denote the diameter of $Y$ by $\diam{d}{Y} \define \sup\{d(x, y) \describe \juxtapose{x}{y} \in Y\}$, the interior of $Y$ by $\interior{Y}$, and the characteristic function of $Y$ by $\indicator{Y}$, which maps each $x \in Y$ to $1 \in \real$ and vanishes otherwise. 
For each $r > 0$ and each $x \in X$, we denote the open (\resp closed) ball of radius $r$ centered at $x$ by $B_{d}(x,r)$ (\resp $\overline{B_{d}}(x,r)$). We often omit the metric $d$ in the subscript when it is clear from the context.

For a compact metrizable topological space $X$, we denote by $C(X)$ (\resp $B(X)$) the space of continuous (\resp bounded Borel) functions from $X$ to $\real$, by $\mathcal{M}(X)$ the set of finite signed Borel measures, and $\mathcal{P}(X)$ the set of Borel probability measures on $X$.
By the Riesz representation theorem (see for example, \cite[Theorems~7.17 and 7.8]{folland2013real}), we identify the dual of $C(X)$ with the space $\mathcal{M}(X)$.
For $\mu \in \mathcal{M}(X)$, we use $\norm{\mu}$ to denote the total variation norm of $\mu$, $\supp{\mu}$ the support of $\mu$ (the smallest closed set $A \subseteq X$ such that $|\mu|(X \setminus A) = 0$), and \[
    \functional{\mu}{u} \define \int \! u \,\mathrm{d}\mu
\]
for each $u \in C(X)$. 
For a point $x \in X$, we define $\delta_x$ as the Dirac measure supported on $\{x\}$.
For a continuous map $g \colon X \mapping X$, we set $\mathcal{M}(X, g)$ to be the set of $g$-invariant Borel probability measures on $X$.
If we do not specify otherwise, we equip $C(X)$ with the uniform norm $\normcontinuous{\cdot}{X} \define \uniformnorm{\cdot}$, and equip $\mathcal{M}(X)$, $\mathcal{P}(X)$, and $\mathcal{M}(X, g)$ with the weak$^*$ topology.

The space of real-valued \holder continuous functions with an exponent $\holderexp \in (0,1]$ on a compact metric space $(X, d)$ is denoted as $\holderspace$. 
For each $\phi \in \holderspace$, \[
    \holderseminorm{\phi}{X} \define \sup\left\{ \frac{|\phi(x) - \phi(y)|}{d(x, y)^{\holderexp}} \describe \juxtapose{x}{y} \in X, \, x \ne y \right\},
\]
and the \holder norm is defined as $\holdernorm{\phi}{X} \define \holderseminorm{\phi}{X} + \normcontinuous{\phi}{X}$. 

\section{Preliminaries}
\label{sec:Preliminaries}

\subsection{Thermodynamic formalism}  
\label{sub:thermodynamic formalism}

We first review some basic concepts from ergodic theory and dynamical systems. 
For more detailed studies of these concepts, we refer the reader to \cite[Chapter~3]{przytycki2010conformal}, \cite[Chapter~9]{walters1982introduction}, or \cite[Chapter~20]{katok1995introduction}.

Let $(X,d)$ be a compact metric space and $g \colon X \mapping X$ a continuous map. Given $n \in \n$, 
\[
    d^n_g(x, y) \define \operatorname{max} \bigl\{  d \bigl(g^k(x), g^k(y) \bigr) \describe k \in \zeroton[n - 1] \!\bigr\}, \quad \text{ for } \juxtapose{x}{y} \in X,
\]
defines a metric on $X$. 
A set $F \subseteq X$ is \emph{$(n, \epsilon)$-separated} (with respect to $g$), for some $n \in \n$ and $\epsilon > 0$, if for each pair of distinct points $\juxtapose{x}{y} \in F$, we have $d^n_g(x, y) \geqslant \epsilon$. 
Given $\epsilon > 0$ and $n \in \n$, let $F_n(\epsilon)$ be a maximal (in the sense of inclusion) $(n,\epsilon)$-separated set in $X$.

For each real-valued continuous function $\psi \in C(X)$, the following limits exist and are equal, and we denote these limits by $P(g, \psi)$ (see for example, \cite[Theorem~3.3.2]{przytycki2010conformal}):

\begin{equation}  \label{eq:def:topological pressure}
    \begin{split}
        P(g, \psi) \define  \lim \limits_{\epsilon \to 0^{+}} \limsup\limits_{n \to +\infty} \frac{1}{n} \log \sum\limits_{x \in F_n(\epsilon)} \myexp{S_n \psi(x)} 
        = \lim \limits_{\epsilon \to 0^{+}} \liminf\limits_{n \to +\infty} \frac{1}{n} \log  \sum\limits_{x\in F_n(\epsilon)} \myexp{S_n \psi(x)}, 
    \end{split}
\end{equation}
where $S_n \psi (x) = \sum_{j=0}^{n-1} \psi ( g^j(x) )$ is defined in \eqref{eq:def:Birkhoff average}. We call $P(g, \psi)$ the \emph{topological pressure} of $g$ with respect to the \emph{potential} $\psi$. 
Note that $P(g, \psi)$ is independent of $d$ as long as the topology on $X$ defined by $d$ remains the same (see for example, \cite[Section~3.2]{przytycki2010conformal}).
The quantity $h_{\operatorname{top}}(g) \define P(g, 0)$ is called the \emph{topological entropy} of $g$. The topological entropy is well-behaved under iterations. Indeed, if $n \in \n$, then $h_{\operatorname{top}}(g^{n}) = n h_{\operatorname{top}}(g)$ (see for example, \cite[Proposition~3.1.7~(3)]{katok1995introduction}).

We denote by $\mathcal{B}$ the $\sigma$-algebra of all Borel sets on $X$. 
A measure on $X$ is understood to be a Borel measure, i.e., one defined on $\mathcal{B}$. We call a measure $\mu$ on $X$ $g$-invariant if \[
    \mu\bigl(g^{-1}(A) \bigr) = \mu(A)
\]
for all $A \in \mathcal{B}$. 
We denote by $\mathcal{M}(X, g)$ the set of all $g$-invariant Borel probability measures on $X$.


Let $\mu \in \mathcal{M}(X, g)$. 
Then we say that $g$ is \emph{ergodic} for $\mu$ (or $\mu$ is \emph{ergodic} for $g$) if for each set $A \in \mathcal{B}$ with $g^{-1}(A) = A$ we have $\mu(A) = 0$ or $\mu(A) = 1$. 
We denote by $\ergmea[X][g]$ the set of all $g$-invariant ergodic measures on $X$.


Let $\mu \in \mathcal{M}(X, g)$. A \emph{measurable partition} $\xi$ for $(X, \mu)$ is a countable collection $\xi = \{ A_{i} \describe i \in I \}$ of sets in $\mathcal{B}$ such that $\mu(A_{i} \cap A_{j}) = 0$ for $i, j \in I$, $i \ne j$, and \[
    \mu\biggl( X \setminus \bigcup_{i \in I} A_i \biggr) = 0.
\]
Here $I$ is a countable (i.e., finite or countably infinite) index set. The measurable partition $\xi$ is finite if the index set $J$ is a finite set.
The \emph{symmetric difference} of two sets $\juxtapose{A}{B} \subseteq X$ is defined as\[
    A \bigtriangleup B \define (A \setminus B) \cup (B \setminus A).
\]
Two measurable partitions $\xi$ and $\eta$ for $(X, \mu)$ are called \emph{equivalent} if there exists a bijection between the sets of positive measure in $\xi$ and the sets of positive measure in $\eta$ such that corresponding sets have a symmetric difference of $\mu$-measure zero. Roughly speaking, this means that the partitions are the same up to sets of measure zero.


Let $\xi = \{A_j \describe j \in J\}$ and $\eta = \{B_k \describe k \in K\}$ be measurable partitions for $(X, \mu)$.
We say $\xi$ is a \emph{refinement} of $\eta$ if for each $A_j \in \xi$, there exists $B_k \in \eta$ such that $A_j \subseteq B_k$. 
The \emph{common refinement} (or \emph{join}) $\xi \vee \eta$ of $\xi$ and $\eta$ defined as\[
    \xi \vee \eta \define \{A_j \cap B_k \describe j \in J, \, k \in K\}
\]
is also a measurable partition. 
Put \[
    g^{-1}(\xi) \define \bigl\{ g^{-1}(A_j) \describe j \in J \bigr\},
\]
and for each $n \in \n$ define \[
    \xi^n_g \define \bigvee \limits_{j = 0}^{n - 1} g^{-j}(\xi) = \xi \vee g^{-1}(\xi) \vee \cdots \vee g^{-(n-1)}(\xi).
\]


Let $\xi$ be a finite measurable partition for $(g, \mu)$ and $\mathcal{A}$ be the smallest $\sigma$-algebra containing all sets in the partitions $\xi^{n}_{g}$, $n \in \n$.
We call $\xi$ a \emph{generator} for $(g, \mu)$ if for each Borel set $B \in \mathcal{B}$ there exists a set $A \in \mathcal{A}$ such that $\mu(A \bigtriangleup B) = 0$.

If for every set $B \in \mathcal{B}$ and every $\varepsilon > 0$, there exists $n \in \n$ and a union $A$ of sets in $\xi^{n}_{g}$ with $\mu(A \bigtriangleup B) < \varepsilon$, then $\xi$ is a generator for $(g, \mu)$.

Let $\xi = \{A_j \describe j \in J \}$ be a measurable partition of $X$ and $\mu \in \mathcal{M}(X, g)$ be a $g$-invariant Borel probability measure on $X$. 
The \emph{entropy} of $\xi$ is $H_{\mu}(\xi) \define -\sum_{j\in J} \mu(A_j) \log\left(\mu (A_j)\right) \in [0, +\infty]$, where $0 \log 0$ is defined to be zero. 
One can show that (see for example, \cite[Chapter~4]{walters1982introduction}) if $H_{\mu}(\xi) < +\infty$, then the following limit exists:
\begin{equation}    \label{eq:def:measure-theoretic entropy with respect to partition}
    h_{\mu}(g, \xi) \define \lim\limits_{n \to +\infty} \frac{1}{n} H_{\mu}(\xi^n_g) \in [0, +\infty).
\end{equation}
The quantity $h_{\mu}(g, \xi)$ is called the \emph{measure-theoretic entropy of $g$ relative to $\xi$}.
The \emph{measure-theoretic entropy} of $g$ for $\mu$ is defined as
\begin{equation}   \label{eq:def:measure-theoretic entropy}
h_{\mu}(g) \define \sup\{h_{\mu}(g,\xi) \describe \xi \text{ is a measurable partition of } X  \text{ with } H_{\mu}(\xi) < +\infty\}.   
\end{equation}
If $\mu \in \mathcal{M}(X, g)$ and $n \in \n$, then (see for example, \cite[Proposition~4.3.16~(4)]{katok1995introduction})
\begin{equation}    \label{eq:measure-theoretic entropy well-behaved under iteration}
    h_{\mu}(g^{n}) = n h_{\mu}(g).
\end{equation}
If $t \in [0, 1]$ and $\nu \in \mathcal{M}(X, g)$ is another measure, then (see for example, \cite[Theorem~8.1]{walters1982introduction})
\begin{equation}    \label{eq:measure-theoretic entropy is affine}
    h_{t \mu + (1 - t)\nu}(g) = t h_{\mu}(g) + (1 - t) h_{\nu}(g).
\end{equation}

For each real-valued continuous function $\psi \in C(X)$, the \emph{measure-theoretic pressure} $P_\mu(g, \psi)$ of $g$ for the measure $\mu \in \mathcal{M}(X, g)$ and the potential $\psi$ is
\begin{equation}  \label{eq:def:measure-theoretic pressure}
P_\mu(g, \psi) \define  h_\mu (g) + \int \! \psi \,\mathrm{d}\mu.
\end{equation}

The topological pressure is related to the measure-theoretic pressure by the so-called \emph{Variational Principle}.
It states that (see for example, \cite[Theorem~3.4.1]{przytycki2010conformal})
\begin{equation}  \label{eq:Variational Principle for pressure}
    P(g, \psi) = \sup \{P_\mu(g, \psi) \describe \mu \in \mathcal{M}(X, g)\}
\end{equation}
for each $\psi \in C(X)$.
In particular, when $\psi$ is the constant function $0$,
\begin{equation}  \label{eq:Variational Principle for entropy}
    h_{\operatorname{top}}(g) = \sup\{h_{\mu}(g) \describe\mu \in \mathcal{M}(X, g)\}.
\end{equation}
A measure $\mu$ that attains the supremum in \eqref{eq:Variational Principle for pressure} is called an \emph{equilibrium state} for the map $g$ and the potential $\psi$. A measure $\mu$ that attains the supremum in \eqref{eq:Variational Principle for entropy} is called a \emph{measure of maximal entropy} of $g$.




Let $\widetilde{X}$ be another compact metric space. If $\mu$ is a measure on $X$ and the map $\pi \colon X \mapping \widetilde{X}$ is continuous, then the \emph{push-forward} $\pi_{*} \mu$ of $\mu$ by $\pi$ is the measure given by $\pi_{*}\mu(A) \define \mu\parentheses[\big]{ \pi^{-1}(A) } $ for each Borel set $A \subseteq \widetilde{X}$. 
Note that if $\widetilde{X} = X$, then $\mu$ is $\pi$-invariant if and only if $\pi_{*}\mu = \mu$.

Suppose $\widetilde{g} \colon \widetilde{X} \mapping \widetilde{X}$ is a continuous map, $\mu \in \mathcal{M}(X, g)$, and $\widetilde{\mu} \in \mathcal{M}(\widetilde{X}, \widetilde{g})$. Then the dynamical system $(\widetilde{X}, \widetilde{g}, \widetilde{\mu})$ is called a \emph{factor} of $(X, g, \mu)$ if there exists a continuous and surjective map $\pi \colon X \mapping \widetilde{X}$ such that $\pi_{*}\mu = \widetilde{\mu}$ and $\widetilde{g} \circ \pi = \pi \circ g$. 
In this case, $h_{\widetilde{\mu}}(\widetilde{g}) \leqslant  h_{\mu}(g)$ (see for example, \cite[Proposition~4.3.16~(1)]{katok1995introduction}).

\subsection{Thurston maps}%
\label{sub:Thurston_maps}
In this subsection, we go over some key concepts and results on Thurston maps, and expanding Thurston maps in particular. For a more thorough treatment of the subject, we refer to \cite{bonk2017expanding}.

Let $S^2$ denote an oriented topological $2$-sphere. A continuous map $f \colon S^2 \mapping S^2$ is called a \emph{branched covering map} on $S^2$ if for each point $x\in S^2$, there exists a positive integer $d\in \n$, open neighborhoods $U$ of $x$ and $V$ of $y \define f(x)$, open neighborhoods $U'$ and $V'$ of $0$ in $\ccx$, and orientation-preserving homeomorphisms $\varphi \colon U \mapping U'$ and $\eta \colon V \mapping V'$ such that $\varphi(x) = 0, \eta(y) = 0$, and $\bigl(\eta \circ f\circ \varphi^{-1}\bigr)(z) = z^d$ for each $z \in U'$. The positive integer $d$ above is called the \emph{local degree} of $f$ at $x$ and is denoted by $\deg_f(x)$ or $\deg(f, x)$.

The \emph{degree} of $f$ is $\deg{f} = \sum_{x\in f^{-1}(y)} \deg_{f}(x)$ for $y\in S^2$ and is independent of $y$. If $f \colon S^2 \mapping S^2$ and $g \colon S^2 \mapping S^2$ are two branched covering maps on $S^2$, then so is $f\circ g$, and $\deg(f\circ g, x) = \deg(g, x) \deg(f, g(x))$ for each $x\in S^2$, and moreover, $\deg(f\circ g) = (\deg{f})(\deg{g})$.

A point  $x\in S^2$ is a \emph{critical point} of $f$ if $\deg_f(x) \geqslant 2$. The set of critical points of $f$ is denoted by $\crit{f}$. A point $y\in S^2$ is a \emph{postcritical point} of $f$ if $y = f^n(x)$ for some $x \in \crit{f}$ and $n\in \n$. The set of postcritical points of $f$ is denoted by $\post{f}$. Note that $\post{f} = \post{f^n}$ for all $n\in \n$.

\begin{definition}[Thurston maps]
    A Thurston map is a branched covering map $f \colon S^2 \mapping S^2$ on $S^2$ with $\deg f \geqslant 2$ and $\card{\post{f}}< +\infty$.
\end{definition}

We now recall the notation for cell decompositions of $S^2$ used in \cite{bonk2017expanding} and \cite{li2017ergodic}. A \emph{cell of dimension $n$} in $S^2$, $n \in \{1, \, 2\}$, is a subset $c \subseteq S^2$ that is homeomorphic to the closed unit ball $\overline{\mathbb{B}^n}$ in $\real^n$, where $\mathbb{B}^{n}$ is the open unit ball in $\real^{n}$. We define the \emph{boundary of $c$}, denoted by $\partial c$, to be the set of points corresponding to $\partial \mathbb{B}^n$ under such a homeomorphism between $c$ and $\overline{\mathbb{B}^n}$. The \emph{interior of $c$} is defined to be $\inte{c} = c \setminus \partial c$. For each point $x\in S^2$, the set $\{x\}$ is considered as a \emph{cell of dimension $0$} in $S^2$. For a cell $c$ of dimension $0$, we adopt the convention that $\partial c = \emptyset$ and $\inte{c} = c$. 

We record the following definition of cell decompositions from \cite[Definition~3.2]{bonk2017expanding}.

\begin{definition}[Cell decompositions]    \label{def:cell decomposition}
    Let $\mathbf{D}$ be a collection of cells in $S^2$. We say that $\mathbf{D}$ is a \emph{cell decomposition of $S^2$} if the following conditions are satisfied:
    \begin{enumerate}[label= (\roman*)]
        \smallskip
        
        \item the union of all cells in $\mathbf{D}$ is equal to $S^2$,
        
        \smallskip
        
        \item if $c \in \mathbf{D}$, then $\partial c$ is a union of cells in $\mathbf{D}$,
        
        \smallskip
        
        \item for $\juxtapose{c_1}{c_2} \in \mathbf{D}$ with $c_1 \ne c_2$, we have $\inte{c_1} \cap \inte{c_2} = \emptyset$,
        
        \smallskip
        
        \item every point in $S^2$ has a neighborhood that meets only finitely many cells in $\mathbf{D}$.
    \end{enumerate}
\end{definition}

We record \cite[Lemma~5.3]{bonk2017expanding} here to review some facts about cell decompositions.

\begin{lemma}    \label{lem:intersection of cells}
    Let $\mathbf{D}$ be a cell decomposition of $S^2$.
    \begin{enumerate}[label=\rm{(\roman*)}]
        \smallskip
        
        \item     \label{item:lem:intersection of cells:intersection of two tiles} 
            If $\sigma$ and $\tau$ are two distinct cells in $\mathbf{D}$ with $\sigma \cap \tau \ne \emptyset$, then one of the following statements hold: $\sigma \subseteq \partial \tau$, $\tau \subseteq \partial \sigma$, or $\sigma \cap \tau = \partial\sigma \cap \partial\tau$ and this intersection consists of cells in $\mathbf{D}$ of dimension strictly less than $\min\{\dim\sigma, \dim\tau\}$.

        \smallskip
        
        \item     \label{item:lem:intersection of cells:intersection with union of tiles}
            If $\sigma, \, \tau_1, \, \dots, \, \tau_n$ are cells in $\mathbf{D}$ and $\inte{\sigma}\, \cap\, (\tau_1 \, \cup \cdots \cup\, \tau_n) \ne \emptyset$, then $\sigma \subseteq \tau_i$ for some $i \in \oneton$.
    \end{enumerate}
\end{lemma}

\begin{definition}[Refinements]
    Let $\mathbf{D}'$ and $\mathbf{D}$ be two cell decompositions of $S^2$. We say that $\mathbf{D}'$ is a \emph{refinement} of $\mathbf{D}$ if the following conditions are satisfied:
    \begin{enumerate}[label = (\roman*)]
        \smallskip

        \item every cell $c \in \mathbf{D}$ is the union of all cells $c' \in \mathbf{D}'$ with $c' \subseteq c$.

        \smallskip

        \item for every cell $c' \in \mathbf{D}'$ there exists a cell $c \in \mathbf{D}$ with $c' \subseteq c$.
    \end{enumerate}
\end{definition}

\begin{definition}[Cellular maps and cellular Markov partitions]
    Let $\mathbf{D}'$ and $\mathbf{D}$ be two cell decompositions of $S^2$. We say that a continuous map $f \colon S^2 \mapping S^2$ is \emph{cellular} for $(\mathbf{D}', \mathbf{D})$ if for every cell $c \in \mathbf{D}'$, the restriction $f|_c$ of $f$ to $c$ is a homeomorphism of $c$ onto a cell in $\mathbf{D}$. We say that $(\mathbf{D}',\mathbf{D})$ is a \emph{cellular Markov partition} for $f$ if $f$ is cellular for $(\mathbf{D}', \mathbf{D})$ and $\mathbf{D}'$ is a refinement of $\mathbf{D}$.
\end{definition}

Let $f \colon S^2 \mapping S^2$ be a Thurston map, and $\mathcal{C}\subseteq S^2$ be a Jordan curve containing $\operatorname{post}f$. Then the pair $f$ and $\mathcal{C}$ induces natural cell decompositions $\mathbf{D}^n(f,\mathcal{C})$ of $S^2$, for each $n\in \n_0$, in the following way:

By the Jordan curve theorem, the set $S^2 \setminus \mathcal{C}$ has two connected components. We call the closure of one of them the \emph{white $0$-tile} for $(f,\mathcal{C})$, denoted by $X^0_{\white}$, and the closure of the other one the \emph{black $0$-tile} for $(f,\mathcal{C})$, denoted be $X^0_{\black}$. 
The set of $0$-\emph{tiles} is $\mathbf{X}^0(f, \mathcal{C}) \define \bigl\{X^0_{\black}, \, X^0_{\white} \bigr\}$. 
The set of $0$-\emph{vertices} is $\mathbf{V}^0(f, \mathcal{C}) \define \post{f}$. 
We set $\overline{\mathbf{V}}^0(f, \mathcal{C}) \define \bigl\{ \{x\} \describe x\in \mathbf{V}^0(f,\mathcal{C}) \bigr\}$. 
The set of $0$-\emph{edges} $\mathbf{E}^0(f,\mathcal{C})$ is the set of the closures of the connected components of $\mathcal{C} \setminus \post{f}$. 
Then we get a cell decomposition\[
    \mathbf{D}^0(f,\mathcal{C}) \define \mathbf{X}^0(f, \mathcal{C}) \cup \mathbf{E}^0(f,\mathcal{C}) \cup \overline{\mathbf{V}}^0(f,\mathcal{C})
\]
of $S^2$ consisting of \emph{cells of level }$0$, or $0$-\emph{cells}.

We can recursively define the unique cell decomposition $\mathbf{D}^n(f,\mathcal{C})$, $n\in \n$, consisting of $n$-\emph{cells} such that $f$ is cellular for $(\mathbf{D}^{n+1}(f,\mathcal{C}), \mathbf{D}^n(f,\mathcal{C}))$. We refer to \cite[Lemma~5.12]{bonk2017expanding} for more details. We denote by $\mathbf{X}^n(f,\mathcal{C})$ the set of $n$-cells of dimension 2, called $n$-\emph{tiles}; by $\mathbf{E}^n(f,\mathcal{C})$ the set of $n$-cells of dimension $1$, called $n$-\emph{edges}; by $\overline{\mathbf{V}}^n(f,\mathcal{C})$ the set of $n$-cells of dimension $0$; and by $\mathbf{V}^n(f,\mathcal{C})$ the set $\{x \describe \{x\} \in \overline{\mathbf{V}}^n(f,\mathcal{C})\}$, called the set of $n$-\emph{vertices}. The $k$-\emph{skeleton}, for $k\in \{0,1,2\}$, of $\mathbf{D}^n(f,\mathcal{C})$ is the union of all $n$-cells of dimension $k$ in this cell decomposition.

We record \cite[Proposition~5.16]{bonk2017expanding} here to summarize properties of the cell decompositions $\mathbf{D}^n(f,\mathcal{C})$ defined above.

\begin{proposition}[M.~Bonk \& D.~Meyer \cite{bonk2017expanding}]     \label{prop:properties cell decompositions}
    Let $\juxtapose{k}{n} \in \n_0$, $f \colon S^2 \mapping S^2$ be a Thurston map, $\mathcal{C} \subseteq S^2$ be a Jordan curve with $\post{f} \subseteq \mathcal{C}$, and $m \define \card{\post{f}}$.
    \begin{enumerate}[label=\rm{(\roman*)}]
        \smallskip

        \item     \label{item:prop:properties cell decompositions:cellular} 
            The map $f^k$ is cellular for $(\mathbf{D}^{n+k}(f,\mathcal{C}), \mathbf{D}^n(f,\mathcal{C}))$. 
            In particular, if $c$ is any $(n+k)$-cell, then $f^k(c)$ is an $n$-cell, and $f^k|_c$ is a homeomorphism of $c$ onto $f^k(c)$.
        
        \smallskip

        \item     \label{item:prop:properties cell decompositions:union of cells}
        Let $c$ be an $n$-cell. Then $f^{-k}(c)$ is equal to the union of all $(n+k)$-cell $c'$ with $f^k(c') = c$.
        
        \smallskip

        \item     \label{item:prop:properties cell decompositions:skeletons}
            The $1$-skeleton of $\mathbf{D}^n(f,\mathcal{C})$ is equal to $f^{-n}(\mathcal{C})$. The $0$-skeleton of $\mathbf{D}^n(f,\mathcal{C})$ is the set $\mathbf{V}^n(f,\mathcal{C}) = f^{-n}(\post{f})$, and we have $\mathbf{V}^n(f,\mathcal{C}) \subseteq \mathbf{V}^{n+k}(f,\mathcal{C})$.

        \smallskip

        \item     \label{item:prop:properties cell decompositions:cardinality}
            $\card{\mathbf{X}^n(f,\mathcal{C})} = 2(\deg f)^n$, $\card{\mathbf{E}^n(f,\mathcal{C})} = m(\deg f)^n$, and $\card{\mathbf{V}^n(f,\mathcal{C})} \leqslant m (\deg f)^n$.

        \smallskip

        \item     \label{item:prop:properties cell decompositions:edge is boundary of tile}
            The $n$-edges are precisely the closures of the connected components of $f^{-n}(\mathcal{C}) \setminus f^{-n}(\post{f})$. The $n$-tiles are precisely the closures of the connected components of $S^2\setminus f^{-n}(\mathcal{C})$.

        \smallskip

        \item     \label{item:prop:properties cell decompositions:tile is gon}
        Every $n$-tile is an $m$-gon, i.e., the number of $n$-edges and the number of $n$-vertices contained in its boundary are equal to $m$.

        \smallskip

        \item     \label{item:prop:properties cell decompositions:iterate of cell decomposition}
            Let $F \define f^k$ be an iterate of $f$ with $k\in \n$. Then $\mathbf{D}^n(F,\mathcal{C}) = \mathbf{D}^{nk}(f,\mathcal{C})$.
    \end{enumerate}
\end{proposition}






\begin{remark}\label{rem:intersection of two tiles}
Note that for each $n$-edge $e^{n} \in \mathbf{E}^{n}(f,\mathcal{C})$, $n \in \n_0$, there exist exactly two $n$-tiles in $\mathbf{X}^n(f,\mathcal{C})$ containing $e^{n}$.
\end{remark}

For $n\in \n_0$, we define the \emph{set of black $n$-tiles} as\[
    \textbf{X}^n_{\black}(f,\mathcal{C}) \define \left\{ X \in \mathbf{X}^n (f,\mathcal{C}) \describe f^n(X) = X^0_{\black} \right\},
\]
and the \emph{set of white $n$-tiles} as\[
    \mathbf{X}^n_{\white}(f,\mathcal{C}) \define \left\{X\in \mathbf{X}^n(f,\mathcal{C}) \describe f^n(X) = X^0_{\white}\right\}.
\]

From now on, if the map $f$ and the Jordan curve $\mathcal{C}$ are clear from the context, we will sometimes omit $(f,\mathcal{C})$ in the notation above.


If we fix the cell decomposition $\mathbf{D}^n(f,\mathcal{C})$, $n\in \n_0$, we can define for each $v\in \mathbf{V}^n$ the \emph{$n$-flower of $v$} as
\begin{equation}    \label{eq:n-flower}
    W^n(v) \define \bigcup \left\{\inte{c} \describe c\in \mathbf{D}^n(f,\mathcal{C}), v\in c\right\}.
\end{equation}
Note that flowers are open (in the standard topology on $S^2$). Let $\overline{W}^n(v)$ be the closure of $W^n(v)$. 

\begin{remark} \label{rem:flower structure}
    For each $n \in \n_0$ and each $v \in \vertex{n} $, we have
    \[
        \cflower{n}{v} = X_1 \cup X_2 \cup \cdots \cup X_m,
    \]
    where $m \define 2\deg_{f^n}(v)$, and $X_1, \, X_2, \, \dots \, X_m$ are all the $n$-tiles that contain $v$ as a vertex (see \cite[Lemma~5.28]{bonk2017expanding}). 
    Moreover, each flower is mapped under $f$ to another flower in such a way that is similar to the map $z \mapsto z^k$ on the complex plane. 
    More precisely, for each $n \in \n_0$ and each $v \in \vertex{n + 1}$, there exist orientation preserving homeomorphisms $\varphi \colon W^{n+1}(v) \rightarrow \mathbb{D}$ and $\eta\colon W^{n}(f(v)) \rightarrow \mathbb{D}$ such that $\mathbb{D}$ is the unit disk on $\mathbb{C}$, $\varphi(v) = 0$, $\eta(f(v)) = 0$, and 
    \begin{equation*}
        (\eta\circ f \circ \varphi^{-1}) (z) = z^k
    \end{equation*}
    for all $z \in \mathbb{D}$, where $k \coloneqq \deg_f(v)$. Let $\overline{W}^{n+1}(v) = X_1 \cup X_2\cup \cdots \cup X_m$ and $\overline{W}^n(f(v)) = X'_1 \cup X'_2 \cup \cdots \cup X'_{m'}$, where $X_1, \, X_2, \, \dots \, X_m$ are all the $(n+1)$-tiles that contain $v$ as a vertex, listed counterclockwise, and $X'_1, \, X'_2, \, \dots \, X'_{m'}$ are all the $n$-tiles that contain $f(v)$ as a vertex, listed counterclockwise, and $f(X_1) = X'_1$. 
    Then $m = m'k$, and $f(X_i) = X'_j$ if $i\equiv j \pmod{k}$, where $ k = \deg_f(v)$ (see Case~3 of the proof of \cite[Lemma~5.24]{bonk2017expanding} for more details). 
    In particular, $W^n(v)$ is simply connected.
\end{remark}

\begin{remark}\label{rem:flower preserve under iterate}
It follows from Remark~\ref{rem:flower structure} and Proposition~\ref{prop:properties cell decompositions} that the map $f$ preserves the structure of flowers, or more precisely,
\begin{equation}    \label{eq:flower preserve under iterate}
    f(W^n(x)) = W^{n-1}(f(x))
\end{equation}
for each $n \in \n$ and each $x \in \Vertex{n}$.
\end{remark}

We denote, for each $x \in S^2$ and each $n \in \z$, the \emph{$n$-bouquet of $x$}
\begin{equation}    \label{eq:Un bouquet of point}
    U^n(x) \define \bigcup \bigl\{ Y^n \in \mathbf{X}^n \describe \text{there exists } X^n \in \mathbf{X}^n \text{ with } x\in X^n, \, X^n \cap Y^n \ne \emptyset \bigr\}
\end{equation}
if $n \geqslant 0$, and set $U^n(x) \define S^2$ otherwise. 


We can now define expanding Thurston maps.

\begin{definition}[Expansion]     \label{def:expanding_Thurston_maps}
    A Thurston map $f \colon S^2 \mapping S^2$ is called \emph{expanding} if there exists a metric $d$ on $S^2$ that induces the standard topology on $S^2$ and a Jordan curve $\mathcal{C} \subseteq S^2$ containing $\post{f}$ such that
    \begin{equation}    \label{eq:definition of expansion}
        \lim_{n \to +\infty} \max\{ \diam{d}{X} \describe X \in \mathbf{X}^n(f,\mathcal{C}) \} = 0.
    \end{equation}
\end{definition}

\begin{remark}\label{rem:Expansion_is_independent}
    It is clear from Proposition~\ref{prop:properties cell decompositions}~\ref{item:prop:properties cell decompositions:iterate of cell decomposition} and Definition~\ref{def:expanding_Thurston_maps} that if $f$ is an expanding Thurston map, so is $f^n$ for each $n\in \n$. We observe that being expanding is a topological property of a Thurston map and independent of the choice of the metric $d$ that generates the standard topology on $S^2$. By Lemma~6.2 in \cite{bonk2017expanding}, it is also independent of the choice of the Jordan curve $\mathcal{C}$ containing $\post{f}$. More precisely, if $f$ is an expanding Thurston map, then\[
        \lim_{n \to +\infty} \max \bigl\{ \diam{\widetilde{d}}{X} \describe X\in \mathbf{X}^n(f,\widetilde{\mathcal{C}}) \bigr\} = 0,        
    \]
    for each metric $\widetilde{d}$ that generates the standard topology on $S^2$ and each Jordan curve $\widetilde{\mathcal{C}} \subseteq S^2$ that contains $\post{f}$.
\end{remark}

For an expanding Thurston map $f$, we can fix a particular metric $d$ on $S^2$ called a \emph{visual metric for $f$}. 
For the existence and properties of such metrics, see \cite[Chapter~8]{bonk2017expanding}. 
For a visual metric $d$ for $f$, there exists a unique constant $\Lambda > 1$ called the \emph{expansion factor} of $d$ (see \cite[Chapter~8]{bonk2017expanding} for more details). 
One major advantage of a visual metric $d$ is that in $(S^2,d)$ we have good quantitative control over the sizes of the cells in the cell decompositions discussed above. We summarize several results of this type (\cite[Proposition~8.4, Lemmas~8.10, and~8.11]{bonk2017expanding}) in the lemma below.

\begin{lemma}[M.~Bonk \& D.~Meyer \cite{bonk2017expanding}]    \label{lem:visual_metric}
    Let $f \colon S^2 \mapping S^2$ be an expanding Thurston map, and $\mathcal{C} \subseteq S^2$ be a Jordan curve containing $\post{f}$. Let $d$ be a visual metric on $S^2$ for $f$ with expansion factor $\Lambda > 1$. Then there exist constants $C \geqslant 1$, $K \geqslant 1$, and $n_0 \in \n_0$ with the following properties:
    \begin{enumerate}[label=\rm{(\roman*)}]

        \smallskip
        
        \item     \label{item:lem:visual_metric:distinct cell separated} 
            $d(\sigma,\tau) \geqslant C^{-1}\Lambda^{-n}$ whenever $\sigma$ and $\tau$ are disjoint $n$-cells for some $n \in \n_0$.

        \smallskip

        \item     \label{item:lem:visual_metric:diameter of cell}
            $C^{-1}\Lambda^{-n} \leqslant \diam{d}{\tau} \leqslant C\Lambda^{-n}$ for all $n$-edges and all $n$-tiles $\tau$ and for all $n \in \n_0$.
        
        \smallskip
        
        \item     \label{item:lem:visual_metric:bouquet bounded by ball}
            $B_{d}(x, K^{-1}\Lambda^{-n}) \subseteq U^n(x) \subseteq B_{d}(x, K\Lambda^{-n})$ for each $x \in S^2$ and each $n \in \n_0$.

        \smallskip
        
        \item     \label{item:lem:visual_metric:ball bounded by bouquet}
            $U^{n+n_0}(x) \subseteq B_{d}(x,r) \subseteq U^{n-n_0}(x)$ where $n \define \lceil -\log{r}/\log{\Lambda} \rceil$ for all $r > 0$ and $x \in S^2$.
        
        \smallskip
                
        \item     \label{item:lem:visual_metric:tile contain ball and bounded by ball}
            For every $n$-tile $X^n \in \mathbf{X}^n(f,\mathcal{C})$, $n\in \n_0$, there exists a point $p\in X^n$ such that $B_{d}(p, C^{-1}\Lambda^{-n}) \subseteq X^n \subseteq B_{d}(p, C\Lambda^{-n})$.
    \end{enumerate}

    Conversely, if $\widetilde{d}$ is a metric on $S^2$ satisfying conditions $\textnormal{(i)}$ and $\textnormal{(ii)}$ for some constant $C \geqslant 1$, then $\widetilde{d}$ is a visual metric with expansion factor $\Lambda > 1$.
\end{lemma}

Recall $U^n(x)$ is defined in \eqref{eq:Un bouquet of point}.



        

\begin{remark}\label{rem:chordal metric visual metric qs equiv}
    If $f \colon \ccx \mapping \ccx$ is a rational expanding Thurston map, then a visual metric is quasisymmetrically equivalent to the chordal metric on the Riemann sphere $\ccx$ (see \cite[Theorem~18.1~(ii)]{bonk2017expanding}). 
    Here the chordal metric $\sigma$ on $\ccx$ is given by $\sigma (z, w) \define \frac{2\abs{z - w}}{\sqrt{1 + \abs{z}^2} \sqrt{1 + \abs{w}^2}}$ for all $\juxtapose{z}{w} \in \cx$, and $\sigma(\infty, z) = \sigma(z, \infty) \define \frac{2}{\sqrt{1 + \abs{z}^2}}$ for all $z \in \cx$. 
    We also note that quasisymmetric embeddings of bounded connected metric spaces are \holder continuous (see \cite[Section~11.1 and Corollary~11.5]{heinonen2001lectures}). 
    Accordingly, the classes of \holder continuous functions on $\ccx$ equipped with the chordal metric and on $S^2 = \ccx$ equipped with any visual metric for $f$ are the same (up to a change of the \holder exponent).
\end{remark}

A Jordan curve $\mathcal{C} \subseteq S^2$ is \emph{$f$-invariant} if $f(\mathcal{C}) \subseteq \mathcal{C}$. If $\mathcal{C}$ is $f$-invariant with $\post{f} \subseteq \mathcal{C}$, then the cell decompositions $\mathbf{D}^{n}(f, \mathcal{C})$ have nice compatibility properties. In particular, $\mathbf{D}^{n+k}(f, \mathcal{C})$ is a refinement of $\mathbf{D}^{n}(f, \mathcal{C})$, whenever $\juxtapose{n}{k} \in \n_0$. Intuitively, this means that each cell $\mathbf{D}^{n}(f, \mathcal{C})$ is ``subdivided'' by the cells in $\mathbf{D}^{n+k}(f, \mathcal{C})$. A cell $c\in \mathbf{D}^{n}(f, \mathcal{C})$ is actually subdivided by the cells in $\mathbf{D}^{n+k}(f, \mathcal{C})$ ``in the same way'' as the cell $f^n(c) \in \mathbf{D}^{0}(f, \mathcal{C})$ by the cells in $\mathbf{D}^{k}(f, \mathcal{C})$. 

For convenience we record Proposition~12.5~(ii) of \cite{bonk2017expanding} here, which is easy to check but useful. 

\begin{proposition}[M.~Bonk \& D.~Meyer \cite{bonk2017expanding}]    \label{prop:cell decomposition: invariant Jordan curve}
    Let $\juxtapose{k}{n} \in \n_0$, $f \colon S^2 \mapping S^2$ be a Thurston map, and $\mathcal{C} \subseteq S^2$ be an $f$-invariant Jordan curve with $\post{f} \subseteq \mathcal{C}$. Then every $(n+k)$-tile $X^{n+k}$ is contained in a unique $k$-tile $X^k$.
\end{proposition}

We are interested in $f$-invariant Jordan curves that contain $\post{f}$, since for such a Jordan curve $\mathcal{C}$, we get a cellular Markov partition $\bigl( \mathbf{D}^{1}(f,\mathcal{C}), \mathbf{D}^{0}(f,\mathcal{C}) \bigr)$ for $f$. 
According to Example~15.11 in \cite{bonk2017expanding}, such $f$-invariant Jordan curves containing $\post{f}$ need not exist. 
However, M.~Bonk and D.~Meyer \cite[Theorem~15.1]{bonk2017expanding} proved that there exists an $f^n$-invariant Jordan curve $\mathcal{C}$ containing $\post{f}$ for each sufficiently large $n$ depending on $f$. We record it below for the convenience of the reader.

\begin{lemma}[M.~Bonk \& D.~Meyer \cite{bonk2017expanding}]    \label{lem:invariant_Jordan_curve}
    Let $f \colon S^2 \mapping S^2$ be an expanding Thurston map, and $\widetilde{\mathcal{C}} \subseteq S^2$ be a Jordan curve with $\post{f} \subseteq \widetilde{\mathcal{C}}$. Then there exists an integer $N(f, \widetilde{\mathcal{C}}) \in \n$ such that for each $n \geqslant N(f,\widetilde{\mathcal{C}})$ there exists an $f^n$-invariant Jordan curve $\mathcal{C}$ isotopic to $\widetilde{\mathcal{C}}$ rel. $\post{f}$.
\end{lemma}

The following distortion lemma for expanding Thurston maps follows immediately from \cite[Lemma~5.1]{li2018equilibrium}.

\begin{lemma}    \label{lem:distortion_lemma}
    Let $f \colon S^2 \mapping S^2$ be an expanding Thurston map, and $\mathcal{C} \subseteq S^2$ be a Jordan curve that satisfies $\post{f} \subseteq \mathcal{C}$ and $f^{n_{\mathcal{C}}}(\mathcal{C}) \subseteq \mathcal{C}$ for some $n_{\mathcal{C}} \in \n$. 
    Let $d$ be a visual metric on $S^2$ for $f$ with expansion factor $\Lambda > 1$.
    Let $\potential \in \holderspacesphere$ be a real-valued \holder continuous function with an exponent $\holderexp \in (0, 1]$.
    Then there exists a constant $C_1 \geqslant 0$ depending only on $f$, $\mathcal{C}$, $d$, $\phi$, and $\holderexp$ such that for all $n \in \n_0$, $X^n \in \Tile{n}$, and $\juxtapose{x}{y} \in X^n$,
    \begin{equation}    \label{eq:distortion_lemma}
        \left| S_n\phi(x) - S_n\phi(y) \right| \leqslant C_1 d(f^n(x), f^n(y))^{\holderexp} \leqslant \Cdistortion.
    \end{equation}
    Quantitatively, we choose
    \begin{equation}     \label{eq:const:C_1}
        C_1 \define C_0  \holderseminorm{\potential}{S^2}  \big/ \bigl( 1 - \Lambda^{-\holderexp} \bigr) ,
    \end{equation}
    where $C_0 > 1$ is a constant depending only on $f$, $\mathcal{C}$, and $d$.
\end{lemma}

We summarize the existence, uniqueness, and some basic properties of equilibrium states for expanding Thurston maps in the following theorem.
\begin{theorem}[Z.~Li \cite{li2018equilibrium}]     \label{thm:properties of equilibrium state}
    Let $f \colon S^2 \mapping S^2$ be an expanding Thurston map and $d$ a visual metric on $S^2$ for $f$. 
    Let $\juxtapose{\phi}{\gamma} \in C^{0,\holderexp}(S^2,d)$ be real-valued \holder continuous functions with an exponent $\holderexp \in (0,1]$. 
    Then the following statements are satisfied:
    \begin{enumerate}[label=\rm{(\roman*)}]
        \smallskip
        
        \item     \label{item:thm:properties of equilibrium state:existence and uniqueness}
        There exists a unique equilibrium state $\mu_{\phi}$ for the map $f$ and the potential $\phi$.

        

        \smallskip
        
        \item     \label{item:thm:properties of equilibrium state:edge measure zero}
        If $\mathcal{C} \subseteq S^2$ is a Jordan curve containing $\post{f}$ with the property that $f^{n_{\mathcal{C}}}(\mathcal{C}) \subseteq \mathcal{C}$ for some $n_{\mathcal{C}} \in \n$, then $\mu_{\phi} \bigl( \bigcup_{i=0}^{+\infty} f^{-i}(\mathcal{C}) \bigr) = 0$.
    \end{enumerate}
\end{theorem}

Theorem~\ref{thm:properties of equilibrium state}~\ref{item:thm:properties of equilibrium state:existence and uniqueness} is part of \cite[Theorem~1.1]{li2018equilibrium}. 
Theorem~\ref{thm:properties of equilibrium state}~\ref{item:thm:properties of equilibrium state:edge measure zero} was established in \cite[Proposition~7.1]{li2018equilibrium}.

Let $f \colon S^2 \mapping S^2$ be an expanding Thurston map, $\mathcal{C} \subseteq S^2$ be a Jordan curve containing $\post{f}$, and $\varphi \in C(S^2)$ be a real-valued continuous function. 
We now define the Gibbs measures with respect to $f$, $\mathcal{C}$, and $\varphi$.
\begin{definition}[Gibbs measures]    \label{def:gibbs measure}
    A Borel probability measure $\mu \in \mathcal{P}(S^2)$ is a \emph{Gibbs measure} with respect to $f$, $\mathcal{C}$, and $\varphi$ if there exist constants $P_{\mu} \in \real$ and $C_{\mu} \geqslant 1$ such that for each $n\in \n_0$, each $n$-tile $X^n \in \mathbf{X}^n(f,\mathcal{C})$, and each $x \in X^n$, we have
    \begin{equation*}
        \frac{1}{C_{\mu}} \leqslant \frac{\mu(X^n)}{\myexp{S_{n}\varphi(x) - nP_{\mu}}} \leqslant C_{\mu}.
    \end{equation*}
\end{definition}

One observes that for each Gibbs measure $\mu$ with respect to $f$, $\mathcal{C}$, and $\varphi$, the constant $P_{\mu}$ is unique. Actually, the equilibrium state $\mu_{\phi}$ is an $f$-invariant Gibbs measure with respect to $f$, $\mathcal{C}$, and $\phi$, with $P_{\mu_{\phi}} = P(f,\phi)$ (see \cite[Theorem~5.16, Proposition~5.17]{li2018equilibrium}). 
We record this result below for the convenience of the reader.
\begin{proposition}[Z.~Li \cite{li2018equilibrium}]    \label{prop:equilibrium state is gibbs measure}
     Let $f \colon S^2 \mapping S^2$ be an expanding Thurston map and $\mathcal{C} \subseteq S^2$ be a Jordan curve containing $\post{f}$ with the property that $f^{n_{\mathcal{C}}}(\mathcal{C}) \subseteq \mathcal{C}$ for some $n_{\mathcal{C}} \in \n$. Let $d$ be a visual metric on $S^2$ for $f$ and $\phi \in C^{0,\holderexp}(S^2,d)$ be a real-valued \holder continuous function with an exponent $\holderexp \in (0,1]$. 
     Then the equilibrium state $\mu_{\phi}$ for $f$ and $\phi$ is a Gibbs measure with respect to $f$, $\mathcal{C}$, and $\phi$, with the constant $P_{\mu_{\phi}} = P(f, \phi)$, i.e., there exists a constant $C_{\mu_{\phi}} \geqslant 1$ such that for each $n\in \n_0$, each $n$-tile $X^n \in \mathbf{X}^n(f,\mathcal{C})$, and each $x \in X^n$, we have
     \begin{equation}    \label{eq:equilibrium state is gibbs measure}
         \frac{1}{C_{\mu_{\phi}}} \leqslant \frac{\mu_{\phi}(X^n)}{\myexp{ S_{n}\phi(x) - n P(f,\phi) }} \leqslant C_{\mu_{\phi}}.
     \end{equation}
\end{proposition}

We next introduce pair structures associated with tile structures induced by an expanding Thurston map.
We refer the reader to \cite[Section~7.2]{shi2023thermodynamic} for details.

\begin{definition}[Pair structures]    \label{def:pair structures}
    Let $f \colon S^2 \mapping S^2$ be an expanding Thurston map with a Jordan curve $\mathcal{C}\subseteq S^2$ satisfying $\post{f} \subseteq \mathcal{C}$.
    Fix an arbitrary $0$-edge $e^0 \in \mathbf{E}^0(f,\mathcal{C})$. 
    For each $n\in \n$, we can pair a white $n$-tile $X^n_{\white} \in \mathbf{X}^n_{\white}$ and a black $n$-tile $X^n_{\black} \in \mathbf{X}^n_{\black}$ whose intersection $X^n_{\white} \cap X^n_{\black}$ contains an $n$-edge contained in $f^{-n}(e^0)$. 
    We define the \emph{set of $n$-pairs} (with respect to $f$, $\mathcal{C}$, and $e^{0}$), denoted by $\mathbf{P}^n(f,\mathcal{C},e^0)$, to be the collection of the union $X^n_{\white} \cup X^n_{\black}$ of such pairs (called \emph{$n$-pairs}), i.e.,
    \begin{equation}    \label{eq:definition of n-pairs}
        \mathbf{P}^n(f,\mathcal{C}, e^0) \define \bigl\{ X^n_{\white} \cup X^n_{\black} \describe X^n_{\white} \in \mathbf{X}^n_{\white}, \, X^n_{\black} \in \mathbf{X}^n_{\black}, \, X^n_{\white} \cap X^n_{\black} \cap f^{-n}\bigl(e^0\bigr) \in \mathbf{E}^n(f,\mathcal{C}) \bigr\}.
    \end{equation}
\end{definition}
Figure~\ref{fig:pair example} illustrates the structure of $n$-pairs. 
Note that there are a total of $(\deg f)^n$ such pairs and each $n$-tile is in exactly one such pair (see Lemma~\ref{lem:pairs are disjoint}).

\begin{figure}[H]
    \centering
    \vspace*{.5cm}
    \begin{overpic}
        [width=12cm, tics=5]{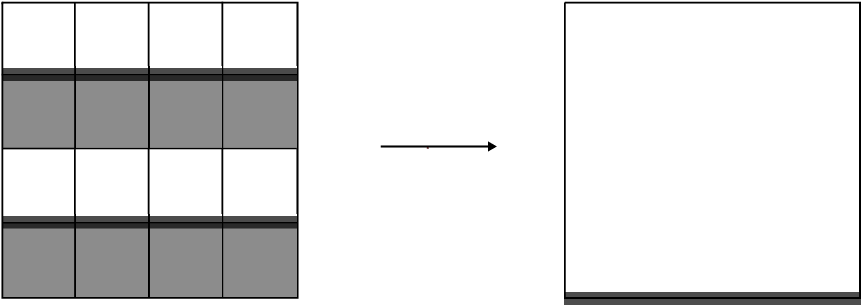}
        \put(13,38){$\mathbf{D}^n(f,\mathcal{C})$}
        \put(78,38){$\mathbf{D}^0(f,\mathcal{C})$}
        \put(49,22){$f^n$}
        \put(82,-3.5){$e^0$}
    \end{overpic}
    \vspace{.1cm}
    \caption{The graph of $n$-pairs.}
    \label{fig:pair example}
\end{figure}
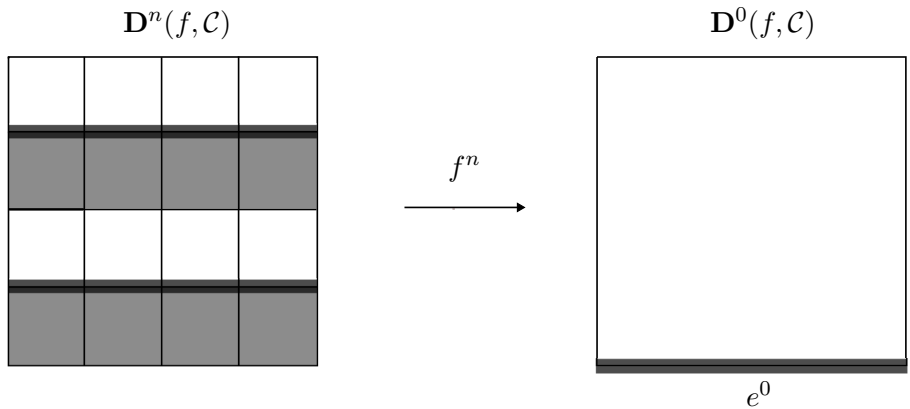 


From now on, if the map $f$, the Jordan curve $\mathcal{C}$, and the $0$-edge $e^0$ are clear from the context, we will sometimes omit $(f,\mathcal{C}, e^0)$ in the notation above.

We record \cite[Lemma~7.6]{shi2023thermodynamic} here. 
\begin{lemma}[Z.~Li, X.~Shi, Y.~Zhang \cite{shi2023thermodynamic}]    \label{lem:pairs are disjoint}
    Let $f \colon S^2 \mapping S^2$ be an expanding Thurston map with a Jordan curve $\mathcal{C}\subseteq S^2$ satisfying $\post{f} \subseteq \mathcal{C}$.
    Fix an arbitrary $0$-edge $e^0 \in \mathbf{E}^0(f,\mathcal{C})$. 
    Then for each $n\in\n$ and any two distinct $n$-pairs $\juxtapose{P^n}{\widetilde{P}^n} \in \mathbf{P}^n$, their interiors are disjoint.
\end{lemma}

We record the following lemma from \cite[Lemma~7.9]{shi2023thermodynamic}. 
Recall from \eqref{eq:Un bouquet of point} that $U^{n}(x)$ is the \emph{$n$-bouquet of $x$}.

\begin{lemma}[Z.~Li, X.~Shi, Y.~Zhang \cite{shi2023thermodynamic}]    \label{lem:pair neighbor disjoint with Jordan curve}
    Let $f \colon S^2 \mapping S^2$ be an expanding Thurston map with a Jordan curve $\mathcal{C}\subseteq S^2$ satisfying $\post{f} \subseteq \mathcal{C}$.
    Let $d$ be a visual metric on $S^2$ for $f$.
    Fix an arbitrary $0$-edge $e^0 \in \mathbf{E}^0(f,\mathcal{C})$.
    We assume in addition that $f(\mathcal{C}) \subseteq \mathcal{C}$. 
    Then there exists an integer $M \in \n$ depending only on $f$, $\mathcal{C}$, $d$, and $e^0$ such that for each color $\colour \in \colours$, there exists an $M$-pair $P^{M}_{\colour} \in \mathbf{P}^{M}$ such that for each integer $n \geqslant M$ and each $x \in P^{M}_{\colour}$, we have $U^{n}(x) \subseteq \inte[\big]{X^{0}_{\colour}}$.
\end{lemma}

\subsection{Subsystems of expanding Thurston maps}%
\label{sub:Subsystems of expanding Thurston maps}

In this subsection, we review some concepts and results on subsystems of expanding Thurston maps. 
We refer the reader to \cite{shi2023thermodynamic,shi2024uniqueness} for details.

\smallskip

We first introduce the definition of subsystems along with relevant concepts and notations that will be used frequently throughout this paper.
Additionally, we will provide examples to illustrate these ideas.

\begin{definition}    \label{def:subsystems}
    Let $f \colon S^2 \mapping S^2$ be an expanding Thurston map with a Jordan curve $\mathcal{C}\subseteq S^2$ satisfying $\post{f} \subseteq \mathcal{C}$. 
    We say that a map $F \colon \domF \mapping S^2$ is a \emph{subsystem of $f$ with respect to $\mathcal{C}$} if $\domF = \bigcup \mathfrak{X}$ for some non-empty subset $\mathfrak{X} \subseteq \Tile{1}$ and $F = f|_{\domF}$.
    We denote by $\subsystem$ the set of all subsystems of $f$ with respect to $\mathcal{C}$.
    Define \[
        \operatorname{Sub}_{*}(f, \mathcal{C}) \define \{ F \in \subsystem \describe \domF \subseteq F(\domF) \}.
    \]
\end{definition}

Consider a subsystem $F \in \subsystem$. 
For each $n \in \n_0$, we define the \emph{set of $n$-tiles of $F$} to be
\begin{equation}    \label{eq:definition of tile of subsystem}
    \Domain{n} \define \{ X^n \in \Tile{n} \describe X^n \subseteq F^{-n}(F(\domF)) \},
\end{equation}
where we set $F^0 \define \id{S^{2}}$ when $n = 0$. We call each $X^n \in \Domain{n}$ an \emph{$n$-tile} of $F$. 
We define the \emph{tile maximal invariant set} associated with $F$ with respect to $\mathcal{C}$ to be
\begin{equation}    \label{eq:def:limitset}
    \limitset(F, \mathcal{C}) \define \bigcap_{n \in \n} \Bigl( \bigcup \Domain{n} \Bigr), 
\end{equation}
which is a compact subset of $S^{2}$. 
Indeed, $\limitset(F, \mathcal{C})$ is forward invariant with respect to $F$, namely, $F(\limitset(F, \mathcal{C})) \subseteq \limitset(F, \mathcal{C})$ (see Proposition~\ref{prop:subsystem:preliminary properties}~\ref{item:subsystem:properties:limitset forward invariant}). 
We denote by $\limitmap$ the map $F|_{\limitset(F, \mathcal{C})} \colon \limitset(F, \mathcal{C}) \mapping \limitset(F, \mathcal{C})$.

Let $\juxtapose{X^0_{\black}}{X^0_{\white}} \in \mathbf{X}^0(f, \mathcal{C})$ be the black $0$-tile and the white $0$-tile, respectively. 
We define the \emph{color set of $F$} as \[
    \colourset \define \bigl\{ \colour \in \colours \describe X^0_{\colour} \in \Domain{0} \bigr\}.
\]
For each $n \in \n_0$, we define the \emph{set of black $n$-tiles of $F$} as\[
    \bFTile{n} \define \bigl\{ X \in \Domain{n} \describe F^{n}(X) = X^0_{\black} \bigr\},
\] 
and the \emph{set of white $n$-tiles of $F$} as\[
    \wFTile{n} \define \bigl\{ X \in \Domain{n} \describe F^{n}(X) = X^0_{\white} \bigr\}. 
\]
Moreover, for each $n \in \n_0$ and each pair of $\juxtapose{\colour}{\colour'} \in \colours$ we define 
\[
    \ccFTile{n}{\colour}{\colour'} \define \bigl\{ X \in \cFTile{n} \describe X \subseteq X^0_{\colour'} \bigr\}.
\]
In other words, for example, a tile $X \in \ccFTile{n}{\black}{\white}$ is a \emph{black $n$-tile of $F$ contained in $\whitetile$}, i.e., an $n$-tile of $F$ that is contained in the white $0$-tile $X^0_{\white}$ as a set, and is mapped by $F^{n}$ onto the black $0$-tile $\blacktile$.

By abuse of notation, we often omit $(F, \mathcal{C})$ in the notations above when it is clear from the context.

We discuss two examples below and refer the reader to \cite[Subsection~5.1]{shi2023thermodynamic} for more examples.

\begin{example}    \label{exam:subsystems}
    Let $f \colon S^2 \mapping S^2$ be an expanding Thurston map with a Jordan curve $\mathcal{C}\subseteq S^2$ satisfying $\post{f} \subseteq \mathcal{C}$.
    Consider $F \in \subsystem$.
    \begin{enumerate}[label=(\roman*)]


        \item     \label{item:exam:subsystems:Sierpinski carpet}
            The map $F \colon \domF \mapping S^2$ is represented by Figure~\ref{fig:subsystem:example:carpet}.
            Here $S^{2}$ is identified with a pillow that is obtained by gluing two squares together along their boundaries.
            Moreover, each square is subdivided into $3\times 3$ subsquares, and $\dom{F}$ is obtained from $S^2$ by removing the interior of the middle subsquare $X^{1}_{\white} \in \cTile{1}{\white}$ and $X^{1}_{\black} \in \cTile{1}{\black}$ of the respective squares. 
            In this case, $\limitset$ is a \sierpinski carpet. 
            It consists of two copies of the standard square \sierpinski carpet glued together along the boundaries of the squares.
            \begin{figure}[H]
                \centering
                \begin{overpic}
                    [width=12cm, tics=20]{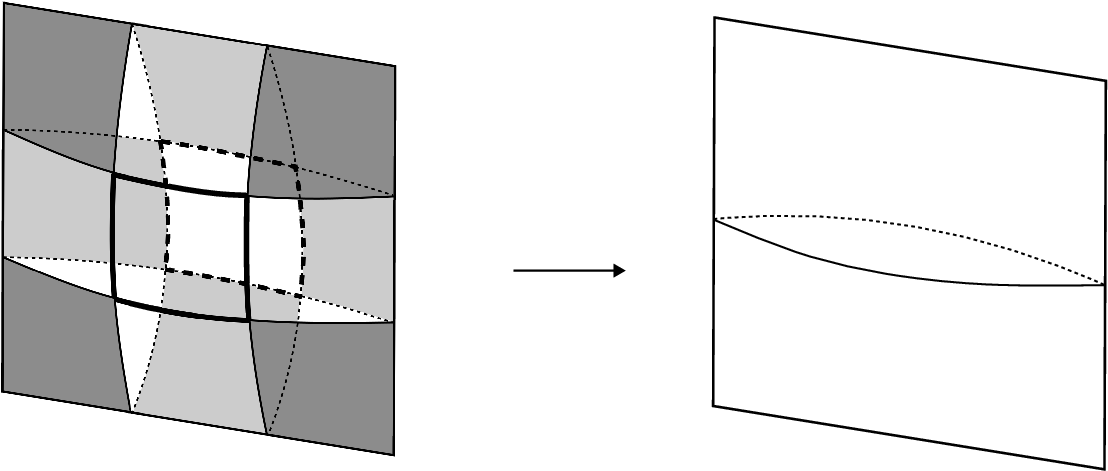}
                    \put(50,20){$F$}
                    \put(16,41){$\domF$}
                    \put(83,40){$S^2$}
                \end{overpic}
                \caption{A \sierpinski carpet subsystem.} 
                \label{fig:subsystem:example:carpet}
            \end{figure}

        \smallskip

        \item     \label{item:exam:subsystems:Sierpinski gasket}
            The map $F \colon \domF \mapping S^2$ is represented by Figure~\ref{fig:subsystem:example:gasket}.
            Here $S^{2}$ is identified with a pillow that is obtained by gluing two equilateral triangles together along their boundaries.
            Moreover, each triangle is subdivided into $4$ small equilateral triangles, and $\dom{F}$ is obtained from $S^2$ by removing the interior of the middle small triangle $X^{1}_{\black} \in \cTile{1}{\black}$ and $X^{1}_{\white} \in \cTile{1}{\white}$ of the respective triangle. 
            In this case, $\limitset$ is a \sierpinski gasket. 
            It consists of two copies of the standard \sierpinski gasket glued together along the boundaries of the triangles.
            \begin{figure}[H]
                \centering
                \begin{overpic}
                    [width=12cm, tics=20]{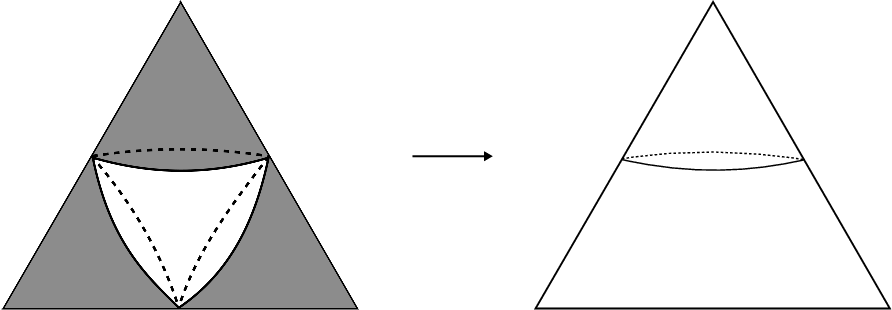}
                    \put(50,19){$F$}
                    \put(0,23){$\domF$}
                    \put(90,24){$S^2$}
                \end{overpic}
                \caption{A \sierpinski gasket subsystem.} 
                \label{fig:subsystem:example:gasket}
            \end{figure}
    \end{enumerate}
\end{example}

We summarize some preliminary results for subsystems in the following proposition.

\begin{proposition}[Z.~Li, X.~Shi, Y.~Zhang \cite{shi2023thermodynamic}]    \label{prop:subsystem:preliminary properties}
    Let $f \colon S^2 \mapping S^2$ be an expanding Thurston map with a Jordan curve $\mathcal{C}\subseteq S^2$ satisfying $\post{f} \subseteq \mathcal{C}$.
    Consider $F \in \subsystem$.
    Then the following statements hold:
    \begin{enumerate}[label=\rm{(\roman*)}]
        \smallskip

        \item     \label{item:subsystem:properties:limitset forward invariant}
            The tile maximal invariant set $\limitset$ is forward invariant with respect to $F$, i.e., $F(\limitset) \subseteq \limitset$.

        \smallskip

        \item     \label{item:subsystem:properties invariant Jordan curve:relation between color and location of tile}
            If $f(\mathcal{C}) \subseteq \mathcal{C}$, then $\cFTile{m} = \ccFTile{m}{\colour}{\black} \cup \ccFTile{m}{\colour}{\white}$ for each $m \in \n_{0}$ and each $\colour \in \colours$.

        \smallskip

        \item     \label{item:subsystem:properties invariant Jordan curve:backward invariant limitset outside invariant Jordan curve}
            If $f(\mathcal{C}) \subseteq \mathcal{C}$, then $F^{-1}(\limitset \setminus \mathcal{C}) \subseteq \limitset \setminus \mathcal{C}$.
    \end{enumerate}
\end{proposition}

We introduce a $2 \times 2$ matrix called the tile matrix to describe tiles of a subsystem according to their colors and locations.

\smallskip

\begin{definition}[Tile matrices]     \label{def:tile matrix} 
    Let $f \colon S^2 \mapping S^2$ be an expanding Thurston map with a Jordan curve $\mathcal{C}\subseteq S^2$ satisfying $\post{f} \subseteq \mathcal{C}$.
    Consider $F \in \subsystem$.
    We define the \emph{tile matrix} of $F$ with respect to $\mathcal{C}$ as
    \begin{equation}    \label{eq:definition of tile matrix}
            A = A(F, \mathcal{C}) \define \begin{bmatrix}
            N_{\white \white} & N_{\black \white} \\
            N_{\white \black} & N_{\black \black}
        \end{bmatrix}
    \end{equation}
    where \[
        N_{\colour \colour'} = N_{\colour\colour'}(A) \define \operatorname{card}{\! \bigl\{ X \in \cFTile{1} \describe X \subseteq X^0_{\colour'} \bigr\}} = \card[\big]{\ccFTile{1}{\colour}{\colour'}}
    \]
    for each pair of colors $\juxtapose{\colour}{\colour'} \in \colours$. For example, $N_{\black \white}$ is the number of black tiles in $\Domain{1}$ which are contained in the white $0$-tile $X^0_{\white}$.
    Recall that $\juxtapose{X^0_{\black}}{X^0_{\white}} \in \mathbf{X}^0(f, \mathcal{C})$ is the black $0$-tile and the white $0$-tile, respectively.
\end{definition}


\begin{remark}    \label{rem:tile matrix associated with set of tiles}
    Note that the tile matrix $A(F, \mathcal{C})$ of $F$ with respect to $\mathcal{C}$ is completely determined by the set $\Domain{1}$. Thus for each integer $n \in \n_0$ and each set of $n$-tiles $\mathbf{E} \subseteq \mathbf{X}^{n}(f, \mathcal{C})$, similarly, we can define the tile matrix $A(\mathbf{E})$ of $\mathbf{E}$ as 
    \[
        A(\mathbf{E}) \define 
        \begin{bmatrix}
            N_{\white \white}(\mathbf{E}) & N_{\black \white}(\mathbf{E}) \\
            N_{\white \black}(\mathbf{E}) & N_{\black \black}(\mathbf{E}),
        \end{bmatrix}
    \]
    where $N_{\colour \colour'}(\mathbf{E}) \define \card[\big]{ \set[\big]{X \in \mathbf{E} \describe X \in \cTile{n}{\colour}, \, X \subseteq X^0_{\colour'}} } = \card[\big]{\ccFTile{n}{\colour}{\colour'}}$ for each pair of $\juxtapose{\colour}{\colour'} \in \colours$.
\end{remark}

\smallskip

We record the following definition from \cite[Subsection~5.5]{shi2023thermodynamic}.
\begin{definition}[Primitivity]    \label{def:primitivity of subsystem}
    Let $f \colon S^2 \mapping S^2$ be an expanding Thurston map with a Jordan curve $\mathcal{C}\subseteq S^2$ satisfying $\post{f} \subseteq \mathcal{C}$.
    Consider $F \in \subsystem$.
    We say that $F$ is a \emph{primitive} (\resp \emph{strongly primitive}) subsystem (of $f$ with respect to $\mathcal{C}$) if there exists an integer $n_{F} \in \n$ such that for each pair of $\juxtapose{\colour}{\colour'} \in \colours$ and each integer $n \geqslant n_{F}$, there exists $X^n \in \cFTile{n}$ satisfying $X^n \subseteq X^0_{\colour'}$ (\resp $X^n \subseteq \inte[\big]{X^0_{\colour'}}$). 
\end{definition}


\begin{remark}\label{rem:expanding Thurston map is strongly primitive subsystem of itself}
    By \cite[Lemma~5.10]{li2018equilibrium}, every expanding Thurston map $f$ is a strongly primitive subsystem of itself with respect to every Jordan curve $\mathcal{C} \subseteq S^2$ satisfying $\post{f} \subseteq \mathcal{C}$.
\end{remark}

We record \cite[Lemmas~5.22]{shi2023thermodynamic} here, which shows that primitive subsystems have nice combinatorial and topological properties.

\begin{lemma}[Z.~Li, X.~Shi, Y.~Zhang \cite{shi2023thermodynamic}]    \label{lem:strongly primitive:tile in interior tile for high enough level}
    Let $f \colon S^2 \mapping S^2$ be an expanding Thurston map with a Jordan curve $\mathcal{C}\subseteq S^2$ satisfying $\post{f} \subseteq \mathcal{C}$.
    Let $F \in \subsystem$ be primitive (\resp strongly primitive).
    Let $n_F \in \n$ be the constant from Definition~\ref{def:primitivity of subsystem}, which depends only on $F$ and $\mathcal{C}$. 
    Then for each $n \in \n$ with $n \geqslant n_{F}$, each $m \in \n_0$, each $\colour \in \colours$, and each $m$-tile $X^m \in \Domain{m}$, there exists an $(n + m)$-tile $X^{n + m}_{\colour} \in \cFTile{n + m}$ such that $X^{n + m}_{\colour} \subseteq X^m$ (\resp $X^{n + m}_{\colour} \subseteq \inte{X^m}$).
\end{lemma}

We now review some concepts and results on ergodic theory of subsystems of expanding Thurston maps. 
We refer the reader to \cite[Section~6]{shi2023thermodynamic} for more details and the proofs. 

\smallskip

We first recall the topological pressure for subsystems.

\begin{definition}[Topological pressure]    \label{def:pressure for subsystem}
    Let $f \colon S^2 \mapping S^2$ be an expanding Thurston map with a Jordan curve $\mathcal{C}\subseteq S^2$ satisfying $\post{f} \subseteq \mathcal{C}$.
    Consider $F \in \subsystem$.
    For a real-valued function $\varphi \colon S^2 \mapping \real$, we denote \[
        Z_{n}(F, \varphi ) \define \sum_{X^n \in \Domain{n}} \myexp[\big]{ \sup\bigl\{S^F_n \varphi(x) \describe x \in X^n \bigr\} } 
    \]
    for each $n \in \n$. We define the \emph{topological pressure} of $F$ with respect to the \emph{potential} $\varphi$ by
    \begin{equation}    \label{eq:pressure of subsystem}
        \pressure[\varphi] \define \liminf_{n \mapping +\infty} \frac{1}{n} \log ( Z_{n}(F, \varphi) ).
    \end{equation}
    In particular, when $\varphi$ is the constant function $0$, the quantity $h_{\operatorname{top}}(F) \define \pressure[0]$ is called the \emph{topological entropy} of $F$.
\end{definition}

We record \cite[Proposition~6.5]{shi2023thermodynamic} below, which shows that the topological entropy of $F$ can in fact be computed explicitly via tile matrices defined in Definition~\ref{def:tile matrix}.

\begin{proposition}[Z.~Li, X.~Shi, Y.~Zhang \cite{shi2023thermodynamic}]    \label{prop:topological entropy of subsystem}
    Let $f \colon S^2 \mapping S^2$ be an expanding Thurston map with a Jordan curve $\mathcal{C}\subseteq S^2$ satisfying $\post{f} \subseteq \mathcal{C}$ and $f(\mathcal{C}) \subseteq \mathcal{C}$.
    Consider a subsystem $F \in \subsystem$.
    Let $A$ be the tile matrix of $F$ with respect to $\mathcal{C}$.
    Then we have
    \begin{equation}    \label{eq:topological entropy of subsystem}
         h_{\operatorname{top}}(F) = \log(\rho(A)),
    \end{equation}
    where $\rho(A)$ is the spectral radius of $A$.
\end{proposition}
\begin{rmk}
    The spectral radius $\rho(A)$ can easily be computed from any matrix norm. If for an $(m \times m)$-matrix $B = (b_{ij})$ we set $\norm{B} \define \sum_{i, j = 1}^{m} |b_{ij}|$ for example, then $\rho(A) = \lim_{n \to +\infty} (\norm{A^n})^{1/n}$.
\end{rmk}

We summarize the existence and some basic properties of equilibrium states for strongly primitive subsystems in the following theorem, which is part of \cite[Theorem~1.1]{shi2024uniqueness}.
Recall that $F(\limitset) \subseteq \limitset$ by Proposition~\ref{prop:subsystem:preliminary properties}~\ref{item:subsystem:properties:limitset forward invariant}.

\begin{theorem}[Z.~Li \& X.~Shi \cite{shi2024uniqueness}]    \label{thm:existence uniqueness and properties of equilibrium state}
    Let $f \colon S^2 \mapping S^2$ be an expanding Thurston map with a Jordan curve $\mathcal{C}\subseteq S^2$ satisfying $\post{f} \subseteq \mathcal{C}$ and $f(\mathcal{C}) \subseteq \mathcal{C}$.
    Let $d$ be a visual metric on $S^2$ for $f$ and $\potential$ be a real-valued \holder continuous function on $S^2$ with respect to the metric $d$.
    Consider a strongly primitive subsystem $F \in \subsystem$.
    Then there exists a unique equilibrium state $\mu_{F, \potential}$ for $F|_{\limitset}$ and $\phi|_{\limitset}$.    
    Moreover, $\mu_{F, \potential}$ is ergodic for $F|_{\limitset}$.
\end{theorem} 
\section{The Assumptions}
\label{sec:The Assumptions}

We state below the hypotheses under which we will develop our theory in most parts of this paper. 
We will selectively use some of those assumptions in the later sections. 

\begin{assumptions}
\quad
    \begin{enumerate}[label=\textrm{(\arabic*)}]
        \smallskip

        \item \label{assumption:expanding Thurston map}
            $f \colon S^2 \mapping S^2$ is an expanding Thurston map.

        \smallskip

        \item \label{assumption:Jordan curve}
            $\mathcal{C} \subseteq S^2$ is a Jordan curve containing $\post{f}$ with the property that there exists an integer $n_{\mathcal{C}} \in \n$ such that $f^{n_{\mathcal{C}}}(\mathcal{C}) \subseteq \mathcal{C}$ and $f^m(\mathcal{C}) \not\subseteq \mathcal{C}$ for each $m \in \oneton[n_{\mathcal{C}} - 1]$.
        
        \smallskip

        \item \label{assumption:subsystem}
            $F \in \subsystem$ is a subsystem of $f$ with respect to $\mathcal{C}$.
        
        \smallskip

        \item \label{assumption:visual metric and expansion factor}
            $d$ is a visual metric on $S^2$ for $f$ with expansion factor $\Lambda > 1$. 

        \smallskip

        \item \label{assumption:holder exponent}
        $\holderexp \in (0, 1]$.

        \smallskip

        \item \label{assumption:holder potential}
        $\potential \in C^{0,\holderexp}(S^2, d)$ is a real-valued H\"{o}lder continuous function with exponent $\holderexp$.

        \smallskip

        \item \label{assumption:equilibrium state}
            $\mu_{\potential}$ is the unique equilibrium state for the map $f$ and the potential $\potential$. 

        \smallskip

        \item \label{assumption:0-edge}
            $e^0 \in \mathbf{E}^0(f,\mathcal{C})$ is a $0$-edge.
    \end{enumerate}
\end{assumptions}

    
Observe that by Lemma~\ref{lem:invariant_Jordan_curve}, for each $f$ in \ref{assumption:expanding Thurston map}, there exists at least one Jordan curve $\mathcal{C}$ that satisfies \ref{assumption:Jordan curve}. 
Since for a fixed $f$, the number $n_{\mathcal{C}}$ is uniquely determined by $\mathcal{C}$ in \ref{assumption:Jordan curve}, in the remaining part of the paper, we will say that a quantity depends on $\mathcal{C}$ even if it also depends on $n_{\mathcal{C}}$.

Recall that the expansion factor $\Lambda$ of a visual metric $d$ on $S^2$ for $f$ is uniquely determined by $d$ and $f$. We will say that a quantity depends on $f$ and $d$ if it depends on $\Lambda$.


In the discussion below, depending on the conditions we will need, we will sometimes say ``Let $f$, $\mathcal{C}$, $d$, $\potential$ satisfy the Assumptions.'', and sometimes say ``Let $f$ and $\mathcal{C}$ satisfy the Assumptions.'', etc.  

\section{Upper semi-continuity}%
\label{sec:Upper semi-continuity}

In this section we show that the entropy map of an expanding Thurston map is upper semi-continuous if and only if the map has no periodic critical points.

\begin{definition}    \label{def:entropy map}
    Let $X$ be a compact metrizable topological space and $T \colon X \mapping X$ be a continuous map.
    The \emph{entropy map} of $T$ is the map $\mu \mapsto h_{\mu}(T)$ which is defined on $\mathcal{M}(X, T)$ and has values in $[0, +\infty]$.
    Here $\mathcal{M}(X, T)$ is the set of all $T$-invariant Borel probability measures on $X$ and is equipped with the weak$^*$ topology.
    We say that the entropy map of $T$ is \emph{upper semi-continuous} if $\limsup_{n \to +\infty} h_{\mu_{n}}(T) \leqslant h_{\mu}(T)$ holds for every sequence $\sequen{\mu_{n}}$ of Borel probability measures on $X$ which converges to $\mu \in \invmea[X][T]$ in the weak$^*$ topology.
\end{definition}

The proof of the following lemma is straightforward, and we include it for the sake of completeness. 

\begin{lemma}    \label{lem:entropy of push-forward average}
    Let $X$ be a compact metrizable topological space and $T \colon X \mapping X$ be a continuous map.
    Consider arbitrary $n \in \n$ and $\mu \in \mathcal{M}(X, T^{n})$.
    Define $\nu \define \frac{1}{n} \sum_{j = 0}^{n - 1} T_{*}^{j} \mu$.
    Then $\nu \in \mathcal{M}(X, T)$ and
    \begin{equation}    \label{eq:entropy of push-forward average}
         h_{T_{*}^{j}\mu}(T^{n}) = h_{\mu}(T^{n}) = h_{\nu}(T^{n}) = n h_{\nu}(T)    \qquad  \text{for each } j \in \{0, \, 1, \, \dots, \, n - 1\}.    
    \end{equation}
    Moreover, if $\mu$ is ergodic for $T^{n}$, then $\nu$ is ergodic for $T$.
\end{lemma}
\begin{proof}
    Fix arbitrary $n \in \n$ and $\mu \in \mathcal{M}(X, T^{n})$.
    Then $T_{*}\nu = \nu \in \mathcal{M}(X, T)$ since $T_{*}^{n} \mu = \mu$.
    By \eqref{eq:measure-theoretic entropy well-behaved under iteration} and \eqref{eq:measure-theoretic entropy is affine}, we have
    \begin{equation}    \label{eq:temp:lem:entropy of push-forward average:equality for entropy}
        n h_{\nu}(T) = h_{\nu}(T^{n}) = \frac{1}{n} \sum_{j = 0}^{n - 1} h_{T_{*}^{j} \mu}(T^{n}).
    \end{equation}

    We now show that $h_{T_{*}^{j} \mu}(T^{n}) = h_{\mu}(T^{n})$ for each $j \in \zeroton[n - 1]$. 
    Indeed, the measure $T_{*} \mu$ is $T^{n}$-invariant and the triple $(X, T^{n}, T_{*}\mu)$ is a factor of $(X, T^{n}, \mu)$ by the map $T$. 
    It follows that $h_{T_{*}\mu}(T^{n}) \leqslant h_{\mu}(T^{n})$ (see Subsection~\ref{sub:thermodynamic formalism}). 
    Iterating this and noting that $T_{*}^{n} \mu = \mu$ by $T^{n}$-invariance of $\mu$, we obtain\[
        h_{\mu}(T^{n}) = h_{T_{*}^{n} \mu}(T^{n}) 
        \leqslant h_{T_{*}^{n - 1} \mu}(T^{n}) \leqslant \cdots 
        \leqslant h_{T_{*} \mu}(T^{n}) \leqslant h_{\mu}(T^{n}).
    \]
    Hence $h_{T_{*}^{j} \mu}(T^n) = h_{\mu}(T^{n})$ for each $j \in \zeroton[n - 1]$. Combining this with \eqref{eq:temp:lem:entropy of push-forward average:equality for entropy}, we establish \eqref{eq:entropy of push-forward average}.

    Finally, we assume that $\mu$ is ergodic for $T^{n}$.
    Let $A$ be a Borel subset of $X$ satisfying $T^{-1}(A) = A$.
    Since $T^{-n}(A) = A$ and $\mu$ is ergodic for $T^{n}$, we have $\mu(A) \in \{0, \, 1\}$.
    Then \[
        \nu(A) = \frac{1}{n} \sum_{i = 0}^{n - 1} T_{*}^{j} \mu 
        = \frac{1}{n} \sum_{i = 0}^{n - 1} \mu\bigl( T^{-i}(A) \bigr) = \mu(A) \in \{0, \, 1\}.
    \]
    This implies that $\nu$ is ergodic for $T$.
\end{proof}

The following proposition is a consequence of Lemma~\ref{lem:entropy of push-forward average}. 
\begin{proposition}    \label{prop:upper semi-continuous equivalence for iteration}
    Let $X$ be a compact metrizable topological space and $T \colon X \mapping X$ be a continuous map.
    Consider arbitrary $n \in \n$.
    Then the entropy map of $T^{n}$ is upper semi-continuous if and only if the entropy map of $T$ is upper semi-continuous.
\end{proposition}
\begin{proof}
    Fix arbitrary $n \in \n$.

    Suppose that the entropy map of $T^{n}$ is upper semi-continuous. Since $\mathcal{M}(X, T) \subseteq \mathcal{M}(X, T^{n})$, it follows immediately from \eqref{eq:measure-theoretic entropy well-behaved under iteration} that the entropy map of $T$ is also upper semi-continuous.

    For the converse direction suppose that the entropy map of $T^{n}$ is not upper semi-continuous. 
    Then there exists a $T^{n}$-invariant Borel probability measure $\mu_{0} \in \mathcal{M}(X, T^{n})$ and a sequence $\{ \mu_{k} \}_{k \in \n}$ of $T^{n}$-invariant Borel probability measures in $\mathcal{M}(X, T^{n})$ such that $\{ \mu_{k} \}_{k \in \n}$ converges to $\mu_{0}$ in the weak$^*$ topology and satisfies
    \begin{equation}    \label{eq:temp:prop:upper semi-continuous equivalence for iteration:not upper semi-continuous limsup of entropy}
        \limsup\limits_{k \to +\infty} h_{\mu_{k}}(T^{n}) > h_{\mu_{0}}(T^{n}).
    \end{equation}
    
    We define $\nu_{0} \define \frac{1}{n} \sum_{j = 0}^{n - 1} T_{*}^{j} \mu_{0}$ and $\nu_{k} \define \frac{1}{n} \sum_{j = 0}^{n - 1} T_{*}^{j} \mu_{k}$ for each $k \in \n$.
    Then $\{ \nu_{k} \}_{k \in \n}$ converges to $\nu_{0}$ in the weak$^*$ topology. 
    Indeed, for each $\varphi \in C(X)$, since $S_{n}^{T} \varphi \in C(X)$ and $\{ \mu_{k} \}_{k \in \n}$ converges to $\mu_{0}$ in the weak$^*$ topology, we obtain
    \[
        \int \! \varphi \,\mathrm{d} \nu_{k} = \frac{1}{n} \int \! S_{n}^{T} \varphi \,\mathrm{d} \mu_{k} \longrightarrow \frac{1}{n} \int \! S_{n}^{T} \varphi \,\mathrm{d} \mu_{0} = \int \! \varphi \,\mathrm{d} \nu_{0} 
    \]
    as $k \to +\infty$.
    
    Now we show that the entropy map of $T$ is not upper semi-continuous at $\nu_{0}$.
    It follows immediately from \eqref{eq:temp:prop:upper semi-continuous equivalence for iteration:not upper semi-continuous limsup of entropy} and Lemma~\ref{lem:entropy of push-forward average} that
    \[
        \limsup\limits_{k \to +\infty} h_{\nu_{k}}(T) = \frac{1}{n} \limsup\limits_{k \to +\infty} h_{\mu_{k}}(T^{n}) > \frac{1}{n} h_{\mu_{0}}(T^{n}) = h_{\nu_{0}}(T).
    \]
    This completes the proof.
\end{proof}

By constructing suitable subsystems, we can prove the ``only if'' part in Theorem~\ref{thm:upper semi-continuous iff no periodic critical points}.
We first establish the following proposition and then prove the general cases (Theorem~\ref{thm:not upper semi-continuous with periodic critical points}) by applying Proposition~\ref{prop:upper semi-continuous equivalence for iteration}.

\begin{proposition}    \label{prop:not upper semi-continuous fixed critical point}
    Let $f \colon S^2 \mapping S^2$ be an expanding Thurston map with a Jordan curve $\mathcal{C} \subseteq S^2$ satisfying $\post{f} \subseteq \mathcal{C}$ and $f(\mathcal{C}) \subseteq \mathcal{C}$. 
    Suppose that $f$ has a fixed critical point $p$. 
    Then there exists a sequence $\{ \mu_{n} \}_{n \in \n}$ of ergodic $f$-invariant Borel probability measures in $\mathcal{M}(S^2, f)$ such that $\{ \mu_{n} \}_{n \in \n}$ converges to $\delta_{p}$ in the weak$^*$ topology and satisfies
    \begin{equation}    \label{eq:temp:not upper semi-continuous:limsup of entropy}
        \lim\limits_{n \to +\infty} h_{\mu_{n}}(f) = \log \parentheses[\big]{ \deg_{f}(p) }  > 0 = h_{\delta_{p}}(f).
    \end{equation}
    In particular, the entropy map of $f$ is not upper semi-continuous at $\delta_{p}$.
\end{proposition}
\begin{proof}
    Let $p \in S^2$ be a critical point of $f$ that is fixed by $f$. 
    Set $k \define \deg_{f}(p)$. 
    Then $k > 1$. 
    Note that $\delta_{p} \in \mathcal{M}(S^2, f)$ and $h_{\delta_{p}}(f) = 0$, where $\delta_{p}$ is the Dirac measure supported on $\{p\}$.
    It suffices to construct a sequence $\{ \mu_{n} \}_{n \in \n}$ of ergodic $f$-invariant Borel probability measures in $\mathcal{M}(S^2, f)$ such that $\{ \mu_{n} \}_{n \in \n}$ converges to $\delta_{p}$ in the weak$^*$ topology and satisfies \eqref{eq:temp:not upper semi-continuous:limsup of entropy}.

    \smallskip

    We first give the construction of $\sequen{\mu_{n}}$.

    Fix arbitrary $n \in \n$. 
    The set of $n$-tiles of $f$ at $p$ is defined as $\neighbortile[f]{n}{}{}{p} \define \{X \in \Tile{n} \describe p \in X \}$. Then by Remark~\ref{rem:flower structure}, we have $\cflower{n}{p} = \bigcup \neighbortile[f]{n}{}{}{p}$ and $\card{\neighbortile[f]{n}{}{}{p}} = 2 (\deg_{f}(p))^{n} = 2 k^{n}$, where $\flower{n}{p}$ is defined in \eqref{eq:n-flower} and $\cflower{n}{p}$ is the closure of $\flower{n}{p}$.
    
    Since $f \in \subsystem$ is strongly primitive, by Lemma~\ref{lem:strongly primitive:tile in interior tile for high enough level}, there exists an integer $n_{f} \in \n$ depending only on $f$ and $\mathcal{C}$ such that for each $n$-tile $X^{n} \in \Tile{n}$, there exists a black $(n + n_{f})$-tile $X^{n + n_{f}}_{\black} \in \cTile{n + n_{f}}{\black}$ and a white $(n + n_{f})$-tile $X^{n + n_{f}}_{\white} \in \cTile{n + n_{f}}{\white}$ such that $X^{n + n_{f}}_{\black} \cup X^{n + n_{f}}_{\white} \subseteq \inte{X^n}$.
    We denote by $\mathbf{E}_{n}$ the set consisting of two such $(n + n_{f})$-tiles, one black and one white, for each $n$-tile $X^n \in \neighbortile[f]{n}{}{}{p}$.
    In particular, we have $\card{\mathbf{E}_{n}} = 2 \card{\neighbortile[f]{n}{}{}{p}} = 4k^{n}$. 
    
    We set $F_{n} \define f^{n + n_{f}}|_{\bigcup \mathbf{E}_{n}}$ and $\widehat{F}_{n} \define F_n|_{\limitset_{n}}$, where $\limitset_{n} \define \limitset(F_n, \mathcal{C})$ is the tile maximal invariant set associated with $F_n$ with respect to $\mathcal{C}$.
    Note that $p \in \post{f} \subseteq \mathcal{C}$.
    By Remark~\ref{rem:flower structure} and Proposition~\ref{prop:cell decomposition: invariant Jordan curve}, $F_n \in \subsystem[f^{n + n_{f}}]$ is strongly primitive.
    Then it follows from Theorem~\ref{thm:existence uniqueness and properties of equilibrium state} and \cite[Theorem~1.1]{shi2023thermodynamic} that there exists $\widehat{\mu}_{n} \in \mathcal{M}(\limitset_{n}, \widehat{F}_{n}) \subseteq \mathcal{M}(S^2, f^{n + n_{f}})$ such that $\supp{\widehat{\mu}_{n}} \subseteq \limitset_{n} \subseteq \bigcup \mathbf{E}_{n} \subseteq \bigcup \neighbortile[f]{n}{}{}{p} = \cflower{n}{p}$ and
    \[
        h_{\widehat{\mu}_{n}}(f^{n + n_{f}}) = h_{\widehat{\mu}_{n}}(\widehat{F}_{n}) = P(F_n, 0) = h_{\operatorname{top}}(F_{n}),
    \]
    where $P(F_n, 0)$ and $h_{\operatorname{top}}(F_{n})$ are defined in Definition~\ref{def:pressure for subsystem}.
    Put \[
        \mu_{n} \define \frac{1}{n + n_{f}} \sum_{j = 0}^{n + n_{f} - 1} f_{*}^{j} \widehat{\mu}_{n}.
    \]
    Applying Lemma~\ref{lem:entropy of push-forward average}, we have $\mu_{n} \in \mathcal{M}(S^{2}, f)$ and $(n + n_{f}) h_{\mu_{n}}(f) = h_{\widehat{\mu}_{n}}(f^{n + n_{f}}) = h_{\operatorname{top}}(F_{n})$.
    
    \smallskip

    Now we calculate $h_{\mu_{n}}(f)$ for each $n \in \n$ and show that \eqref{eq:temp:not upper semi-continuous:limsup of entropy} holds. 
    
    By Definition~\ref{def:pressure for subsystem} and Proposition~\ref{prop:topological entropy of subsystem}, we have $h_{\operatorname{top}}(F_{n}) = \log(\rho(A_n))$, where $A_n$ is the tile matrix of $F_n$ with respect to $\mathcal{C}$ and $\rho(A_{n})$ is the spectral radius of $A_{n}$.
    Recall from Definition~\ref{def:tile matrix} and Remark~\ref{rem:tile matrix associated with set of tiles} that \[
        A_n = A(\mathbf{E}_{n}) = \begin{bmatrix}
            N_{\white \white} & N_{\black \white} \\
            N_{\white \black} & N_{\black \black}
        \end{bmatrix},
    \]
    where $N_{\colour \colour'} \define \card[\big]{ \set[\big]{ X \in \mathbf{E}_{n} \describe X \in \cTile{n + n_{f}}{\colour}, \, X \subseteq X^0_{\colour'} } }$ for each pair of $\juxtapose{\colour}{\colour'} \in \colours$. 
    In particular, since $f(\mathcal{C}) \subseteq \mathcal{C}$, by the construction of $\mathbf{E}_{n}$ and Proposition~\ref{prop:subsystem:preliminary properties}~\ref{item:subsystem:properties invariant Jordan curve:relation between color and location of tile}, one has $N_{\black \colour'} = N_{\white \colour'}$ and $N_{\colour \black} + N_{\colour \white} = \card{\mathbf{E}_{n}} / 2 = 2k^{n}$ for each pair of $\juxtapose{\colour}{\colour'} \in \colours$. 
    Then it follows that $\rho(A_{n}) = 2 k^{n}$. 
    Hence $h_{\operatorname{top}}(F_{n}) = \log(2 k^{n})$ and \eqref{eq:temp:not upper semi-continuous:limsup of entropy} holds since\[
        h_{\mu_{n}}(f) = \frac{ h_{\operatorname{top}}(F_{n}) }{n + n_{f}} = \frac{\log(2 k^{n})}{n + n_{f}} \converge \log k \qquad \text{as } n \to +\infty.
    \]
    
    \smallskip

    Finally, we show that $\{ \mu_{n} \}_{n \in \n}$ converges to $\delta_{p}$ in the weak$^*$ topology.
    It suffices to show that for each $\varphi \in C(S^2)$, $\int \! \varphi \,\mathrm{d}\mu_{n} \to \varphi(p)$ as $n \to +\infty$.
    
    We fix a visual metric $d$ that satisfies the Assumptions in Section~\ref{sec:The Assumptions}.

    Fix arbitrary $\varphi \in C(S^{2})$ and $\varepsilon > 0$. 
    Since $\varphi$ is continuous at $p$, there exists a number $\delta > 0$ such that for each $x \in S^2$ with $d(x, p) < \delta$, we have $|\varphi(x) - \varphi(p)| < \varepsilon$.
    By Lemma~\ref{lem:visual_metric}~\ref{item:lem:visual_metric:diameter of cell}, there exists an integer $N \in \n$ such that for each integer $\ell > N$, $\cflower{\ell}{p} \subseteq B_{d}(p, \delta)$.
    For each $n \in \n$ and each $j \in \{0, \, 1, \, \dots, \, n - 1\}$, it follows from $\supp{\widehat{\mu}_{n}} \subseteq \cflower{n}{p}$ and \eqref{eq:flower preserve under iterate} in Remark~\ref{rem:flower preserve under iterate} that $\supp{f_{*}^{j}\widehat{\mu}_{n}} \subseteq f^{j}\bigl( \cflower{n}{p} \bigr) = \cflower{n - j}{p}$.
    Thus for sufficiently large $n$, we have $\cflower{n - j}{p} \subseteq B_{d}(p, \delta)$ for all $j \in \{0, \, 1, \, \dots, \, n - N - 1\}$. 
    Then
    \begin{align*}
        \Big|\varphi(p) - \int \! \varphi \,\mathrm{d}\mu_{n} \Big| 
        &= \bigg|\varphi(p) - \frac{1}{n + n_{f}} \sum_{j = 0}^{n + n_{f}} \int \! \varphi \,\mathrm{d} f_{*}^{j}\widehat{\mu}_{n} \bigg|  \\
        &\leqslant \bigg|\varphi(p) - \frac{1}{n + n_{f}} \sum_{j = 0}^{n - N - 1} \int \! \varphi \,\mathrm{d} f_{*}^{j}\widehat{\mu}_{n} \bigg|   +   \frac{N + n _{f}}{n + n_{f}} \uniformnorm{\varphi}  \\
        &\leqslant \frac{1}{n+ n_{f}} \sum_{j = 0}^{n - N - 1} \int \! | \varphi - \varphi(p) | \,\mathrm{d} f_{*}^{j}\widehat{\mu}_{n}   +   \frac{2(N + n _{f})}{n + n_{f}} \uniformnorm{\varphi}  \\
        &\leqslant \frac{n - N}{n + n _{f}} \varepsilon  +  \frac{2(N + n _{f})}{n + n_{f}} \uniformnorm{\varphi} \\
        &\leqslant 2 \varepsilon
    \end{align*}
    for sufficiently large $n$.
    This implies $\int \! \varphi \,\mathrm{d}\mu_{n} \to \varphi(p)$ as $n \to +\infty$ for each $\varphi \in C(S^2)$.

    The proof is complete.
\end{proof}

Recall that a point $x \in S^{2}$ is a periodic point of $f \colon S^{2} \mapping S^{2}$ with period $n \in \n$ if $f^{n}(x) = x$ and $f^{i}(x) \ne x$ for each $i \in \{1, \, 2, \, \dots, \, n - 1\}$.

\begin{theorem}    \label{thm:not upper semi-continuous with periodic critical points}
    Let $f \colon S^2 \mapping S^2$ be an expanding Thurston map.
    Suppose that $f$ has a periodic critical point $p$ with period $n$ for some $n \in \n$. 
    Denote $\deltameasure{p} \define \frac{1}{n} \sum_{i = 0}^{n - 1} \delta_{f^{i}(p)}$.
    Then there exists a sequence $\{ \nu_{k} \}_{k \in \n}$ of ergodic $f$-invariant Borel probability measures in $\mathcal{M}(S^2, f)$ such that $\{ \nu_{k} \}_{k \in \n}$ converges to $\deltameasure{p}$ in the weak$^*$ topology and satisfies
    \begin{equation}    \label{eq:thm:not upper semi-continuous with periodic critical points:not upper semi-continuous:limsup of entropy}
        \lim\limits_{k \to +\infty} h_{\nu_{k}}(f) = \frac{1}{n} \log \parentheses[\big]{ \deg_{f^{n}}(p) } 
        > 0 = h_{\deltameasure{p}}(f).
    \end{equation}
    In particular, the entropy map of $f$ is not upper semi-continuous at $\deltameasure{p}$.
\end{theorem}
\begin{proof}
    Suppose that $p \in S^{2}$ is a periodic critical point of $f$ with period $n$ for some $n \in \n$.
    
    By Lemma~\ref{lem:invariant_Jordan_curve}, we can find a sufficiently high iterate $f^{m}$ of $f$ that has an $f^{m}$-invariant Jordan curve $\mathcal{C} \subseteq S^{2}$ with $\post{f^{nm}} = \post{f} \subseteq \mathcal{C}$. 
    We fix such Jordan curve $\mathcal{C}$ and set $F \define f^{N}$ with $N \define nm$. 
    Thus $F(\mathcal{C}) \subseteq \mathcal{C}$ and $p$ is a fixed critical point of $F$ with $\deg_{F}(p) = \parentheses[\big]{ \deg_{f^{n}}(p) } ^{m}$.
    Note that $F$ is also an expanding Thurston map by Remark~\ref{rem:Expansion_is_independent}.

    By Proposition~\ref{prop:not upper semi-continuous fixed critical point}, there exists a sequence $\{ \mu_{k} \}_{k \in \n}$ of ergodic $F$-invariant Borel probability measures in $\mathcal{M}(S^2, F)$ such that $\{ \mu_{k} \}_{k \in \n}$ converges to $\delta_{p}$ in the weak$^*$ topology and satisfies
    \begin{equation}    \label{eq:temp:thm:not upper semi-continuous with periodic critical points:not upper semi-continuous:limsup of entropy}
        \lim\limits_{k \to +\infty} h_{\mu_{k}}(F) = \log \parentheses[\big]{ \deg_{F}(p) }  > h_{\delta_{p}}(F) = 0.
    \end{equation}

    We define $\nu_{k} \define \frac{1}{N} \sum_{j = 0}^{N - 1} f_{*}^{j} \mu_{k}$ for each $k \in \n$.
    It follows immediately from Lemma~\ref{lem:entropy of push-forward average} that $\nu_{k} \in \invmea[S^2][f]$ and $\nu_{k}$ is ergodic for $f$.
    Note that $\{ \nu_{k} \}_{k \in \n}$ converges to $\deltameasure{p}$ in the weak$^*$ topology. 
    Indeed, for each $\varphi \in C(X)$, since $S_{N}^{f} \varphi \in C(X)$ and $\{ \mu_{k} \}_{k \in \n}$ converges to $\delta_{p}$ in the weak$^*$ topology, we have
    \[
        \int \! \varphi \,\mathrm{d} \nu_{k} 
        = \frac{1}{N} \int \! S_{N}^{f} \varphi \,\mathrm{d} \mu_{k} \longrightarrow \frac{1}{N} \int \! S_{N}^{f} \varphi \,\mathrm{d} \delta_{p}
        = \int \! \varphi \,\mathrm{d} \deltameasure{p}
    \]
    as $k \to +\infty$.
    
    Finally, by \eqref{eq:temp:thm:not upper semi-continuous with periodic critical points:not upper semi-continuous:limsup of entropy} and Lemma~\ref{lem:entropy of push-forward average} we have
    \[
        \begin{split}
            \lim\limits_{k \to +\infty} h_{\nu_{k}}(f) 
            &= \frac{1}{N} \lim\limits_{k \to +\infty} h_{\mu_{k}}(F) = \frac{1}{N} \log \parentheses[\big]{ \deg_{F}(p) } = \frac{1}{n} \log \parentheses[\big]{ \deg_{f^{n}}(p) } \\
            &> \frac{1}{N} h_{\delta_{p}}(F) = h_{\deltameasure{p}}(f) = 0.
        \end{split}
    \]
    This completes the proof.
\end{proof}

\begin{proof}[Proof of Theorem~\ref{thm:upper semi-continuous iff no periodic critical points}]
    By \cite[Corollary~1.3]{li2015weak}, if $f$ has no periodic critical points, then the entropy map of $f$ is upper semi-continuous.
    The other direction follows immediately from Theorem~\ref{thm:not upper semi-continuous with periodic critical points}.
\end{proof}


\section{Entropy density}
\label{sec:Entropy density}

This section is devoted to the proof of entropy density of ergodic measures for expanding Thurston maps, with the main result being Theorem~\ref{thm:entropy dense}.



We first introduce some notations.

\noindent\textbf{Notations.} 
For $\ell \in \n$ we define\[
    \multispace \define \{ \vecfun = \mutifun \describe \varphi_{j} \in C(S^{2}) \text{ for each } j \in \oneton[\ell] \}.
\]
For $\vecfun = \mutifun \in \multispace$, $\vecavg = \mutiavg \in \real^{\ell}$, and $\mu \in \probsphere$, the expression $\int \! \vecfun \,\mathrm{d}\mu > \vecavg$ indicates that $\int \! \varphi_{j} \,\mathrm{d}\mu > \alpha_{j}$ holds for each $j \in \oneton[\ell]$.
The meaning of $\int \! \vecfun \,\mathrm{d}\mu \geqslant \vecavg$ is analogous.
We write $S_{n}\vecfun \define \mutifun[\ell][S_{n}\varphi] \in \multispace$. 
Put $\norm{\vecavg} \define \max_{1 \leqslant j \leqslant \ell} |\alpha_{j}|$ and $\norm{\vecfun} \define \max_{1 \leqslant j \leqslant \ell} \uniformnorm{\varphi_{j}}$.
For $\varepsilon \in \real$ we use the convention that $\vecavg + \varepsilon \define ( \alpha_{1} + \varepsilon, \, \dots, \, \alpha_{\ell} + \varepsilon ) \in \real^{\ell}$.

Let $f \colon S^2 \mapping S^2$ be an expanding Thurston map and $\mathcal{C} \subseteq S^2$ be a Jordan curve containing $\post{f}$. 
For a real-valued function $\psi : S^2 \mapping \real$ and an integer $n\in \n$ define
\begin{equation}    \label{eq:def:distortion of potential on tiles}
    D_n(\psi) = D_{n}^{f, \, \mathcal{C}}(\psi) \define \sup_{X^n \in \mathbf{X}^n(f, \mathcal{C})} \sup_{\juxtapose{x}{y} \in X^n} | S_n\psi(x) - S_n\psi(y) |.
\end{equation}
Note that $D_{n}(\psi) \leqslant n D_{1}(\psi)$ holds for every $n \in \n$. 
For $\ell \in \n$ and $\vecfun \in \multispace$, we write $D_{n}(\vecfun) \define \max_{1 \leqslant j \leqslant \ell} D_{n}(\varphi_{j})$ for each $n \in \n$.

Indeed, for $\varphi \in C(S^{2})$ we have $\lim_{n \to +\infty} D_{n}(\varphi) / n = 0$.

\begin{lemma} \label{lem:distortion lemma for continuous function}
    Let $f$ and $\mathcal{C}$ satisfy the Assumptions in Section~\ref{sec:The Assumptions}.
    Then \[
        \lim_{n \to +\infty} \frac{1}{n} D_{n}(\varphi) = 0
    \]
    for each $\varphi \in C(S^{2})$.
\end{lemma}
\begin{proof}
    We fix a visual metric $d$ that satisfies the Assumptions in Section~\ref{sec:The Assumptions}.

    Fix arbitrary $\varphi \in C(S^{2})$ and $\varepsilon > 0$. 
    Since $\varphi$ is uniformly continuous on $S^{2}$, there exists a number $\delta > 0$ such that for each pair of $p, \, q \in S^2$ that satisfy $d(p, q) < \delta$, we have $|\varphi(p) - \varphi(q)| < \varepsilon$.
    By Lemma~\ref{lem:visual_metric}~\ref{item:lem:visual_metric:diameter of cell}, we have
    \begin{equation}    \label{eq:temp:lem:distortion lemma for continuous function:diam bound for tile}
        \diam[\big]{d}{X^{k}} \leqslant C \Lambda^{-k} \qquad \text{for all } k \in \n_0 \text{ and } X^{k} \in \Tile{k},
    \end{equation}
    where $C \geqslant 1$ is the constant from Lemma~\ref{lem:visual_metric}.  
    This implies that there exists $N \in \n$ such that $\diam[\big]{d}{X^{k}} < \delta$ for all integer $k > N$ and $X^{k} \in \Tile{k}$.
    Then for each sufficiently large $n \in \n$, each $X^{n} \in \Tile{n}$, and each pair of $\juxtapose{x}{y} \in X^{n}$, by Proposition~\ref{prop:properties cell decompositions}~\ref{item:prop:properties cell decompositions:cellular}, we have\[
        \begin{split}
            | S_{n}\varphi(x) - S_{n}\varphi(y) | &\leqslant \sum^{n - N - 1}_{k = 0} | \varphi(f^{i}(x)) - \varphi(f^{i}(y)) | + \sum^{n - 1}_{k = n - N} | \varphi(f^{i}(x)) - \varphi(f^{i}(y)) | \\
            &\leqslant (n - N) \varepsilon + N \uniformnorm{\varphi}.
        \end{split}
    \]
    It follows that\[
        \limsup_{n \to +\infty} \frac{1}{n} D_{n}(\varphi) \leqslant \lim_{n \to +\infty}  \frac{n - N}{n}\varepsilon + \frac{N \uniformnorm{\varphi}}{n} = \varepsilon.
    \]
    Since $\varepsilon > 0$ is arbitrary, the proof is complete.
\end{proof}

Let $f$ and $\mathcal{C}$ satisfy the Assumptions in Section~\ref{sec:The Assumptions}.
We assume in addition that $f(\mathcal{C}) \subseteq \mathcal{C}$.
In the following, we set
\begin{equation}    \label{eq:def:set of all edges}
    \alledge \define \bigcup_{n \in \n_{0}} f^{-n}(\mathcal{C}).
\end{equation}
Then $\alledge$ is a Borel set. Proposition~\ref{prop:properties cell decompositions}~\ref{item:prop:properties cell decompositions:skeletons} implies that $\alledge$ is equal to the union of all edges.
Since every vertex is contained in an edge, the set $\alledge$ also contains all vertices.
Moreover, we have
\begin{equation}    \label{eq:set of all edges is f-invariant}
    f^{-1}(\alledge) = \alledge.
\end{equation}
Indeed, note that $\mathcal{C} \subseteq f^{-1}(\mathcal{C})$ and so\[
    \begin{split}
        f^{-1}(\alledge) &= f^{-1} \biggl( \bigcup_{n \in \n_{0}} f^{-n}(\mathcal{C}) \biggr) = \bigcup_{n \in \n_{0}} f^{-(n + 1)}(\mathcal{C})  \\
        &= \bigcup_{n \in \n} f^{-n}\mathcal{C} = \mathcal{C} \cup \bigcup_{n \in \n} f^{-n}\mathcal{C} = \alledge.
    \end{split}
\]

The next lemma approximates ergodic measures with a finite collection of tiles in a particular sense.

\begin{lemma} \label{lem:approximates ergodic measures by tiles}    \def \approximation #1{ \mathbf{T}^{#1} }
    Let $f$ and $\mathcal{C}$ satisfy the Assumptions in Section~\ref{sec:The Assumptions}.
    We assume in addition that $f(\mathcal{C}) \subseteq \mathcal{C}$.
    Consider $\ell \in \n$ and $\vecfun \in \multispace$.
    Then for each ergodic $f$-invariant measure $\mu \in \probsphere$ and each $\varepsilon > 0$, there exists $n_0 \in \n$ such that for each integer $n \geqslant n_{0}$, there exists a non-empty subset $\approximation{n}$ of $\Tile{n}$ such that 
    \begin{align}
        \label{eq:lem:approximates ergodic measures by tiles:cardinality}
        \bigg| \frac{1}{n} \log \card{\approximation{n}} - h_{\mu}(f) \bigg| &\leqslant \varepsilon \qquad \text{and}  \\
        \label{eq:lem:approximates ergodic measures by tiles:Birkhoff average}
        \sup_{x \in \bigcup \approximation{n}} \norm[\bigg]{ \frac{1}{n} S_n \vecfun(x) - \int \! \vecfun \,\mathrm{d}\mu} &\leqslant \varepsilon.
    \end{align}
\end{lemma}

To prove Lemma~\ref{lem:approximates ergodic measures by tiles}, we need the following lemma, which is a generalization of \cite[Lemma~17.7]{bonk2017expanding}. 
\begin{lemma} \label{lem:generator}
    Let $f$ and $\mathcal{C}$ satisfy the Assumptions in Section~\ref{sec:The Assumptions}.
    We assume in addition that $f(\mathcal{C}) \subseteq \mathcal{C}$.
    Consider an $f$-invariant Borel probability measure $\mu \in \probmea{S^{2}}$.
    Then the following statements hold:
    \begin{enumerate}[label=\rm{(\roman*)}]
        \smallskip

        \item     \label{item:lem:generator:tiles}
            Suppose that $\mu(\alledge) = 0$.
            Then for each $n \in \n$, the set $\Tile{n}$ forms a measurable partition for $(S^{2}, \mu)$ and is equivalent to the partition $\xi^{n}_{f}$ where $\xi = \Tile{1}$.
            Moreover, $\xi = \Tile{1}$ is a generator for $(f, \mu)$.

        \smallskip

        \item     \label{item:lem:generator:edges}
            Suppose that $\mu(\alledge) = 1$ and $\mu$ is non-atomic.
            Then $\mu(\mathcal{C}) = 1$.
            Additionally, for each $n \in \n$, the set $\Edge{n}$ forms a measurable partition for $(S^{2}, \mu)$ and is equivalent to the partition $\eta^{n}_{f}$ where $\eta = \Edge{1}$.
            Moreover, $\eta = \Edge{1}$ is a generator for $(f, \mu)$.

        \smallskip

        \item     \label{item:lem:generator:periodic orbit}
            Suppose that there exists a point $x \in S^{2}$ such that $\mu(\{x\}) > 0$.
            Then $x$ is a periodic point of $f$.
            Moreover, if we assume in addition that $\mu$ is ergodic, then $\mu = \frac{1}{n} \sum_{i = 0}^{n - 1} \delta_{f^{i}(x)}$, where $n$ is the period of $x$.
     \end{enumerate}
\end{lemma}
\begin{proof}
    \ref{item:lem:generator:tiles}
    Statement~\ref{item:lem:generator:tiles} was established in \cite[Lemma~17.7]{bonk2017expanding}.

    \smallskip

    \ref{item:lem:generator:edges}      
    We first note that it follows immediately from our assumptions that $\mu(\mathcal{C}) = 1$.
    Indeed, since $\mu$ is $f$-invariant, we have $\mu(f^{-n}(\mathcal{C})) = \mu(\mathcal{C})$ for all $n \in \n_{0}$.
    On the other hand, $\mathcal{C} \subseteq f^{-n}(\mathcal{C})$, and so $\mu(f^{-n}(\mathcal{C}) \setminus \mathcal{C}) = 0$. 
    This implies that $\mu(\alledge \setminus \mathcal{C}) = 0$. So $\mu$ is actually concentrated on $\mathcal{C}$.

    Since $\mu$ is non-atomic, we have $\mu(v) = 0$ for all vertices $v \in \bigcup_{n \in \n_{0}} \Vertex{n}$.
    For each $n \in \n$, since $f(\mathcal{C}) \subseteq \mathcal{C} \subseteq f^{-n}(\mathcal{C})$, the set $\bigcup \Edge{n}$ has full measure, and two distinct $n$-edges have only vertices, i.e., a set of $\mu$-measure zero, in common.
    Hence $\Edge{n}$ is a measurable partition for $(S^{2}, \mu)$.

    Fix arbitrary $n \in \n$ and $e \in \Edge{n}$.
    We now show that $\Edge{n}$ is equivalent to the partition $\eta^{n}_{f}$, where $\eta = \Edge{1}$.
    For $i = 1, \, \dots, \, n$ there exist unique $i$-edges $e^{i} \in \Edge{i}$ with $e = e^{n} \subseteq e^{n - 1} \subseteq \cdots \subseteq e^{1}$.
    Set $e_{i} \define f^{i - 1}(e^{i})$ for $i = 1, \, \dots, \, n$. Then $e_{1}, \, \dots, \, e_{n}$ are $1$-edges.
    We claim that
    \begin{equation}    \label{eq:temp:item:lem:generator:edges:edge}
        e = e_{1} \cap f^{-1}(e_{2}) \cap \cdots \cap f^{-(n - 1)}(e_{n}).
    \end{equation}
    To see this, denote the right hand side in this equation by $\widetilde{e}$. Then it is clear that $e \subseteq \widetilde{e}$.
    We verify $e = \widetilde{e}$ by inductively showing that for any point $x \in \widetilde{e}$ we have $x \in e^{i}$ for $i = 1, \, \dots, \, n$, and so $x \in e^{n} = e$.

    Indeed, since $\widetilde{e} \subseteq e_{1} = e^{1}$ this is clear for $i = 1$.
    Suppose $x \in e^{i - 1}$ for some $i \in \n$ with $2 \leqslant i \leqslant n$. 
    To complete the inductive step, we have to show $x \in e^{i}$.
    Note that $x \in \widetilde{e} \subseteq f^{-(i - 1)}(e_{i})$ and so $f^{i - 1}(x) \in e_{i}$.
    By Proposition~\ref{prop:properties cell decompositions}~\ref{item:prop:properties cell decompositions:cellular}, the map $f^{i - 1}|_{e^{i - 1}}$ is a homeomorphism of $e^{i - 1}$ onto the $0$-edge $f^{i - 1}(e^{i - 1})$.
    Moreover, $x \in e^{i - 1}$, $e^{i} \subseteq e^{i - 1}$, and $f^{i - 1}(x) \in e_{i} = f^{i - 1}(e^{i})$.
    Hence by injectivity of $f^{i - 1}$ on $e^{i - 1}$ we have $x \in e^{i}$ as desired.

    Equation~\eqref{eq:temp:item:lem:generator:edges:edge} shows that every element in $\Edge{n}$ belongs to $\eta^{n}_{f}$, where $\eta = \Edge{1}$.
    This implies that the measurable partitions $\Edge{n}$ and $\eta^{n}_{f}$ are equivalent ($\eta^{n}_{f}$ may contain additional sets, but they have to be of measure zero).

    To establish that $\eta = \Edge{1}$ is a generator, let $B \subseteq S^2$ be an arbitrary Borel set and $\varepsilon > 0$.
    Since the measurable partitions $\Edge{n}$ and $\eta^{n}_{f}$ are equivalent for each $n \in \n$, it suffices to show that there exists $k \in \n$ and a union $A$ of $k$-edges such that $\mu(A \bigtriangleup B) < \varepsilon$.

    By regularity of $\mu$ there exists a compact set $K \subseteq B$ and an open set $U \subseteq S^2$ with $K \subseteq B \subseteq U$ and $\mu(U \setminus K) < \varepsilon$.
    Since the diameters of edges approach $0$ uniformly as their levels become larger, we can choose $k \in \n$ large enough such that every $k$-edge that meets $K$ is contained in the open neighborhood $U$ of $K$.
    Define $K_{\mathcal{C}} \define K \cap \mathcal{C}$ and \[
        A \define \bigcup \bigl\{ e^{k} \in \Edge{k} \describe e^{k} \cap K_{\mathcal{C}} \ne \emptyset \bigr\}.
    \]
    Then $K_{\mathcal{C}} \subseteq A \subseteq U$. This implies $A \bigtriangleup B \subseteq U \setminus K_{\mathcal{C}}$, and so\[
        \mu(A \bigtriangleup B) \leqslant \mu(U \setminus K_{\mathcal{C}}) \leqslant \mu(U \setminus K) + \mu(U \setminus \mathcal{C}) 
        = \mu(U \setminus K) < \varepsilon
    \]
    as desired. The proof of statement~\ref{item:lem:generator:edges} is complete.

    \smallskip

    \ref{item:lem:generator:periodic orbit}
    Assume first that there exists a pint $x \in S^{2}$ such that $\mu(\{x\}) > 0$.

    We claim that $x$ is preperiodic.
    Otherwise, for each pair of $\juxtapose{k}{\ell} \in \n$ with $k \ne \ell$, we have $f^{k}(x) \ne f^{\ell}(x)$. 
    We write $x_{n} \define f^{n}(x)$ for each $n \in \n$. 
    Then $x \in f^{-n}(x_{n})$ and $\mu(\{x_{n}\}) = \mu(f^{-n}(x_{n})) \geqslant \mu(\{x\})$ for each $n \in \n$ since $\mu$ is $f$-invariant.
    This implies \[
        1 = \mu(S^{2}) \geqslant \sum_{n = 1}^{+\infty} \mu(\{x_{n}\}) \geqslant \sum_{n = 1}^{+\infty} \mu(\{x\}) = +\infty,
    \]
    which is a contradiction. This proves the claim that $x$ is preperiodic.

    We now show that $x$ is periodic.
    Since $x$ is preperiodic, there exist $m \in \n_{0}$ and $n \in \n$ such that $f^{n + m}(x) = f^{m}(x)$. 
    Denote $y \define f^{nm}(x)$. Then $f^{n}(y) = f^{n + nm}(x) = f^{nm}(x) = y$, i.e., $y$ is a periodic point of $f$. 
    Since $f^{nm}(x) = f^{nm}(y) = y$ and $\mu$ is $f$-invariant, we have \[
        \mu(\{y\}) = \mu(f^{-nm}(y)) \geqslant \mu(\{x\} \cup \{y\}).
    \]
    This implies $x = y$ since $\mu(\{x\}) > 0$. Thus $x$ is periodic.

    Finally, we assume in addition that $\mu$ is ergodic. 
    Let $n \in \n$ be the period of $x$.
    We set $\operatorname{Orb}(x) \define \bigl\{ f^{k}(x) \describe k \in \n_{0} \bigr\} = \bigl\{ f^{k}(x) \describe k \in \{0, \, \dots, \, n - 1\} \bigr\}$ and \[
        x^{\infty} \define \{x\} \cup \bigcup_{k \in \n} f^{-k}(x) = \bigcup_{k \in \n_0} f^{-k}(\operatorname{Orb}(x)).
    \]
    Since $x \in f^{-n}(x)$, we have\[
        \begin{split}
            f^{-1}(x^{\infty}) &= f^{-1} \biggl( \{x\} \cup \bigcup_{k \in \n} f^{-k}(x) \biggr) = f^{-1}(x) \cup \bigcup_{k \in \n} f^{-(k + 1)}(x)  \\
            &= \bigcup_{k \in \n} f^{-k}(x) = \{x\} \cup \bigcup_{k \in \n} f^{-k}(x) = x^{\infty}.
        \end{split}
    \]
    This implies $\mu(x^{\infty}) = 1$ since $\mu$ is ergodic and $\mu(x^{\infty}) \geqslant \mu(\{x\}) > 0$.
    Since $\mu$ is $f$-invariant, we have $\mu \bigl( f^{-k}(\operatorname{Orb}(x)) \bigr) = \mu(\operatorname{Orb}(x))$ for all $k \in \n_{0}$. On the other hand, $\operatorname{Orb}(x) \subseteq f^{-k}(\operatorname{Orb}(x))$, and so $\mu\bigl( f^{-k}(\operatorname{Orb}(x)) \setminus \operatorname{Orb}(x) \bigr) = 0$. 
    Thus we have $\mu(x^{\infty} \setminus \operatorname{Orb}(x)) = 0$. 
    So $\mu$ actually concentrates on $\operatorname{Orb}(x)$.

    It suffices to show that $\mu(\{f^{k}(x)\}) = 1/n$ for each $k \in \{0, \, \dots, \, n - 1\}$. 
    Indeed, since $\mu$ is $f$-invariant and $f^{k}(x) \in f^{-1}( f^{k + 1}(x) )$, we have $\mu(\{f^{k + 1}(x)\}) = \mu( f^{-1}(f^{k + 1}(x))) \geqslant \mu(\{f^{k}(x)\})$. 
    This implies \[
        \mu(\{x\}) \leqslant \mu(\{f^{k}(x)\}) \leqslant \mu(\{f^{n}(x)\}) = \mu(\{x\}).
    \]
    Hence $\mu(\{f^{k}(x)\}) = \mu(\{x\}) = 1/n$ for each $k \in \{0, \, \dots, \, n - 1\}$.

    The proof of statement~\ref{item:lem:generator:periodic orbit} is complete.
\end{proof}

Now we can prove Lemma~\ref{lem:approximates ergodic measures by tiles}. 

\begin{proof}[Proof of Lemma~\ref{lem:approximates ergodic measures by tiles}]   \def \approximation#1{\mathbf{T}^{#1}}  \def \edgeapproximation#1{\mathbf{S}^{#1}}
    Fix arbitrary $\varepsilon > 0$, $\ell \in \n$, $\vecfun \in \multispace$, and ergodic $f$-invariant measure $\mu \in \probmea{S^{2}}$.
    It suffices to show that for every sufficiently large $n \in \n$, there exists a non-empty subset $\approximation{n}$ of $\Tile{n}$ such that \eqref{eq:lem:approximates ergodic measures by tiles:cardinality} and \eqref{eq:lem:approximates ergodic measures by tiles:Birkhoff average} hold. 
    We split the proof into three cases according to the properties of measure $\mu$. 
    Recall the definition of $\alledge$ in \eqref{eq:def:set of all edges}. 

    \smallskip
    \emph{Case 1:} $\mu(\alledge) = 0$.
    \smallskip

    Since $\mu$ is $f$-invariant and $\mu(\alledge) = 0$, by Lemma~\ref{lem:generator}~\ref{item:lem:generator:tiles}, $\Tile{1}$ is a generator for $(f, \mu)$, and for each $n \in \n$ the set $\Tile{n}$ forms a measurable partition for $(S^{2}, \mu)$ and is equivalent to the partition $\xi^{n}_{f}$ where $\xi = \Tile{1}$.
    Then one can use Birkhoff's ergodic theorem and Shannon--McMillan--Breiman's theorem to show that for every sufficiently large $n \in \n$, there exists a non-empty subset $\approximation{n}$ of $\Tile{n}$ such that \eqref{eq:lem:approximates ergodic measures by tiles:cardinality} and \eqref{eq:lem:approximates ergodic measures by tiles:Birkhoff average} hold.  

    \smallskip
    \emph{Case 2:} $\mu(\alledge) > 0$ and $\mu$ is non-atomic.
    \smallskip

    First note that $\mu(\alledge) = 1$ in this case since $\mu$ is ergodic and $f^{-1}(\alledge) = \alledge$ (see \eqref{eq:set of all edges is f-invariant}).
    Then it follows from Lemma~\ref{lem:generator}~\ref{item:lem:generator:edges} that $\Edge{1}$ is a generator for $(f, \mu)$, and for each $n \in \n$ the set $\Edge{n}$ forms a measurable partition for $(S^{2}, \mu)$ and is equivalent to the partition $\eta^{n}_{f}$ where $\eta = \Edge{1}$.
    Similar to Case~1, one can use Birkhoff's ergodic theorem and Shannon--McMillan--Breiman's theorem to show that there exists $N \in \n$ such that for each integer $n \geqslant N$, there exists a non-empty subset $\edgeapproximation{n}$ of $\Edge{n}$ such that
    \begin{align}
        \label{eq:temp:lem:approximates ergodic measures by tiles:case 2:cardinality for edge}
        \bigg| \frac{1}{n} \log \card{\edgeapproximation{n}} - h_{\mu}(f) \bigg| &\leqslant \varepsilon/2 \qquad \text{and}  \\
        \label{eq:temp:lem:approximates ergodic measures by tiles:case 2:Birkhoff average for edge}
        \sup_{x \in \bigcup \edgeapproximation{n}} \norm[\bigg]{\frac{1}{n} S_n \vecfun(x) - \int \! \vecfun \,\mathrm{d}\mu} &\leqslant \varepsilon/2.    
    \end{align}
    For each integer $n \geqslant N$, we set \[
        \approximation{n} \define \{ X^{n} \in \Tile{n} \describe e^{n} \subseteq \partial X^{n}, \, e^{n} \in \edgeapproximation{n} \}.
    \]
    Then by Proposition~\ref{prop:properties cell decompositions}~\ref{item:prop:properties cell decompositions:tile is gon} and Remark~\ref{rem:intersection of two tiles}, we have \[
        \card{\edgeapproximation{n}} / \card{\post{f}} \leqslant \card{\approximation{n}} \leqslant 2 \card{\edgeapproximation{n}}.
    \]
    Combining this with \eqref{eq:temp:lem:approximates ergodic measures by tiles:case 2:cardinality for edge}, we see that \eqref{eq:lem:approximates ergodic measures by tiles:cardinality} holds for every sufficiently large $n \in \n$.
    Noting that $\vecfun = \mutifun \in \multispace$, by Lemma~\ref{lem:distortion lemma for continuous function}, we have \[
        \sup_{X^n \in \mathbf{X}^n(f,\mathcal{C})} \sup_{\juxtapose{x}{y} \in X^n} \norm[\bigg]{ \frac{1}{n} S_n \vecfun(x) - \frac{1}{n} S_n \vecfun(y) } \longrightarrow 0 \quad \text{ as } n \to +\infty.
    \]
    Therefore, \eqref{eq:temp:lem:approximates ergodic measures by tiles:case 2:Birkhoff average for edge} implies that \eqref{eq:lem:approximates ergodic measures by tiles:Birkhoff average} holds for every sufficiently large $n \in \n$.

    \smallskip
    \emph{Case 3:} $\mu(\alledge) > 0$ and there exists $p \in S^{2}$ such that $\mu(p) > 0$.
    \smallskip

    By Lemma~\ref{lem:generator}~\ref{item:lem:generator:periodic orbit}, $p$ is a periodic point of $f$ and $\mu = \deltameasure[k]{p}$, where $k$ is the period of $p$.
    In particular, we have $h_{\mu}(f) = 0$.
    Fix arbitrary $n \in \n$. For each $i \in \oneton[k - 1]$, we choose an $n$-tile $X^{n}_{i} \in \Tile{n}$ such that $f^{i}(p) \in X^{n}_{i}$. We denote by $\approximation{n}$ the set of those $n$-tiles. 
    Then $1 \leqslant \card{\approximation{n}} \leqslant k$. This implies \eqref{eq:lem:approximates ergodic measures by tiles:cardinality} holds for every sufficiently large $n \in \n$.
    Since $\lim_{n \to +\infty} \frac{1}{n} S_n \vecfun(f^{i}(x)) = k^{-1} S_{k}\vecfun(p) = \int \! \vecfun \,\mathrm{d}\mu$ for each $i \in \oneton[k - 1]$, by Lemma~\ref{lem:distortion lemma for continuous function}, \eqref{eq:lem:approximates ergodic measures by tiles:Birkhoff average} holds for every sufficiently large $n \in \n$.

    \smallskip

    The proof is complete.
\end{proof}

Apart from Lemma~\ref{lem:approximates ergodic measures by tiles}, we need the following lemma to construct strongly primitive subsystems in the proof of Theorem~\ref{thm:entropy dense}.

\begin{lemma} \label{lem:pair in the interior}
	Let $f$, $\mathcal{C}$, $e^0$ satisfy the Assumptions in Section~\ref{sec:The Assumptions}. 
	We assume in addition that $f(\mathcal{C}) \subseteq \mathcal{C}$. 
	Consider $\ell \in \n$ and $\vecfun = \mutifun \in \multispace$.
	Let $\mu \in \invmea$ be an ergodic measure.
	Then for each $\varepsilon > 0$, there exists an integer $N \in \n$ such that for each integer $n \geqslant N$ and each color $\colour \in \colours$, there exists an $n$-pair $P^{n}_{\colour} \in \Pair{n}$ such that $P^{n}_{\colour} \subseteq \inte[\big]{X^0_{\colour}}$ and \[
		\sup_{x \in P^{n}_{\colour}} \norm[\bigg]{ \frac{1}{n} S_n \vecfun(x) - \int \! \vecfun \,\mathrm{d} \mu} \leqslant \varepsilon.
	\]
\end{lemma}
\begin{proof}
	Let $\varepsilon > 0$ be arbitrary.
	Since $\mu$ is ergodic, it follows from Birkhoff's ergodic theorem that for $\mu$-a.e.\ $x \in S^2$, $\frac{1}{n} S_n \vecfun(x) \to \int \! \vecfun \,\mathrm{d} \mu$ as $n \to +\infty$.
	Thus we can fix a point $y \in S^2$ and an integer $n_{0} \in \n$ such that for each integer $n \geqslant n_{0}$,
	\begin{equation}    \label{eq:temp:approximation of ergodic average}
		\norm[\bigg]{ \frac{1}{n} S_n \vecfun(y) - \int \! \vecfun \,\mathrm{d} \mu} \leqslant \frac{\varepsilon}{2}.
	\end{equation}

	By Lemma~\ref{lem:pair neighbor disjoint with Jordan curve}, there exists an integer $M \in \n$ such that for each color $\colour \in \colours$, there exists an $M$-pair $P_{\colour}^{M} \in \Pair{M}$ such that for each integer $n \geqslant M$ and each $x \in P^{M}_{\colour}$, we have $U^{n}(x) \subseteq \inte[\big]{X^0_{\colour}}$.

	We fix such an integer $M \in \n$ and the corresponding $M$-pairs $P^{M}_{\black}$ and $P^{M}_{\white}$ in the following.

	Let color $\colour \in \colours$ and integer $n \geqslant M + n_{0}$ be arbitrary.
	Since $n - M \geqslant n_{0}$, by \eqref{eq:temp:approximation of ergodic average}, we have\[
		\norm[\bigg]{ \frac{1}{n - M} S_{n - M} \vecfun(y) - \int \! \vecfun \,\mathrm{d} \mu} \leqslant \frac{\varepsilon}{2}.
	\] 
	Since $f^{M}(P^{M}_{\colour}) = S^2$, there exists $x_{\colour} \in P^{M}_{\colour}$ such that $f^{M}(x_{\colour}) = y$.
	Then we have\[
		\begin{split}
			\norm[\bigg]{ \frac{1}{n} S_n \vecfun(x_{\colour}) - \int \! \vecfun \,\mathrm{d} \mu} 
			&= \norm[\bigg]{ \frac{1}{n} S_M \vecfun(x_{\colour}) + \frac{1}{n} S_{n - M} \vecfun(y) - \int \! \vecfun \,\mathrm{d} \mu } \\
			&\leqslant \frac{1}{n} \norm{S_{M} \vecfun} + \abs[\bigg]{\frac{1}{n} - \frac{1}{n - M}} \norm{S_{n - M} \vecfun} + \norm[\bigg]{ \frac{1}{n - M} S_{n - M} \vecfun(y) - \int \! \vecfun \,\mathrm{d} \mu} \\
			&\leqslant \frac{2M}{n} \norm{\vecfun} + \frac{\varepsilon}{2}.
		\end{split}
	\]
	By Definition~\ref{def:pair structures}, there exists an $n$-pair $P^{n}_{\colour} \in \Pair{n}$  containing $x_{\colour}$.
	Noting that $x_{\colour} \in P^{M}_{\colour}$ and $n \geqslant M$ , by Lemma~\ref{lem:pair neighbor disjoint with Jordan curve}, we get $U^{n}(x_{\colour}) \subseteq \inte[\big]{X^0_{\colour}}$.
	Thus it follows from the definition of $U^{n}(X_{\colour})$ and $P^{n}_{\colour}$ that $x_{\colour} \in P^{n}_{\colour} \subseteq U^{n}(x_{\colour}) \subseteq \inte[\big]{X^{0}_{\colour}}$.
	Then we have \[
		\begin{split}
			\sup_{x \in P^{n}_{\colour}} \norm[\bigg]{ \frac{1}{n} S_n \vecfun(x) - \int \! \vecfun \,\mathrm{d} \mu}
			&\leqslant \sup_{x \in P^{n}_{\colour}} \norm[\bigg]{ \frac{1}{n} S_n \vecfun(x) - \frac{1}{n} S_n \vecfun(x_{\colour}) } + \norm[\bigg]{ \frac{1}{n} S_n \vecfun(x_{\colour}) - \int \! \vecfun \,\mathrm{d} \mu } \\
			&\leqslant \frac{2 D_{n}(\vecfun)}{n} + \frac{2M}{n} \norm{\vecfun} + \frac{\varepsilon}{2}.
		\end{split}
	\]
	Therefore, by Lemma~\ref{lem:distortion lemma for continuous function}, we can find a sufficiently large integer $N \in \n$ such that for each integer $n \geqslant N$ and each color $\colour \in \colours$, there exists an $n$-pair $P^{n}_{\colour} \in \Pair{n}$ such that $P^{n}_{\colour} \subseteq \inte[\big]{X^0_{\colour}}$ and \[
		\sup_{x \in P^{n}_{\colour}} \norm[\bigg]{ \frac{1}{n} S_n \vecfun(x) - \int \! \vecfun \,\mathrm{d} \mu} \leqslant \varepsilon.
	\]
	The proof is complete.
\end{proof}

The following result shows that in order to prove Theorem~\ref{thm:entropy dense}, it suffices to prove entropy density of ergodic measures for some iterate of the map $f$.

\begin{proposition}    \label{prop:entropy-approachablility equivalence for iteration}
    Let $(X, d)$ be a compact metric space and $T \colon X \mapping X$ be a continuous map.
    Consider arbitrary $n \in \n$.
    Then $\ergmea[X][T]$ is entropy-dense in $\invmea[X][T]$ if and only if $\ergmea[X][T^{n}]$ is entropy-dense in $\invmea[X][T^{n}]$.
\end{proposition}
\begin{proof}
    Fix arbitrary $n \in \n$.
    Since $T^n$ itself is also a continuous map from $X$ to $X$, it suffices to prove the ``if'' part.
    We assume that for each $\nu \in \invmea[X][T^n]$, there exists a sequence $\{\nu_{k}\}_{k \in \n}$ of $T^n$-invariant ergodic measures of $T^n$ which converges to $\nu$ in the weak$^{*}$-topology with $h_{\nu_{k}}(T^n) \to h_{\nu}(T^n)$ as $k \to +\infty$.

    Let $\mu \in \invmea[X][T]$ be arbitrary.
    Since $\mu \in \invmea[X][T] \subseteq \invmea[X][T^n]$, there exists a sequence $\{\nu_{k}\}_{k \in \n}$ of $T^n$-invariant ergodic measures of $T^n$ which converges to $\mu$ in the weak$^{*}$-topology with $h_{\nu_{k}}(T^n) \to h_{\mu}(T^n)$ as $k \to +\infty$.
    For each $k \in \n$, we define \[
    	\mu_{k} \define \frac{1}{n} \sum^{n - 1}_{i = 0} T_{*}^{i} \nu_{k}.
    \]
    Then it follows from Lemma~\ref{lem:entropy of push-forward average} that $\mu_{k} \in \invmea[X][T]$ is ergodic for $T$ and $n h_{\mu_{k}}(T) = h_{\nu_{k}}(T^n)$.
    This implies $h_{\mu_{k}}(T) \to h_{\mu}(T)$ as $k \to +\infty$.
    Noting that the sequence $\{\mu_{k}\}_{k \in \n}$ also converges to $\mu$ in the weak$^{*}$-topology, the proof is complete.
\end{proof}

After these preparations, we are ready to prove the main result of this section.

\begin{proof}[Proof of Theorem~\ref{thm:entropy dense}]
	\def\numofmea{s}  \def\length#1{n_{#1}}  \def\period#1{r_{#1}}  
	\def\collect{T}  \def \tempapproximation #1{ \widetilde{\mathbf{\collect}}^{#1} }  \def \approximation #1{ \mathbf{\collect}^{#1} } \def \totalstring { \mathbf{\collect} }
	By Proposition~\ref{prop:entropy-approachablility equivalence for iteration}, it suffices to prove that $\ergmea[S^2][f^{i}]$ is entropy-dense in $\invmea[S^2][f^{i}]$ for some $i \in \n$.
	Hence by Lemma~\ref{lem:invariant_Jordan_curve}, we may assume without loss of generality that there exists a Jordan curve $\mathcal{C} \subseteq S^2$ containing $\post{f}$ such that $f(\mathcal{C}) \subseteq \mathcal{C}$.

	Let $\mu \in \invmea$, $\varepsilon > 0$, $\ell \in \n$, and $\vecfun = (\varphi_{1}, \, \dots, \, \varphi_{\ell}) \in C(S^2)^{\ell}$ be arbitrary.
    By the definition of entropy density (see Subsection~\ref{sub:Main results}), it suffices to find an ergodic measure $\nu \in \ergmea$ such that $\norm[\big]{ \int \! \vecfun \,\mathrm{d}\nu - \int \! \vecfun \,\mathrm{d}\mu } \leqslant \varepsilon$ and $\abs{ h_{\nu}(f) - h_{\mu}(f) } \leqslant \varepsilon$.

	By virtue of the Choquet representation theorem (see for example, \cite[Theorem~3.1.11]{przytycki2010conformal}) and Jacobs' Theorem (see for example, \cite[Theorem~8.4]{walters1982introduction}), for every neighborhood $\Gamma$ of $\mu$ in $\mathcal{M}(S^{2}, f)$ and every $\delta > 0$ there exist $\numofmea \in \n$ and ergodic measures $\listings{\mu}[\numofmea]$ and $\listings{\rho}[\numofmea] \in (0, 1)$ such that $ \rho_{1} + \cdots + \rho_{\numofmea} = 1$, and the measure $\bar{\mu} \define \rho_{1}\mu_{1} + \cdots + \rho_{\numofmea}\mu_{\numofmea}$ belongs to $\Gamma$ and satisfies $|h_{\bar{\mu}}(f) - h_{\mu}(f)| < \delta$.
	Hence, we may assume without loss of generality that $\mu$ is a convex combination of finitely many ergodic measures, i.e.,
	\begin{equation}    \label{eq:temp:thm:entropy dense:ergodic decomposition}
		\mu = \rho_{1}\mu_{1} + \cdots + \rho_{\numofmea}\mu_{\numofmea}
	\end{equation}
	with $\numofmea \in \n$, $\listings{\mu}[\numofmea] \in \ergmea$, $\listings{\rho}[\numofmea] \in (0, 1)$, and $\rho_{1} + \cdots + \rho_{\numofmea} = 1$.

	We fix a $0$-edge $e^0 \in \mathbf{E}^0(f,\mathcal{C})$.
	By Lemmas~\ref{lem:approximates ergodic measures by tiles} and~\ref{lem:pair in the interior}, for each $i \in \oneton[\numofmea]$, there exists $\widetilde{N}_{i} \in \n$ such that for each integer $\length{i} \geqslant \widetilde{N}_{i}$, there exists a non-empty subset $\tempapproximation{\length{i}}$ of $\Tile{\length{i}}$ such that
    \begin{align}
        \label{eq:temp:thm:entropy dense:cardinality}
        \abs[\bigg]{ \frac{1}{\length{i}} \log \card[\big]{\tempapproximation{\length{i}}} - h_{\mu_{i}}(f) }  &\leqslant \frac{\varepsilon}{4} \qquad \text{and}  \\
        \label{eq:temp:thm:entropy dense:Birkhoff average}
        \sup_{x \in \bigcup \tempapproximation{\length{i}}} \norm[\bigg]{ \frac{1}{\length{i}} S_{\length{i}} \vecfun(x) - \int \! \vecfun \,\mathrm{d}\mu_{i}  }  &\leqslant \frac{\varepsilon}{6},
    \end{align}	
    and for each color $\colour \in \colours$, there exists an $\length{i}$-pair $P^{\length{i}}_{\colour} \in \Pair{\length{i}}$ such that $P^{\length{i}}_{\colour} \subseteq \inte[\big]{X^0_{\colour}}$ and 
    \begin{equation}    \label{eq:temp:thm:entropy dense:Birkhoff average on interior pair}
    	\sup_{x \in P^{\length{i}}_{\colour}} \norm[\bigg]{ \frac{1}{\length{i}} S_{\length{i}} \vecfun(x) - \int \! \vecfun \,\mathrm{d} \mu_{i}} \leqslant \frac{\varepsilon}{6}.
    \end{equation}
	Fix arbitrary $i \in \oneton[\numofmea]$.
	For each $\colour \in \colours$ we can write $P^{\length{i}}_{\colour} = X^{\length{i}}_{\black \colour} \cup X^{\length{i}}_{\white \colour}$ for some $X^{\length{i}}_{\black \colour} \in \cTile{\length{i}}{\black}$ and $X^{\length{i}}_{\white \colour} \in \cTile{\length{i}}{\white}$.
	Note that by Proposition~\ref{prop:properties cell decompositions}~\ref{item:prop:properties cell decompositions:cellular}, for each $X^{\length{i}} \in \Tile{\length{i}}$ there exists exactly one $\length{i}$-edge $e^{\length{i}} \subseteq \partial X^{\length{i}}$ such that $f^{\length{i}}(e^{\length{i}}) = e^{0}$.
	Then by Remark~\ref{rem:intersection of two tiles}, there exist exactly two $\length{i}$-tiles in $\Tile{\length{i}}$ containing $e^{\length{i}}$, one of which is $X^{\length{i}}$ itself, and we denote by $\widetilde{X}^{\length{i}}$ the other. 
	Indeed, $X^{\length{i}} \cup \widetilde{X}^{\length{i}}$ is an $\length{i}$-pair in $\Pair{\length{i}}$.
	Now we construct a new subset $\approximation{\length{i}}$ of $\Tile{\length{i}}$ from the old one $\tempapproximation{\length{i}}$  by setting
	\[
		\approximation{\length{i}} \define \tempapproximation{\length{i}} 
		\cup \bigl\{ \widetilde{X}^{\length{i}} \describe X^{\length{i}} \in \tempapproximation{\length{i}} \bigr\} 
		\cup \bigl\{ X^{\length{i}}_{\black \black} \bigr\} \cup \bigl\{ X^{\length{i}}_{\white \black} \bigr\}
		\cup \bigl\{ X^{\length{i}}_{\black \white} \bigr\} \cup \bigl\{ X^{\length{i}}_{\white \white} \bigr\}
	\]
	Then $\bigcup \approximation{\length{i}}$ is actually a union of some pairs in $\Pair{\length{i}}$ and we have
	\[
		\card[\big]{ \tempapproximation{\length{i}} } \leqslant \card{ \approximation{\length{i}} } \leqslant 2 \card[\big]{ \tempapproximation{\length{i}} } + 4.
	\]
	Moreover, it follows from \eqref{eq:temp:thm:entropy dense:Birkhoff average}, \eqref{eq:temp:thm:entropy dense:Birkhoff average on interior pair}, and the structure of pairs that 
	\[
		\sup_{x \in \bigcup \approximation{\length{i}}} \norm[\bigg]{ \frac{1}{\length{i}} S_{\length{i}} \vecfun(x) - \int \! \vecfun \,\mathrm{d}\mu_{i}  }  
		\leqslant \frac{\varepsilon}{6} + \frac{2 D_{\length{i}}(\vecfun) }{\length{i}}.
	\]
	Therefore, by \eqref{eq:temp:thm:entropy dense:cardinality} and Lemma~\ref{lem:distortion lemma for continuous function}, we can find a sufficiently large integer $N_{i} \in \n$ such that for each integer $\length{i} \geqslant N_{i}$, there exists a non-empty subset $\approximation{\length{i}}$ of $\Tile{\length{i}}$ such that 
	\begin{align}
        \label{eq:thm:entropy dense:cardinality}
        \abs[\bigg]{ \frac{1}{\length{i}} \log \card{\approximation{\length{i}}} - h_{\mu_{i}}(f) }  &\leqslant \frac{\varepsilon}{2} \qquad \text{and}  \\
        \label{eq:thm:entropy dense:Birkhoff average}
        \sup_{x \in \bigcup \approximation{\length{i}}} \norm[\bigg]{ \frac{1}{\length{i}} S_{\length{i}} \vecfun(x) - \int \! \vecfun \,\mathrm{d}\mu_{i}  }  &\leqslant \frac{\varepsilon}{3}.
    \end{align}	
    Moreover, we have $\card[\big]{ \approximation{\length{i}} \cap \cTile{\length{i}}{\black} } = \card[\big]{ \approximation{\length{i}} \cap \cTile{\length{i}}{\white} }$, and for each pair of $\juxtapose{\colour}{\colour'} \in \colours$ there exists $X^{\length{i}}_{\colour \colour'} \in \approximation{\length{i}}$ such that $f^{\length{i}}\bigl( X^{\length{i}}_{\colour \colour'} \bigr) = X^0_{\colour}$ and $X^{\length{i}}_{\colour \colour'} \subseteq X^0_{\colour'}$.
    In the rest of the proof for each $i \in \oneton[\numofmea]$ we fix an integer $\length{i} \geqslant N_{i}$ and a corresponding non-empty subset $\approximation{\length{i}}$ of $\Tile{\length{i}}$ obtained from the above construction. 

    We now introduce some notions that will be used in the rest of the proof.
    Let $n \in \n$ and $X^n \in \Tile{n}$ be arbitrary. 
    Set $Y_{n - j} \define f^{j}(X^{n})$ for each $j \in \zeroton[n - 1]$.
    We label the $1$-tiles by $\listings{X^1}[2\deg{f}]$.
    Then by Proposition~\ref{prop:cell decomposition: invariant Jordan curve}, for each $j \in \zeroton[n - 1]$, there exists a unique integer $t_{j} \in \oneton[2\deg{f}]$ such that $Y_{n - j} \subseteq X^1_{t_{j}}$.
    We denote by $w(X^n)$ the $n$-string $t_{0} t_{1} \cdots t_{n - 1}$ and by $[w(X^{n})]$ the $n$-tile $X^n$.
    
    Let $k \in \n$ and $Y^{k} \in \Tile{k}$ be arbitrary.
    If $Y^k \subseteq f^{n}(X^{n})$, then it follows from Proposition~\ref{prop:properties cell decompositions}~\ref{item:prop:properties cell decompositions:cellular} and \cite[Lemma~5.17~(i)]{bonk2017expanding} that $Z^{n + k} \define (f^{n}|_{X^{n}})^{-1}\bigl(Y^{k}\bigr) \in \Tile{n + k}$. 
    One can check that $w\bigl(Z^{n + k}\bigr) = w(X^{n}) w\bigl(Y^{k}\bigr)$ in this case.
    By Definition~\ref{def:primitivity of subsystem} and Remark~\ref{rem:expanding Thurston map is strongly primitive subsystem of itself}, there exists a constant $N \in \n$ such that for each pair of $\juxtapose{\colour}{\colour'} \in \colours$, there exists $X^{N}_{\colour \colour'} \in \cTile{N}{\colour}$ satisfying $X^{N}_{\colour \colour'} \subseteq X^0_{\colour'}$.
    We define $\lambda_{\colour \colour'} \define w\bigl(X^{N}_{\colour \colour'}\bigr)$ for each pair of $\juxtapose{\colour}{\colour'} \in \colours$.
    Then $[\lambda_{\colour \colour'}] = X^{N}_{\colour \colour'}$ for each pair of $\juxtapose{\colour}{\colour'} \in \colours$.
    If $f^{n}(X^{n}) = X^{0}_{\colour}$ and $Y^{k} \subseteq X^{0}_{\colour'}$ for some $\juxtapose{\colour}{\colour'} \in \colours$, we define $\lambda\bigl({X^{n}, Y^{k}}\bigr) \define \lambda_{\colour \colour'}$.
    One can check that there exists $Z^{n + k + N} \in \Tile{n + k + N}$ such that $w\bigl(Z^{n + k + N}\bigr) = w(X^{n}) \lambda\bigl({X^{n}, Y^{k}}\bigr) w\bigl(Y^{k}\bigr)$, $Z^{n + k + N} \subseteq X^n$, and $f^{n + N}\bigl(Z^{n + k + N}\bigr) = Y^k$.

    \def\sect#1{ \mathbf{T}^{\length{#1}, \period{#1}} }  \def\seclen#1{ M_{#1} }  \def\lentotal{ R }
    For each $i \in \oneton[\numofmea]$, let $\period{i} \in \n$ be arbitrary. 
    Denote by $\seclen{i}$ the integer $\length{i}\period{i} + N(\period{i} - 1)$ and by $\sect{i}$ the non-empty subset of $\Tile{\seclen{i}}$ consisting of $\seclen{i}$-tiles of the form
    \begin{equation}    \label{eq:temp:thm:entropy dense:sect form of tile}
    	\bigl[ w(X_{1}) \lambda(X_{1}, X_{2}) w(X_{2}) \lambda(X_{2}, X_{3}) w(X_{3}) \cdots \lambda(X_{\period{i} - 1}, X_{\period{i}}) w(X_{\period{i}}) \bigr]
    \end{equation}
    with $X_{1}, \, \dots, \, X_{\period{i}} \in \approximation{\length{i}}$.
    Denote by $\lentotal$ the integer $sN + \sum_{j = 1}^{\numofmea} \length{j} \period{j}$ and by $\totalstring$ the non-empty subset of $\Tile{\lentotal}$ consisting of $\lentotal$-tiles of the form
    \begin{equation}    \label{eq:temp:thm:entropy dense:totalstring form of tile}
    	\bigl[ w(Y_{1}) \lambda(Y_{1}, Y_{2}) w(Y_{2}) \lambda(Y_{2}, Y_{3}) w(Y_{3}) \cdots \lambda(Y_{\numofmea - 1}, Y_{\numofmea}) w(Y_{\numofmea}) \lambda(Y_{\numofmea}, Y_{1}) \bigr]
    \end{equation}
    with $Y_{j} \in \sect{j}$ for each $j \in \oneton[\numofmea]$.
    Note that
    \begin{equation}    \label{eq:temp:thm:entropy dense:cardinality of totalstring}
    	\card{\totalstring} = \prod_{j = 1}^{\numofmea} \bigl( \card[\big]{\approximation{\length{j}}} \bigr)^{\period{j}}.
    \end{equation}
    Enlarging each $\length{i}$ if necessary, it is possible to choose integers $\period{i}$ such that the following holds:
    \begin{equation}    \label{eq:temp:thm:entropy dense:sufficiently large lentotal}
    	\frac{\log 2}{R} \leqslant \frac{\varepsilon}{6},
    \end{equation}
    \begin{equation}    \label{eq:temp:thm:entropy dense:accuracy of approximation}
    	\sum_{i = 1}^{\numofmea} \biggl( h_{\mu_{i}}(f) + \norm[\bigg]{ \int \! \vecfun \,\mathrm{d}\mu_{i}  }  \biggr) \abs[\Big]{ \rho_{i} - \frac{ \length{i} \period{i }}{\lentotal} } \leqslant \frac{\varepsilon}{3}, \qquad \text{and}
    \end{equation}
    \begin{equation}    \label{eq:temp:thm:entropy dense:cost of linker}
    	\frac{1}{\lentotal} \sum_{i = 1}^{\numofmea} \period{i} \sup_{\colour, \colour' \in \colours} \sup_{x \in [\lambda_{\colour \colour'}]} \norm[\big]{ S_{N} \vecfun(x) } \leqslant \frac{\varepsilon}{3}.
    \end{equation}

    By our construction of $\totalstring$ and $\approximation{\length{i}}$ for $i \in \oneton[\numofmea]$, we have
    \begin{equation}    \label{eq:temp:thm:entropy dense:totalstring is balenced}
    	\card[\big]{ \totalstring \cap \cTile{\lentotal}{\black} } = \card[\big]{ \totalstring \cap \cTile{\lentotal}{\white} } = \card{ \totalstring } / 2
    \end{equation}
    and $f^{\lentotal}|_{\bigcup \totalstring}$ is a strongly primitive subsystem of $f^{\lentotal}$ with respect to $\mathcal{C}$.
    We set $F \define f^{\lentotal}|_{\bigcup \totalstring}$ and $\widehat{F} \define F|_{\limitset} = f^{R}|_{\limitset}$, where $\limitset \define \limitset(F, \mathcal{C})$ is the tile maximal invariant set associated with $F$ with respect to $\mathcal{C}$.
    Then it follows from Theorem~\ref{thm:existence uniqueness and properties of equilibrium state} and \cite[Theorem~1.1]{shi2023thermodynamic} that there exists $\widehat{\nu} \in \mathcal{M}(\limitset, \widehat{F}) \subseteq \mathcal{M}\bigl(S^2, f^R \bigr)$ such that $h_{\widehat{\nu}}\bigl( f^{R} \bigr) = h_{\widehat{\nu}}(\widehat{F}) = P(F, 0) = h_{\operatorname{top}}(F)$ (recall Definition~\ref{def:pressure for subsystem}) and $\widehat{\nu}$ is ergodic for $\widehat{F}$.
    Define 
    \begin{equation}    \label{eq:temp:thm:entropy dense:push-forward average of equilibrium measure}
    	\nu \define \frac{1}{R} \sum_{j = 0}^{R - 1} f_{*}^{j} \widehat{\nu}.
    \end{equation}
    Noting that $\widehat{\nu}$ is also ergodic for $f^{R}$ and then applying Lemma~\ref{lem:entropy of push-forward average}, we deduce that $\nu \in \mathcal{M}(S^{2}, f)$ is ergodic for $f$ and $R h_{\nu}(f) = h_{\widehat{\nu}}(f^{R}) = h_{\operatorname{top}}(F)$.
    
    \smallskip

    We now calculate $h_{\nu}(f)$.
    By Definition~\ref{def:pressure for subsystem} and Proposition~\ref{prop:topological entropy of subsystem}, we have $h_{\operatorname{top}}(F) = \log(\rho(A))$, where $A$ is the tile matrix of $F$ with respect to $\mathcal{C}$ and $\rho(A)$ is the spectral radius of $A$.
    Recall from Definition~\ref{def:tile matrix} and Remark~\ref{rem:tile matrix associated with set of tiles} that \[
        A = A(\totalstring) = \begin{bmatrix}
            N_{\white \white} & N_{\black \white} \\
            N_{\white \black} & N_{\black \black}
        \end{bmatrix},
    \]
    where $N_{\colour \colour'} \define \card[\big]{\set[\big]{ X \in \totalstring \describe X \in \cTile{R}{\colour}, \, X \subseteq X^0_{\colour'} }}$ for each pair of $\juxtapose{\colour}{\colour'} \in \colours$. 
    In particular, since $f(\mathcal{C}) \subseteq \mathcal{C}$, by \eqref{eq:temp:thm:entropy dense:totalstring is balenced} and Proposition~\ref{prop:subsystem:preliminary properties}~\ref{item:subsystem:properties invariant Jordan curve:relation between color and location of tile}, one has $N_{\colour \black} + N_{\colour \white} = \card{\totalstring} / 2$ for each $\colour \in \colours$. 
    Then by some elementary calculations in linear algebra we obtain $\rho(A) = \card{\totalstring} / 2$. 
    Hence $h_{\operatorname{top}}(F) = \log( \card{\totalstring} / 2 )$ and $h_{\nu}(f) = (1 / R) \log(  \card{\totalstring} / 2 )$.

    By \eqref{eq:temp:thm:entropy dense:ergodic decomposition}, \eqref{eq:temp:thm:entropy dense:cardinality of totalstring}, \eqref{eq:thm:entropy dense:cardinality}, \eqref{eq:temp:thm:entropy dense:sufficiently large lentotal}, and \eqref{eq:temp:thm:entropy dense:accuracy of approximation}, we have\[
    	\begin{split}
    		\abs{ h_{\mu}(f) - h_{\nu}(f) } 
    		&\leqslant \frac{\log 2}{R} + \sum_{i = 0}^{\numofmea} \abs[\Big]{ \rho_{i}h_{\mu_{i}} - \frac{1}{R} \period{i} \log\bigl(\card[\big]{\approximation{\length{j}}}\bigr) } \\
    		&\leqslant \frac{\log 2}{R} + \sum_{i = 0}^{\numofmea} h_{\mu_{i}}(f) \abs[\Big]{ \rho_{i} - \frac{\length{i}\period{i}}{R} } + \sum_{i = 0}^{\numofmea} \frac{\length{i}\period{i}}{R} \abs[\Big]{ h_{\mu_{i}} - \frac{1}{\length{i}} \log\bigl(\card[\big]{\approximation{\length{j}}}\bigr) }  \\
    		&\leqslant \frac{\log 2}{R} + \sum_{i = 0}^{\numofmea} h_{\mu_{i}}(f) \abs[\Big]{ \rho_{i} - \frac{\length{i}\period{i}}{R} } + \frac{\varepsilon}{2R} \sum_{i = 0}^{\numofmea} \length{i}\period{i}  \\
    		&\leqslant \frac{\varepsilon}{6} + \frac{\varepsilon}{3} + \frac{\varepsilon}{2} \\
            &= \varepsilon.
    	\end{split}
    \]
    Recall that $R = sN + \sum_{i = 1}^{\numofmea} \length{i} \period{i}$ and each tile $X^{R}$ in $\totalstring$ has the form in~\eqref{eq:temp:thm:entropy dense:totalstring form of tile}, i.e., \[
    	X^{R} = \bigl[ w(Y_{1}) \lambda(Y_{1}, Y_{2}) w(Y_{2}) \lambda(Y_{2}, Y_{3}) w(Y_{3}) \cdots \lambda(Y_{\numofmea - 1}, Y_{\numofmea}) w(Y_{\numofmea}) \lambda(Y_{\numofmea}, Y_{\numofmea + 1}) \bigr]
    \]
	with $Y_{\numofmea + 1} = Y_{1}$ and $Y_{i} \in \sect{i}$ for each $i \in \oneton[\numofmea]$.
    Here $\sect{i}$ is a non-empty subset of $\Tile{\seclen{i}}$ with $\seclen{i} = \length{i}\period{i} + N\period{i} - N$.
    By \eqref{eq:temp:thm:entropy dense:sect form of tile} and \eqref{eq:thm:entropy dense:Birkhoff average}, for each $i \in \oneton[\numofmea]$, \[
    	\begin{split}
    		&\sup_{ x \in [ w(Y_{i}) \lambda(Y_{i}, Y_{i + 1}) ] }  \norm[\bigg]{ S_{\seclen{i} + N} \vecfun(x) - R \rho_{i} \int \! \vecfun \,\mathrm{d}\mu_{i} }  \\
    		&\qquad \leqslant \sup_{ x \in [ w(Y_{i}) \lambda(Y_{i}, Y_{i + 1}) ] }  \norm[\bigg]{ S_{\seclen{i} + N} \vecfun(x) - \length{i}\period{i} \int \! \vecfun \,\mathrm{d}\mu_{i} } + R \abs[\Big]{ \rho_{i} - \frac{\length{i}\period{i}}{R} } \norm[\bigg]{ \int \! \vecfun \,\mathrm{d}\mu_{i} }  \\
    		&\qquad \leqslant \frac{\varepsilon}{3} \length{i}\period{i} + \period{i} \sup_{\colour, \colour' \in \colours} \sup_{x \in [\lambda_{\colour \colour'}]} \norm[\big]{ S_{N} \vecfun(x) } + R \abs[\Big]{ \rho_{i} - \frac{\length{i}\period{i}}{R} } \norm[\bigg]{ \int \! \vecfun \,\mathrm{d}\mu_{i} }.
    	\end{split}
    \]
    Summing this over all $i \in \oneton[\numofmea]$, dividing the result by $R$, and then using \eqref{eq:temp:thm:entropy dense:accuracy of approximation} and \eqref{eq:temp:thm:entropy dense:cost of linker} yield\[
    	\sup_{x \in X^{R}} \norm[\bigg]{ \frac{1}{R} S_{R} \vecfun(x) - \int \! \vecfun \,\mathrm{d}\mu } \leqslant \frac{\varepsilon}{3} + \frac{\varepsilon}{3} + \frac{\varepsilon}{3} = \varepsilon
    \] 
    for each $X^R \in \totalstring$.
    This implies that \[
    	\sup_{x \in \limitset} \norm[\bigg]{ \frac{1}{R} S_{R} \vecfun(x) - \int \! \vecfun \,\mathrm{d}\mu }
    	\leqslant \sup_{x \in \bigcup \totalstring} \norm[\bigg]{ \frac{1}{R} S_{R} \vecfun(x) - \int \! \vecfun \,\mathrm{d}\mu } \leqslant \varepsilon.
    \]
    Note that $\supp{\widehat{\nu}} \subseteq \limitset$.
    Then it follows from \eqref{eq:temp:thm:entropy dense:push-forward average of equilibrium measure} that\[
    	\norm[\bigg]{ \int \! \vecfun \,\mathrm{d}\nu - \int \! \vecfun \,\mathrm{d}\mu } 
    	= \norm[\bigg]{ \frac{1}{R} \int \! S_{R}\vecfun \,\mathrm{d}\widehat{\nu} - \int \! \vecfun \,\mathrm{d}\mu } \leqslant \varepsilon.
    \]
    This shows that the ergodic measure $\nu$ fulfills our requirements (see the beginning of the proof) and completes the proof.
\end{proof}

\section{Large deviation principles}
\label{sec:Large deviation principles}



\subsection{Level-2 large deviation principles}%
\label{sub:Level-2 large deviation principles}
In this subsection we review some basic concepts and results from large deviation theory.
We refer the reader to \cite{dembo2009large,ellis2012entropy,rassoul2015course} for a systematic and detailed introduction. 

Let $\{ \xi_{n} \}_{n \in \n}$ be a sequence of Borel probability measures on a topological space $\mathcal{X}$.
We say that $\{ \xi_{n} \}_{n \in \n}$ satisfies a \emph{large deviation principle} in $\mathcal{X}$ if there exists a lower semi-continuous function $I \colon \mathcal{X} \mapping [0, +\infty]$ such that 
\begin{equation}    \label{eq:level-2 LDP:lower bound}
    \liminf_{n \to +\infty} \frac{1}{n} \log \xi_{n}(\mathcal{G}) \geqslant - \inf_{\mathcal{G}} I  \quad \text{for all open } \mathcal{G} \subseteq \mathcal{X},
\end{equation}
and
\begin{equation}    \label{eq:level-2 LDP:upper bound}
    \limsup_{n \to +\infty} \frac{1}{n} \log \xi_{n}(\mathcal{K}) \leqslant - \inf_{\mathcal{K}} I  \quad \text{for all closed } \mathcal{K} \subseteq \mathcal{X},
\end{equation}
where $\log 0 = -\infty$ and $\inf \emptyset = +\infty$ by convention.
Such a function $I$ is called a \defn{rate function}, and is called a \defn{good rate function} if the set $\{ x \in \mathcal{X} \describe I(x) \leqslant \alpha \}$ is compact for every $\alpha \in [0, +\infty)$.
If $\mathcal{X}$ is a regular topological space, then the rate function $I$ is unique.
A Borel set $\mathcal{A} \subseteq \mathcal{X}$ is called a \emph{$I$-continuity set} if
\[
    \inf \{ I(x) \describe x \in \interior{\mathcal{A}} \} = \inf \{ I(x) \describe x \in \overline{\mathcal{A}} \}.
\]
When \eqref{eq:level-2 LDP:lower bound} and \eqref{eq:level-2 LDP:upper bound} hold then for each $I$-continuity set $\mathcal{A} \subseteq \mathcal{X}$ the limit $\lim_{n \to +\infty} n^{-1} \log \xi_{n}(\mathcal{A})$ exists and satisfies
\begin{equation}    \label{eq:property of I-continuity set}
    \lim_{n \to +\infty} \frac{1}{n} \log \xi_{n}(\mathcal{A}) = - \inf_{\mathcal{A}} I,
\end{equation}
and we can replace $\mathcal{A}$ by either its interior or its closure.
When only \eqref{eq:level-2 LDP:lower bound} (\resp \eqref{eq:level-2 LDP:upper bound}) is satisfied, we say that the \emph{large deviation lower (\resp upper) bounds} hold with the function $I$.

We call $x \in \mathcal{X}$ a \emph{minimizer} if $I(x) = 0$ holds. 
The set of minimizers is a closed set. 
For a closed subset $\mathcal{K}$ of $\mathcal{X}$ which is disjoint from the set of minimizers, the large deviation principle ensures that $\xi_{n}(\mathcal{K})$ decays exponentially as $n \to +\infty$.
If moreover $I$ is a good rate function, the support of any accumulation point of $\sequen{\xi_{n}}$ is contained in the set of minimizers.
Hence, it is important to determine the set of minimizers. %
The non-uniqueness of minimizers is referred to as a phase transition. %
The uniqueness of minimizers implies several strong conclusions. %

The \emph{contraction principle} asserts that when $\{ \xi_{n} \}_{n \in \n}$ is supported on a compact subset of $\mathcal{X}$ and $\{ \xi_{n} \}_{n \in \n}$ satisfies a large deviation principle with rate function $I$, then for every Hausdorff topological space $\mathcal{Y}$ and any continuous map $g \colon \mathcal{X} \mapping \mathcal{Y}$ the sequence of measures $\{ g_{*} (\xi_{n}) \}_{n \in \n}$ satisfies a large deviation principle with rate function defined on $\mathcal{Y}$ by \[
    y \mapsto \inf\{ I(x) \describe x \in \mathcal{X}, \, g(x) = y \}.
\]

The above notations will be applied with $\mathcal{X} = \probmea{X}$ (for some compact metric space $X$), $\mathcal{Y} = \real$, and $g = \widehat{\psi}$ for some $\psi \in C(X)$, where $\widehat{\psi}$ is the evaluation map (i.e., $\widehat{\psi}(\mu) = \int \! \psi \,\mathrm{d}\mu$).
In this context, the large deviation principles in $\probmea{X}$, are usually referred to as ``level-2'', and the ones in $\real$ (in particular those obtained by contraction) as ``level-1''.

\subsection{Uniqueness of the minimizer}%
\label{sub:Uniqueness of the minimizer}

In this subsection, we prove that $\mu_{\potential}$ is the unique minimizer of the rate function $\ratefun$ defined in \eqref{eq:def:rate function}.
Recall that we call $\mu \in \probmea{S^2}$ a minimizer of $\ratefun$ if $\ratefun(\mu) = 0$.

\begin{definition}    \label{def:semiconjugacy}
    Let $X$ and $\widetilde{X}$ be topological spaces, and $T \colon X \mapping X$ and $\widetilde{T} \colon \widetilde{X} \mapping \widetilde{X}$ be continuous maps.
    We say that $T$ is a \emph{factor} (or \emph{topological factor}) of $\widetilde{T}$ if there exists a surjective continuous map $\pi \colon \widetilde{X} \mapping X$ such that $\pi \circ \widetilde{T} = T \circ \pi$. 
    Such map $\pi$ is called a \emph{semiconjugacy}.
\end{definition}

For an expanding Thurston map, by the results in \cite{das2021thermodynamic}, we have the following proposition, which gives a semiconjugacy with the one-sided shift map.

\begin{proposition}    \label{prop:semiconjugacy with full shift}
    Let $f$, $d$, $\potential$ satisfy the Assumptions in Section~\ref{sec:The Assumptions}. 
    Let $\sigma \colon \Sigma \mapping \Sigma$ be the one-sided shift map on $\deg{f}$ symbols.
    Then there exists a semiconjugacy $\pi \colon \Sigma \mapping S^2$ with $\pi \circ \sigma = f \circ \pi$ satisfying the following properties:
    \begin{enumerate}[label= \rm(\roman*)]
        \smallskip

        \item     \label{item:prop:semiconjugacy with full shift:holder}
            The real-valued function $\potential \circ \pi \colon \Sigma \mapping \real$ is \holder continuous with respect to the standard metric on $\Sigma$.

        \smallskip

        \item     \label{item:prop:semiconjugacy with full shift:topological pressure}
            $P(f, \potential) = P(\sigma, \potential \circ \pi)$.

        \smallskip

        \item     \label{item:prop:semiconjugacy with full shift:no entropy drop at equilibrium state}
            Let $\widetilde{\mu} \in \invmea[\Sigma][\sigma]$ be an equilibrium state for the map $\sigma$ and the potential $\potential \circ \pi$.
            Denote $\mu \define \pi_{*}\widetilde{\mu}$.
            Then $h_{\mu}(f) = h_{\widetilde{\mu}}(\sigma)$.
    \end{enumerate}
\end{proposition}

Recall that the standard metric on the shift space $\Sigma$ is given by $\rho(\xi, \eta) \define 2^{- \min\{ i \in \n_{0} \describe \xi_{i} \ne \eta_{i} \} }$ for distinct sequences $\xi = \{ \xi_{i} \}_{i \in \n_{0}}$ and $\eta = \{ \eta_{i} \}_{i \in \n_{0}}$ in $\Sigma$.

Proposition~\ref{prop:semiconjugacy with full shift}~\ref{item:prop:semiconjugacy with full shift:holder} follows immediately from \cite[Lemmma~5.4]{das2021thermodynamic}.
Proposition~\ref{prop:semiconjugacy with full shift}~\ref{item:prop:semiconjugacy with full shift:topological pressure} and \ref{item:prop:semiconjugacy with full shift:no entropy drop at equilibrium state} was established in the proof of \cite[Proposition~5.5]{das2021thermodynamic}.

\smallskip

The following lemma shows that one can ``lift'' invariant measures by a semiconjugacy, whose proof is verbatim the same as that of \cite[Lemma~4.1]{das2021thermodynamic}.

\begin{lemma} \label{lem:lift measures by semiconjugacy}
    Let $X$ and $\widetilde{X}$ be compact metrizable topological spaces, and $T \colon X \mapping X$ and $\widetilde{T} \colon \widetilde{X} \mapping \widetilde{X}$ be continuous maps.
    Suppose that $T$ is a factor of $\widetilde{T}$ and $\pi \colon \widetilde{T} \mapping T$ is a semiconjugacy with $\pi \circ \widetilde{T} = T \circ \pi$.
    Then for each $\mu \in \invmea[X][T]$, there exists $\widetilde{\mu} \in \invmea[\widetilde{X}][\widetilde{T}]$ such that $\pi_{*} \widetilde{\mu} = \mu$.
\end{lemma}

We now prove the uniqueness of the minimizer.

\begin{theorem}    \label{thm:uniqueness of the minimizer}
    Let $f$, $d$, $\potential$, $\mu_{\potential}$ satisfy the Assumptions in Section~\ref{sec:The Assumptions}. 
    Then $\mu_{\potential}$ is the unique minimizer of the rate function $\ratefun$ defined in \eqref{eq:def:rate function}.
\end{theorem}
\begin{proof}
    It follows immediately from \eqref{eq:def:rate function} and \eqref{eq:def:free energy} that $\ratefun(\mu_{\potential}) = 0$, i.e., $\mu_{\potential}$ is a minimizer of $\ratefun$.
    It suffices to show the uniqueness.

    \def\minimizer{\mu_{*}}
    Suppose that $\minimizer \in \probmea{S^2}$ is a minimizer of $\ratefun$, i.e., $\ratefun(\minimizer) = 0$.
    Then by \eqref{eq:def:rate function} and \eqref{eq:def:free energy}, there exists a sequence $\sequen{\mu_{n}}$ of $f$-invariant Borel probability measures which converges to $\minimizer$ in the weak$^{*}$-topology with $\freeenergy(\mu_{n}) \to 0$ as $n \to +\infty$.
    In particular, this implies $\minimizer \in \invmea$ and $\lim_{n \to +\infty} h_{\mu_{n}}(f) = P(f, \potential) - \int_{S^2} \potential \,\mathrm{d}\mu_{*}$ by \eqref{eq:def:free energy}.

    Let $\sigma \colon \Sigma \mapping \Sigma$ be the one-sided shift map on $\deg{f}$ symbols and $\pi \colon \Sigma \mapping S^2$ be the semiconjugacy given by Proposition~\ref{prop:semiconjugacy with full shift}.
    Then by Lemma~\ref{lem:lift measures by semiconjugacy}, there exists a sequence $\sequen{\widetilde{\mu}_{n}}$ of $\sigma$-invariant Borel probability measures on $\Sigma$ such that $\pi_{*}\widetilde{\mu}_{n} = \mu_{n}$ for each $n \in \n$.
    Since the space $\invmea[\Sigma][\sigma]$ is sequentially compact (in the weak$^{*}$-topology), the sequence $\{\widetilde{\mu}_{n}\}_{n \in \n}$ has a convergent subsequence.
    Without loss of generality we may assume that the sequence $\{\widetilde{\mu}_{n}\}_{n \in \n}$ itself converges to $\widetilde{\mu}_{*} \in \invmea[\Sigma][\sigma]$ in the weak$^{*}$-topology.
    Then we have $\mu_{n} = \pi_{*}\widetilde{\mu}_{n} \weakconverge \pi_{*}\widetilde{\mu}_{*}$ as $n \to +\infty$.
    This implies $\pi_{*}\widetilde{\mu}_{*} = \mu_{*}$ since $\mu_{n} \weakconverge \mu_{*}$ as $n \to +\infty$.

    We now show that $\widetilde{\mu}_{*}$ is an equilibrium state for the shift map $\sigma$ and the potential $\potential \circ \pi$.
    Since $\pi_{*}\widetilde{\mu}_{*} = \mu_{*}$, we have $\int_{S^2} \potential \,\mathrm{d}\mu_{*} = \int_{\Sigma} \potential \circ \pi \,\mathrm{d}\widetilde{\mu}_{*}$.
    Noting that for each $n \in \n$, the dynamical system $(S^2, f, \mu_{n})$ is a factor of $(\Sigma, \sigma, \widetilde{\mu}_{n})$, we have $h_{\mu_{n}}(f) \leqslant h_{\widetilde{\mu}_{n}}(\sigma)$ (see for example, \cite[Proposition~4.3.16~(1)]{katok1995introduction}).
    For the shift map $\sigma$, it is a classical result that the entropy map of $\sigma$ is upper semi-continuous (see for example, \cite[Theorem~8.2]{walters1982introduction}). 
    This implies $h_{\widetilde{\mu}_{*}}(\sigma) \geqslant \limsup_{n \to +\infty} h_{\widetilde{\mu}_{n}}(\sigma) \geqslant \limsup_{n \to +\infty} h_{\mu_{n}}(f)$.
    Hence by the Variational principle, we have \[
        P(\sigma, \potential \circ \pi) 
        \geqslant h_{\widetilde{\mu}_{*}}(\sigma) + \int_{\Sigma} \! \potential \circ \pi \,\mathrm{d}\widetilde{\mu}_{*} 
        \geqslant \limsup_{n \to +\infty} h_{\mu_{n}}(f) + \int_{S^2} \! \potential \,\mathrm{d}\mu_{*}.
    \]
    Since $\lim_{n \to +\infty} h_{\mu_{n}}(f) = P(f, \potential) - \int_{S^2} \potential \,\mathrm{d}\mu_{*}$ (see the beginning of the proof), we deduce that
    \[
        P(\sigma, \potential \circ \pi) 
        \geqslant h_{\widetilde{\mu}_{*}}(\sigma) + \int_{\Sigma} \! \potential \circ \pi \,\mathrm{d}\widetilde{\mu}_{*} 
        \geqslant P(f, \potential). 
    \]
    By Proposition~\ref{prop:semiconjugacy with full shift}~\ref{item:prop:semiconjugacy with full shift:topological pressure}, $P(f, \potential) = P(\sigma, \potential \circ \pi)$.
    Thus, $\widetilde{\mu}_{*}$ is an equilibrium state for the shift map $\sigma$ and the potential $\potential \circ \pi$.

    Finally, since $\mu_{*} = \pi_{*}\widetilde{\mu}_{*}$ and $\widetilde{\mu}_{*}$ is an equilibrium state for the shift map $\sigma$ and the potential $\potential \circ \pi$, it follows from Proposition~\ref{prop:semiconjugacy with full shift}~\ref{item:prop:semiconjugacy with full shift:no entropy drop at equilibrium state} that $h_{\mu_{*}}(f) = h_{\widetilde{\mu}_{*}}(\sigma)$.
    Therefore, we have
    \[
        h_{\mu_{*}}(f) + \int_{S^2} \potential \,\mathrm{d}\mu_{*} = h_{\widetilde{\mu}_{*}}(\sigma) + \int_{\Sigma} \! \potential \circ \pi \,\mathrm{d}\widetilde{\mu}_{*} 
        = P(\sigma, \potential \circ \pi) = P(f, \potential),
    \]
    i.e., $\mu_{*}$ is an equilibrium state for the map $f$ and the potential $\potential$.
    This implies $\mu_{*} = \mu_{\potential}$ by the uniqueness of the equilibrium state (see Theorem~\ref{thm:properties of equilibrium state}~\ref{item:thm:properties of equilibrium state:existence and uniqueness}).
    The proof is complete.
\end{proof}

\subsection{Characterizations of topological pressures}
\label{sub:Characterizations of topological pressures}

In this subsection, we characterize the topological pressure in terms of periodic points and iterated preimages.

\begin{lemma}    \label{lem:good location n-tile has high level periodic point and preimage point}
    Let $f$ and $\mathcal{C}$ satisfy the Assumptions in Section~\ref{sec:The Assumptions}. 
    Let $n_f \in \n$ be the constant from Definition~\ref{def:primitivity of subsystem}, which depends only on $f$ and $\mathcal{C}$.
    Consider arbitrary integer $m \geqslant n_{f}$, integer $n \in \n_0$, point $x \in S^2$, and $n$-tile $X^n \in \Tile{n}$.
    Then the following statements hold:
    \begin{enumerate}[label=\rm{(\roman*)}]
        \smallskip

        \item    \label{item:lem:good location n-tile has high level periodic point and preimage point:periodic point} 
            If $X^n \subseteq X^0_{\colour}$ for some $\colour \in \colours$, then there exists a fixed point of $f^{n + m}$ in $\inte{X^n}$.
        
        \smallskip
        
        \item    \label{item:lem:good location n-tile has high level periodic point and preimage point:preimage point} 
             There exists a preimage of $x$ under $f^{n + m}$ in $\inte{X^n}$.
    \end{enumerate}
\end{lemma}
\begin{proof}
    Let integer $m \geqslant n_f$, $n \in \n_0$, and $X^n \in \Tile{n}$ be arbitrary. 
    Since $f \in \subsystem$ is strongly primitive, by Lemma~\ref{lem:strongly primitive:tile in interior tile for high enough level}, for each $\colour' \in \colours$ there exists $X^{n + m}_{\colour'} \in \cTile{n + m}{\colour'}$ such that $X^{n + m}_{\colour'} \subseteq \inte{X^n}$. 

    \smallskip

    \ref{item:lem:good location n-tile has high level periodic point and preimage point:periodic point} 
    Suppose that $X^{n} \subseteq X^0_{\colour}$ for some $\colour \in \colours$.
    By Proposition~\ref{prop:properties cell decompositions}~\ref{item:prop:properties cell decompositions:cellular}, $f^{n + m}|_{X^{n + m}_{\colour}}$ is a homeomorphism of $X^{n + m}_{\colour}$ onto $X^{0}_{\colour}$.
    Note that $X^{n + m}_{\colour} \subseteq \inte{X^n} \subseteq X^0_{\colour}$.
    Then by applying Brouwer's Fixed Point Theorem (see for example, \cite[Theorem~1.9]{hatcher2002algebraic}) to the inverse of $f^{n + m}$ restricted to $X^{n + m}_{\colour}$, we get a fixed point $x \in X^{n + m}_{\colour} \subseteq \inte{X^n}$ of $f^{n + m}$.

    \smallskip

    \ref{item:lem:good location n-tile has high level periodic point and preimage point:preimage point}
    Since $X^{n + m}_{\black} \cup X^{n + m}_{\white} \subseteq \inte{X^n}$, it follows from Proposition~\ref{prop:properties cell decompositions}~\ref{item:prop:properties cell decompositions:cellular} that $x \in S^{2} = f^{n + m}(X^{n + m}_{\black}) \cup f^{n + m}(X^{n + m}_{\black}) \subseteq f^{n + m}(\inte{X^{n}})$.
    This implies that there exists $y \in f^{- n - m}(x)$ such that $y \in \inte{X^n}$.
\end{proof}

We recall the following characterizations of topological pressure in terms of periodic points and iterated preimages (see~\cite[Propositions~6.8 and~6.7]{li2015weak}, respectively). 

\begin{proposition}[Z.~Li \cite{li2015weak}]    \label{prop:characterization of pressure weighted periodic points}
    Let $f$, $d$, $\potential$ satisfy the Assumptions in Section~\ref{sec:The Assumptions}.
    Fix an arbitrary sequence $\sequen{w_n}$ of real-valued functions on $S^2$ satisfying $w_{n}(x) \in \bigl[ 1, \deg_{f^{n}}(x) \bigr]$ for each $n \in \n$ and each $x \in S^2$.
    Then\[
        P(f, \potential) = \lim_{n \to +\infty} \frac{1}{n} \log \sum_{ x \in \periodorbit } w_{n}(x) \myexp{ S_{n}\potential(x) }.
    \]
\end{proposition}

\begin{proposition}[Z.~Li \cite{li2015weak}]    \label{prop:characterization of pressure iterated preimages count by degree}
    Let $f$, $d$, $\potential$ satisfy the Assumptions in Section~\ref{sec:The Assumptions}.
    Then for each sequence $\sequen{x_{n}}$ in $S^2$, we have 
    \[
        P(f, \potential) = \lim_{n \to +\infty} \frac{1}{n} \log \sum_{ y \in f^{-n}(x_{n}) } \deg_{f^{n}}(y) \myexp{ S_{n}\potential(y) }.
    \]
    If we also assume that $f$ has no periodic critical points, then for an arbitrary sequence $\sequen{w_n}$ of real-valued functions on $S^2$ satisfying $w_{n}(x) \in \bigl[ 1, \deg_{f^{n}}(x) \bigr]$ for each $n \in \n$ and each $x \in S^2$, we have
    \[
        P(f, \potential) = \lim_{n \to +\infty} \frac{1}{n} \log \sum_{ y \in f^{-n}(x_{n}) } w_{n}(y) \myexp{ S_{n}\potential(y) }.
    \]

\end{proposition}

We now prove a generalization of Proposition~\ref{prop:characterization of pressure iterated preimages count by degree}.

\begin{proposition}    \label{prop:characterization of pressure iterated preimages arbitrary weight}
    Let $f$, $d$, $\potential$ satisfy the Assumptions in Section~\ref{sec:The Assumptions}.
    Fix an arbitrary sequence $\sequen{w_n}$ of real-valued functions on $S^2$ satisfying $w_{n}(x) \in \bigl[ 1, \deg_{f^{n}}(x) \bigr]$ for each $n \in \n$ and each $x \in S^2$.
    Then for each sequence $\sequen{x_{n}}$ in $S^2$, we have \[
        P(f, \potential) = \lim_{n \to +\infty} \frac{1}{n} \log \sum_{ y \in f^{-n}(x_{n}) } w_{n}(y) \myexp{ S_{n}\potential(y) }.
    \]
\end{proposition}
\begin{proof}
    By Proposition~\ref{prop:characterization of pressure iterated preimages count by degree}, it suffices to show that
    \begin{equation}    \label{eq:temp:prop:characterization of pressure iterated preimages arbitrary weight:upper bound for preimage cound by degree}
        \lim_{n \to +\infty} \frac{1}{n} \sum_{y \in f^{-n}(x_{n})} \deg_{f^{n}}(y) \myexp{S_{n}\potential}(y) 
        \leqslant \liminf_{m \to +\infty} \frac{1}{m} \sum_{\widehat{y} \in f^{-m}(x_{m})} \myexp{S_{m}\potential(\widehat{y})}. 
    \end{equation}

    We fix a Jordan curve $\mathcal{C} \subseteq S^2$ that satisfies the Assumptions in Section~\ref{sec:The Assumptions}.
    Suppose that $\potential \in C^{0,\holderexp}(S^2, d)$ is a real-valued H\"{o}lder continuous function with exponent $\holderexp \in (0, 1]$.

    \def\pttilepreimage{\widehat{y}(X^n)}  \def\iteptpreimage{\widehat{y}(X^{n}(y))}  \def\itepreimageset{f^{-n -N}(x_{n + N})}  \def\preimageset{f^{-n}(x_{n})}
    Let $N \define n_{f} \in \n$ be the constant from Definition~\ref{def:primitivity of subsystem}, which depends only on $f$ and $\mathcal{C}$.
    For each $n \in \n$ and each $X^{n} \in \Tile{n}$, it follows from Lemma~\ref{lem:good location n-tile has high level periodic point and preimage point}~\ref{item:lem:good location n-tile has high level periodic point and preimage point:preimage point} that there exists a preimage of $x_{n + N}$ under $f^{n + N}$ in $\inte{X^{n}}$.
    We fix a preimage of of $x_{n + N}$ under $f^{n + N}$ in $\inte{X^{n}}$ and denote it by $\pttilepreimage$.
    Then for each $n \in \n$, the map $X^n \mapsto \pttilepreimage$ from $\Tile{n}$ to $\itepreimageset$ is injective.

    For each $n \in \n$ and each $y \in \preimageset$, let $X^n(y) \in \Tile{n}$ be an $n$-tile that contains $y$.
    By Proposition~\ref{prop:properties cell decompositions}~\ref{item:prop:properties cell decompositions:cellular}, for each $n \in \n$ and each $X^n \in \Tile{n}$, $f^{n}|_{X^{n}}$ is a homeomorphism of $X^{n}$ onto $f^{n}(X^n)$.
    This implies that for each integer $n \in \n$, the map $y \mapsto \iteptpreimage$ from $\preimageset$ to $\itepreimageset$ is injective, where $\iteptpreimage \in \inte{X^{n}(y)}$.
    
    \def\tileflower{\mathbf{X}^{n}(f, \mathcal{C}, y')}
    Let $n \in \n$ be arbitrary. 
    Consider $y' \in \preimageset \cap \vertex{n}$, where $\vertex{n} = \Vertex{n}$ is the set of $n$-vertices.
    We set $\tileflower \define \{X \in \Tile{n} \describe y' \in X \}$.
    By Remark~\ref{rem:flower structure}, we have $\cflower{n}{y'} = \bigcup \tileflower$ and $\card{\tileflower} = 2 \deg_{f^{n}}(y')$, where $\flower{n}{y'}$ is defined in \eqref{eq:n-flower} and $\cflower{n}{y'}$ is the closure of $\flower{n}{y'}$.
    Note that $X^n(y) \notin \tileflower$ for every $y \in \preimageset \setminus \Vertex{n}$.

    We now establish \eqref{eq:temp:prop:characterization of pressure iterated preimages arbitrary weight:upper bound for preimage cound by degree}.
    By the arguments above, for each integer $n \in \n$, we have
    \[
        \begin{split}
            &\sum_{ y \in \preimageset }  \deg_{f^{n}}(y) \myexp{ S_{n}\potential(y) } \\
            &\qquad\leqslant \sum_{ y' \in \preimageset \cap \vertex{n} }  \deg_{f^{n}}(y') \myexp{ S_{n}\potential(y') }
                + \sum_{ y \in \preimageset \setminus \vertex{n} }  \myexp{ S_{n}\potential(y) }  \\
            &\qquad\leqslant \sum_{ y' \in \preimageset \cap \vertex{n} }  \sum_{ X^{n} \in \tileflower } e^{D_{n}(\potential)} \myexp{ S_{n}\potential(\pttilepreimage) } \\
            &\qquad\qquad + \sum_{ y \in \preimageset \setminus \vertex{n} }  e^{D_{n}(\potential)} \myexp{ S_{n}\potential(\iteptpreimage) }  \\
            &\qquad\leqslant e^{D_{n}(\potential)} \sum_{ \widehat{y} \in \itepreimageset } \myexp{ S_{n} \potential(\widehat{y}) } \\
            &\qquad\leqslant e^{D_{n}(\potential)} e^{ N \uniformnorm{\potential} } \sum_{ \widehat{y} \in \itepreimageset } \myexp{ S_{n + N} \potential(\widehat{y}) }.
        \end{split}
    \]
    Then by Lemma~\ref{lem:distortion_lemma}, we get
    \[
        \frac{1}{n} \log \sum_{ y \in \preimageset }  \deg_{f^{n}}(y) \myexp{ S_{n}\potential(y) } 
        \leqslant \frac{C}{n} + \frac{1}{n} \log \sum_{ \widehat{y} \in \itepreimageset } \myexp{ S_{n} \potential(\widehat{y}) },
    \]
    where $C \define N \uniformnorm{\potential} + \Cdistortion$ and $C_{1} \geqslant 0$ is the constant defined in \eqref{eq:const:C_1} in Lemma~\ref{lem:distortion_lemma} and depends only on $f$, $\mathcal{C}$, $d$, $\phi$, and $\holderexp$.
    Letting $n \to +\infty$ yields the desired inequality.
\end{proof}

\subsection{Large deviation lower bound}%
\label{sub:Large deviation lower bound}

This subsection is devoted to the proof of the lower bound \eqref{eq:level-2 LDP:lower bound} for all open sets, with the main result being Proposition~\ref{prop:lower bound for open sets}.
In Section~\ref{ssub:Reduction to ergodic measures} we show that the proof of the lower bound can be reduced to the case where the invariant measure in question is ergodic.
In Section~\ref{ssub:Lower bound for fundamental open sets} we prove lower bounds for certain fundamental open subsets of $\probsphere$, where we apply Lemma~\ref{lem:approximates ergodic measures by tiles} to approximate each ergodic measure with a collection of tiles.
Finally, in Section~\ref{ssub:End of proof of the lower bound} we establish Proposition~\ref{prop:lower bound for open sets}.

\begin{proposition}    \label{prop:lower bound for open sets}
    Let $f$, $d$, $\phi$ satisfy the Assumptions in Section~\ref{sec:The Assumptions}.
    Then for each sequence $\sequen{\xi_{n}} \in \bigl\{ \sequen{\birkhoffmeasure}, \sequen{\Omega_{n}}, \sequen{\Omega_{n}(x_{n})} \bigr\}$ (as defined in Theorem~\ref{thm:level-2 large deviation principle}), we have
    \[
        \liminf_{n \to +\infty} \frac{1}{n} \log \xi_{n}(\mathcal{G}) \geqslant - \inf_{\mathcal{G}} \ratefun  \quad \text{for all open } \mathcal{G} \subseteq \probsphere,
    \]
    where $\ratefun \colon \probsphere \mapping [0, +\infty]$ is defined in \eqref{eq:def:rate function}. 
\end{proposition}


\subsubsection{Reduction to ergodic measures}%
\label{ssub:Reduction to ergodic measures}

A weaker property related to entropy density (defined in Subsection~\ref{sub:Main results}) is entropy approachablility (see Definition~\ref{def:weak entropy density}).
Entropy approachablility is a useful property in theories such as multifractal analysis and large deviations in which all invariant measures come into play, in order to reduce one's consideration to ergodic measures only.

\begin{definition}    \label{def:weak entropy density}
    Let $(X, d)$ be a compact metric space and $T \colon X \mapping X$ be a continuous map.
    We say that a measure $\mu \in \invmea$ is \emph{entropy-approachable by ergodic measures} if for each $\varepsilon > 0$ and each weak$^{*}$-open set $U$ containing $\mu$ there exists an ergodic measure $\nu \in U \cap \mathcal{M}(X, T)$ such that $h_{\nu}(T) > h_{\mu}(T) - \varepsilon$.
\end{definition}

\begin{remark}\label{rem:relation between entropy density and entropy approachablility}
    It is clear that if ergodic measures are entropy-dense, then any invariant measure is entropy-approachable by ergodic measures.
    One sees that these two notions are equivalent when the entropy map is upper semi-continuous.
\end{remark}

It follows immediately from Theorem~\ref{thm:entropy dense} and Remark~\ref{rem:relation between entropy density and entropy approachablility} that for expanding Thurston maps, any invariant measure is entropy-approachable by ergodic measures.

\begin{corollary}    \label{coro:entropy dense}
    For an expanding Thurston map $f \colon S^2 \mapping S^2$, any invariant measure $\mu \in \invmea$ is entropy-approachable by ergodic measures.
\end{corollary}

\begin{remark}\label{rem:specification implies entropy density}
    For a continuous map $T \colon X \mapping X$ on a compact metric space $(X, d)$, it is known that if $T$ has the specification property in the sense of K.~Sigmund (see the definition in \cite[Section~2]{sigmund1974dynamical}), then any invariant measure is entropy-approachable by ergodic measures (see for example, \cite{eizenberg1994large} and \cite[Theorem~2.1]{pfister2005large}).
    In particular, this result applies to expanding Thurston maps since every expanding Thurston map has the specification property (see the proof of \cite[Lemma~6.5]{li2023ground}).
\end{remark}

\subsubsection{Lower bound for fundamental open sets}%
\label{ssub:Lower bound for fundamental open sets}

We use the notations as introduced in the beginning of Section~\ref{sec:Entropy density}. 

We first prove the following result under the additional assumption that there exists an $f$-invariant Jordan curve $\mathcal{C} \subseteq S^2$ with $\post{f} \subseteq \mathcal{C}$ and then for the general case.
\begin{proposition}    \label{prop:lower bound for fundamental open sets}
    Let $f$, $d$, $\phi$, $\holderexp$ satisfy the Assumptions in Section~\ref{sec:The Assumptions}.
    Consider $\ell \in \n$, $\vecfun \in \multispace$, and $\vecavg \in \real^{\ell}$.
    Let $\mathcal{G} \subseteq \probsphere$ be an open set of the form\[
        \mathcal{G} \define \biggl\{ \mu \in \probsphere \describe \int \! \vecfun \,\mathrm{d}\mu > \vecavg \biggr\}.
    \]
    Then for each $\mu \in \mathcal{G}$ and each sequence $\sequen{\xi_{n}} \in \bigl\{ \sequen{\birkhoffmeasure}, \sequen{\Omega_{n}}, \sequen{\Omega_{n}(x_{n})} \bigr\}$ (as defined in Theorem~\ref{thm:level-2 large deviation principle}), we have
    \[
        \liminf_{n \to +\infty} \frac{1}{n} \log \xi_{n}(\mathcal{G}) \geqslant \freeenergy(\mu),
    \]
    where $\freeenergy \colon \probsphere \mapping [-\infty, 0]$ is defined in \eqref{eq:def:free energy}. 
\end{proposition}

\begin{proof}[Proof of Proposition~\ref{prop:lower bound for fundamental open sets} under an additional assumption]
    We assume in addition that there exists an $f$-invariant Jordan curve $\mathcal{C} \subseteq S^2$ with $\post{f} \subseteq \mathcal{C}$.

    Let $\mu \in \mathcal{G}$ be arbitrary. 
    We may assume without loss of generality that $\mu \in \invmea$ since $\freeenergy(\mu) = -\infty$ when $\mu \notin \invmea$.
    Moreover, by virtue of Corollary~\ref{coro:entropy dense} and the definition of $\freeenergy$, we may assume that $\mu$ is ergodic.

    Let $\varepsilon > 0$ be such that
    \begin{equation}    \label{eq:temp:prop:lower bound for fundamental open sets:choice of epsilon}
        \int \! \vecfun \,\mathrm{d}\mu > \vecavg + \varepsilon.
    \end{equation}
    \def \approximation #1{ \mathbf{T}^{#1} }
    By Lemma~\ref{lem:approximates ergodic measures by tiles}, there exists $n_0 \in \n$ such that for each integer $n \geqslant n_{0}$, there exists a non-empty subset $\approximation{n}$ of $\Tile{n}$ such that 
    \begin{align}
        \label{eq:temp:prop:lower bound for fundamental open sets:approximation cardinality}
        \bigg| \frac{1}{n} \log \card{\approximation{n}} - h_{\mu}(f) \bigg| &\leqslant \frac{\varepsilon}{2} ,   \\
        \label{eq:temp:prop:lower bound for fundamental open sets:approximation weak*-topology}
        \sup_{x \in \bigcup \approximation{n}} \norm[\bigg]{ \frac{1}{n} S_n \vecfun(x) - \int \! \vecfun \,\mathrm{d}\mu} &\leqslant \frac{\varepsilon}{2}, \qquad \text{and} \\
        \label{eq:temp:prop:lower bound for fundamental open sets:approximation potential}
        \sup_{x \in \bigcup \approximation{n}} \abs[\bigg]{ \frac{1}{n} S_n \potential(x) - \int \! \potential \,\mathrm{d}\mu} &\leqslant \frac{\varepsilon}{2}.
    \end{align}

    We split the rest of the proof into three cases according to the type of the sequence $\sequen{\xi_{n}}$.

    \smallskip
    \emph{Case 1 (Birkhoff averages):} $\xi_{n} = \birkhoffmeasure = (V_{n})_{*}(\mu_{\potential})$ for each $n \in \n$ (recall \eqref{eq:def:delta measure for orbit}).
    \smallskip

    For each integer $n \geqslant n_{0}$, \eqref{eq:temp:prop:lower bound for fundamental open sets:choice of epsilon} and \eqref{eq:temp:prop:lower bound for fundamental open sets:approximation weak*-topology} yield
    \begin{equation}    \label{eq:temp:prop:lower bound for fundamental open sets:approximation fundamental open set by tiles}
        \Bigl\{ \deltameasure{x} \describe x \in \bigcup \approximation{n} \Bigr\} \subseteq \mathcal{G}.
    \end{equation}

    Recall from Proposition~\ref{prop:equilibrium state is gibbs measure} that $\mu_{\potential}$ is a Gibbs measure with respect to $f$, $\mathcal{C}$, and $\phi$, with the constants $P_{\mu_{\phi}} = P(f, \phi)$ and $C_{\mu_{\potential}} \geqslant 1$.
    Then for each integer $n \geqslant n_{0}$ and each $X^{n} \in \approximation{n}$, it follows from \eqref{eq:equilibrium state is gibbs measure} and \eqref{eq:temp:prop:lower bound for fundamental open sets:approximation potential} that \[
        \mu_{\potential}(X^{n}) \geqslant C_{\mu_{\potential}}^{-1} e^{- n P(f, \potential)} \inf_{x \in X^n} \myexp{ S_{n}\potential(x) } \geqslant C_{\mu_{\potential}}^{-1} e^{- n P(f, \potential)} \myexp[\bigg]{ n \biggl(\int \! \potential \,\mathrm{d}\mu - \frac{\varepsilon}{2} \biggr) }.
    \]
    Summing this inequality over all $X^{n} \in \approximation{n}$ and applying \eqref{eq:temp:prop:lower bound for fundamental open sets:approximation fundamental open set by tiles}, Theorem~\ref{thm:properties of equilibrium state}~\ref{item:thm:properties of equilibrium state:edge measure zero}, and \eqref{eq:temp:prop:lower bound for fundamental open sets:approximation cardinality}, we have\[
        \begin{split}
            \frac{1}{n} \log \birkhoffmeasure(\mathcal{G}) 
            &= \frac{1}{n} \log \mu_{\potential} ( \{ x \in S^2 \describe \deltameasure{x} \in \mathcal{G} \} )  \\
            &\geqslant \frac{1}{n} \log \mu_{\potential} \Bigl( \bigcup \approximation{n} \Bigr) \geqslant 
            \frac{1}{n} \log \Bigl( \card{\approximation{n}} \inf_{X^{n} \in \approximation{n}} \mu_{\potential}(X^{n}) \Bigr) \\
            &\geqslant h_{\mu}(f) - \frac{\varepsilon}{2} + \int \! \potential \,\mathrm{d}\mu - \frac{\varepsilon}{2} - P(f, \potential) - \frac{1}{n} \log C_{\mu_{\potential}}  \\
            &= \freeenergy(\mu) - \varepsilon - \frac{1}{n} \log C_{\mu_{\potential}}.
        \end{split}
    \]
    Letting $n \to +\infty$ and then $\varepsilon \to 0$ yields the desired inequality.

    \smallskip
    \emph{Case 2 (Periodic points):} $\xi_{n} = \Omega_{n}$ for each $n \in \n$ (recall \eqref{eq:def:Periodic points distribution}). 
    \smallskip

    By Proposition~\ref{prop:cell decomposition: invariant Jordan curve} and Lemma~\ref{lem:good location n-tile has high level periodic point and preimage point}~\ref{item:lem:good location n-tile has high level periodic point and preimage point:periodic point}, there exists a constant $N \in \n$ depending only on $f$ and $\mathcal{C}$ such that for each $n \in \n_0$ and each $X^n \in \Tile{n}$, there exists a fixed point of $f^{n + N}$ in $\inte{X^n}$.
    For each $n \in \n_0$ and each $X^n \in \Tile{n}$, let $p(X^n)$ be a fixed point of $f^{n + N}$ in $\inte{X^n}$.
    Then the map $X^n \mapsto p(X^n)$ from $\Tile{n}$ to $\periodorbit[n + N]$ is injective.

    By \eqref{eq:temp:prop:lower bound for fundamental open sets:choice of epsilon} and \eqref{eq:temp:prop:lower bound for fundamental open sets:approximation weak*-topology}, for each integer $n \geqslant n_{0}$ and each $x \in \bigcup \approximation{n}$, we have\[
        S_{n + N} \vecfun(x) \geqslant S_n \vecfun(x) - N \norm{\vecfun} > n \vecavg + \varepsilon n / 2 - N \norm{\vecfun} \geqslant (n + N) \vecavg + \varepsilon n / 2 - N (\norm{\vecfun} + \norm{\vecavg}).
    \]
    This implies that for each sufficiently large $n \in \n$ and each $x \in \bigcup \approximation{n}$, $(n + N)^{-1} S_{n + N} \vecfun(x) > \vecavg$, i.e., $\deltameasure[n + N]{x} \in \mathcal{G}$.
    Therefore, it follows from \eqref{eq:temp:prop:lower bound for fundamental open sets:approximation cardinality} and \eqref{eq:temp:prop:lower bound for fundamental open sets:approximation potential} that for each sufficiently large $n \in \n$, 
    \begin{equation}    \label{eq:temp:prop:lower bound for fundamental open sets:periodic points:lower bound by approximation}
        \begin{split}
            \sum_{ \substack{ p \in \periodorbit[n + N] \\ \deltameasure[n + N]{p} \in \mathcal{G} } }  w_{n + N}(p) \myexp{ S_{n + N}\potential(p) } 
            &\geqslant \sum_{X^n \in \approximation{n}} \myexp{ S_{n + N}\potential(p(X^{n})) } \\
            &\geqslant \card{\approximation{n}} \inf_{ x \in \bigcup\approximation{n} } \myexp{ S_{n}\potential(x) } \myexp{ - N \uniformnorm{\potential} } \\
            &\geqslant \myexp{ \freeenergy(\mu) n + P(f, \potential) n - \varepsilon n - N \uniformnorm{\potential} },
        \end{split}
    \end{equation}
    where $\{ w_{j} \}_{j \in \n}$ is an arbitrary sequence of real-valued functions on $S^2$ with $w_{j}(x) \in \bigl[ 1, \deg_{f^{j}}(x) \bigr]$ for each $j \in \n$ and each $x \in S^{2}$.
    By \eqref{eq:def:Periodic points distribution}, for each $n \in \n$, \[
        \begin{split}
            \frac{1}{n + N} \log \Omega_{n + N}(\mathcal{G}) 
            &= - \frac{1}{n + N} \log \sum_{ p' \in \periodorbit[n + N] } w_{n + N}(p') \myexp{ S_{n}\potential(p') } \\
            &\quad + \frac{1}{n + N} \log \sum_{ \substack{ p \in \periodorbit[n + N] \\ \deltameasure[n + N]{x} \in \mathcal{G} } }  w_{n + N}(p) \myexp{ S_{n + N}\potential(p) }.
        \end{split}
    \]
    Note that as $n \to +\infty$, the first term of the right hand side in the equation above converges to $- P(f, \potential)$ by Proposition~\ref{prop:characterization of pressure weighted periodic points}.
    Combining this with \eqref{eq:temp:prop:lower bound for fundamental open sets:periodic points:lower bound by approximation}, we get\[
        \liminf_{n \to +\infty} \frac{1}{n} \log \Omega_{n}(\mathcal{G}) 
        = \liminf_{n \to +\infty} \frac{1}{n + N} \log \Omega_{n + N}(\mathcal{G}) 
        \geqslant \freeenergy(\mu) - \varepsilon.
    \]
    Then by letting $\varepsilon \to 0$, the desired inequality follows.

    \smallskip
    \emph{Case 3 (Iterated preimages):} $\xi_{n} = \Omega_{n}(x_{n})$ for each $n \in \n$ (recall \eqref{eq:def:Iterated preimages distribution}), where $\{ x_{n} \}_{n \in \n}$ is an arbitrary sequence of points in $S^{2}$.
    \smallskip

    By Lemma~\ref{lem:good location n-tile has high level periodic point and preimage point}~\ref{item:lem:good location n-tile has high level periodic point and preimage point:preimage point}, there exists a constant $N \in \n$ depending only on $f$ and $\mathcal{C}$ such that for each $n \in \n_0$ and each $X^n \in \Tile{n}$, there exists $x \in f^{- n - N}(x_{n + N})$ such that $x \in \inte{X^n}$.
    For each $n \in \n_0$ and each $X^n \in \Tile{n}$, let $x_{X^n}$ be a preimage point of $x_{n + N}$ under $f^{n + N}$ such that $x_{X^{n}} \in \inte{X^n}$.
    Then the map $X^n \mapsto x_{X^n}$ from $\Tile{n}$ to $f^{- n - N}(x_{n + N})$ is injective.
    
    By the same reasoning as in Case~2, we have $\deltameasure[n + N]{x} \in \mathcal{G}$ for each sufficiently large $n \in \n$ and each $x \in \bigcup \approximation{n}$.
    Similarly, it follows from \eqref{eq:temp:prop:lower bound for fundamental open sets:approximation cardinality} and \eqref{eq:temp:prop:lower bound for fundamental open sets:approximation potential} that for each sufficiently large $n \in \n$, 
    \begin{equation}    \label{eq:temp:prop:lower bound for fundamental open sets:iterated preimages:lower bound by approximation}
        \begin{split}
            \sum_{ \substack{ y \in f^{- n - N}(x_{n + N}) \\ \deltameasure[n + N]{p} \in \mathcal{G} } }  w_{n + N}(y) \myexp{ S_{n + N}\potential(y) } 
            &\geqslant \sum_{X^n \in \approximation{n}} \myexp{ S_{n + N}\potential(x_{X^n}) } \\
            &\geqslant \card{\approximation{n}} \inf_{ x \in \bigcup\approximation{n} } \myexp{ S_{n}\potential(x) } \myexp{ - N \uniformnorm{\potential} } \\
            &\geqslant \myexp{ \freeenergy(\mu) n + P(f, \potential) n - \varepsilon n - N \uniformnorm{\potential} },
        \end{split}
    \end{equation}
    where $\{ w_{j} \}_{j \in \n}$ is an arbitrary sequence of real-valued functions on $S^2$ with $w_{j}(x) \in \bigl[ 1, \deg_{f^{j}}(x) \bigr]$ for each $j \in \n$ and each $x \in S^{2}$.
    By \eqref{eq:def:Iterated preimages distribution}, for each $n \in \n$, \[
        \begin{split}
            \frac{1}{n + N} \log \Omega_{n + N}(x_{n + N})(\mathcal{G}) 
            &= - \frac{1}{n + N} \log \sum_{ z \in f^{- n - N}(x_{n + N}) } w_{n + N}(z) \myexp{ S_{n}\potential(z) } \\
            &\quad + \frac{1}{n + N} \log \sum_{ \substack{ y \in f^{- n - N}(x_{n + N}) \\ \deltameasure[n + N]{p} \in \mathcal{G} } }  w_{n + N}(y) \myexp{ S_{n + N}\potential(y) }.
        \end{split}
    \]
    Note that as $n \to +\infty$, the first term of the right hand side in the equation above converges to $- P(f, \potential)$ by Proposition~\ref{prop:characterization of pressure iterated preimages arbitrary weight}.
    Combining this with \eqref{eq:temp:prop:lower bound for fundamental open sets:iterated preimages:lower bound by approximation}, we get\[
        \liminf_{n \to +\infty} \frac{1}{n} \log \Omega_{n}(x_{n})(\mathcal{G}) 
        = \liminf_{n \to +\infty} \frac{1}{n + N} \log \Omega_{n + N}(x_{n + N})(\mathcal{G}) 
        \geqslant \freeenergy(\mu) - \varepsilon.
    \]
    Then by letting $\varepsilon \to 0$, the desired inequality follows.

    \smallskip
        
    The proof is complete.
\end{proof}

We now prove the general case.

\begin{proof}[Proof of Proposition~\ref{prop:lower bound for fundamental open sets}]
    Let $\mu \in \mathcal{G}$ be arbitrary.
    We may assume without loss of generality that $\mu \in \invmea$ since $\freeenergy(\mu) = -\infty$ when $\mu \notin \invmea$.

    By Lemma~\ref{lem:invariant_Jordan_curve}, we can find a sufficiently high iterate $\widehat{f} \define f^{K}$ of $f$ that has an $\widehat{f}$-invariant Jordan curve $\mathcal{C} \subseteq S^{2}$ with $\post{\widehat{f}} = \post{f} \subseteq \mathcal{C}$. 
    Then $\widehat{f}$ is also an expanding Thurston map (recall Remark~\ref{rem:Expansion_is_independent}). 

    \def\itevecfun{\vecfun[\Phi]}  \def\itevecavg{K\vecavg}  \def\itepotential{\widehat{\potential}}  \def\itefreeenergy{\widehat{F}_{\itepotential}}  \def\iteequstate{\widehat{\mu}_{\itepotential}} \def\iteinvmea{\mathcal{M}(S^2, \widehat{f})}

    Denote $\itevecfun \define S_{K}^{f} \vecfun$ and $\itepotential \define S_K^f \potential$.
    We define $\itefreeenergy \colon \probsphere \mapping [-\infty, 0]$ by\[
        \itefreeenergy(\nu) \define 
        \begin{cases}
            h_{\nu}(\widehat{f}) + \int\! \itepotential \,\mathrm{d}\nu - P(\widehat{f}, \itepotential) & \mbox{if } \nu \in \iteinvmea; \\
            -\infty & \mbox{if } \nu \in \probsphere \setminus \iteinvmea.
        \end{cases}
    \]
    Note that $P(\widehat{f}, \itepotential) = K P(f, \potential)$, $h_{\mu}(\widehat{f}) = K h_{\mu}(f)$, and $\int \! \itepotential \,\mathrm{d}\mu = K \int \! \potential \,\mathrm{d}\mu$ (recall \eqref{eq:def:topological pressure} and \eqref{eq:measure-theoretic entropy well-behaved under iteration}).
    Then we have $\itefreeenergy(\mu) = K \freeenergy(\mu)$ since $\mu \in \invmea \subseteq \iteinvmea$.
    Let $\iteequstate$ be the unique equilibrium state for the map $\widehat{f}$ and the potential $\itepotential$ (recall Theorem~\ref{thm:properties of equilibrium state}~\ref{item:thm:properties of equilibrium state:existence and uniqueness}).
    Since $P_{\mu_{\potential}}(\widehat{f}, \itepotential) = K P_{\mu_{\potential}}(f, \potential) = K P(f, \potential) = P(\widehat{f}, \itepotential)$, it follows from the uniqueness of the equilibrium state that $\iteequstate = \mu_{\potential}$.

    \def\iteopenset{\widehat{\mathcal{G}}_{\varepsilon}}

    Let $\varepsilon > 0$ be such that $\int \! \vecfun \,\mathrm{d}\mu > \vecavg + \varepsilon$.
    This implies $\int \! \itevecfun \,\mathrm{d}\mu = K \int \! \vecfun \,\mathrm{d} \mu > K \vecavg + K \varepsilon$.
    Let $\iteopenset \subseteq \probsphere$ be the open set defined by\[
        \iteopenset \define \biggl\{ \nu \in \probsphere \describe \int \! \itevecfun \,\mathrm{d} \nu > K \vecavg + K \varepsilon \biggr\}.
    \]
    Then we have $\mu \in \iteopenset$.

    We split the proof into three cases according to the type of the sequence $\sequen{\xi_{n}}$.

    \smallskip
    \emph{Case 1 (Birkhoff averages):} $\xi_{n} = \birkhoffmeasure = (V_{n})_{*}(\mu_{\potential})$ for each $n \in \n$ (recall \eqref{eq:def:delta measure for orbit}).
    \smallskip

    \def\itebirkhoffmeasure#1{\widehat{\Sigma}_{#1}}

    For each integer $k \in \zeroton[K - 1]$ and each integer $m \in \n$ that satisfies $(\norm{\vecavg} + \norm{\vecfun}) / m < \varepsilon$, we have \[
        \begin{split}
            \bigl\{ x \in S^2 \describe S_{mK + k}^{f}\vecfun(x) > (mK + k) \vecavg \bigr\} 
            &\supseteq \bigl\{ x \in S^2 \describe S_{mK}^{f} \vecfun(x) > (mK + k) \vecavg + k \norm{\vecfun} \bigr\} \\
            &\supseteq \bigl\{ x \in S^2 \describe S_{mK}^{f}\vecfun(x) > mK \vecavg + K (\norm{\vecavg} + \norm{\vecfun}) \bigr\} \\
            &\supseteq \bigl\{ x \in S^2 \describe m^{-1} S_m^{\widehat{f}}\itevecfun(x) > K \vecavg + K \varepsilon \bigr\}.
        \end{split}
    \]
    For each $n \in \n$, we set $m \define \lfloor n / K \rfloor$ and write $n = mK + k$ for some integer $k \in \zeroton[K - 1]$.
    Then for each sufficiently large $n \in \n$, we have
    \begin{align*}
        \frac{1}{n} \log \birkhoffmeasure(\mathcal{G}) 
            &= \frac{1}{n} \log \mu_{\potential} \bigl( \bigl\{ x \in S^2 \describe S_{n}^{f} \vecfun(x) > n \vecavg \bigr\} \bigr)  \\
            &\geqslant \frac{1}{n} \log \iteequstate \bigl( \bigl\{ x \in S^2 \describe m^{-1} S_m^{\widehat{f}}\itevecfun(x) > K \vecavg + K \varepsilon \bigr\} \bigr) \\
            &= \frac{1}{n} \log \itebirkhoffmeasure{m}(\iteopenset)  \\
            &\geqslant \frac{1}{m K} \log \itebirkhoffmeasure{m}(\iteopenset),
    \end{align*}
    
    where $\sequen[\big]{\itebirkhoffmeasure{j}}[j]$ is defined by replacing $f$ with $\widehat{f}$ and $\potential$ with $\itepotential$ in the definition of $\sequen{\birkhoffmeasure[j]}[j]$.
    Since $\widehat{f}$ has an $\widehat{f}$-invariant Jordan curve $\mathcal{C} \subseteq S^{2}$ with $\post{\widehat{f}} \subseteq \mathcal{C}$, Proposition~\ref{prop:lower bound for fundamental open sets} holds for $\widehat{f}$.
    Therefore, \[
        \liminf_{n \to +\infty} \frac{1}{n} \log \birkhoffmeasure(\mathcal{G}) 
        \geqslant \frac{1}{K} \liminf_{m \to +\infty} \frac{1}{m} \log \itebirkhoffmeasure{m}(\iteopenset) 
        \geqslant \frac{1}{K} \itefreeenergy(\mu) = \freeenergy(\mu).
    \]

    \smallskip
    \emph{Case 2 (Periodic points):} $\xi_{n} = \Omega_{n}$ for each $n \in \n$ (recall \eqref{eq:def:Periodic points distribution}). 
    \smallskip

    \def\ptperiodic{p(k, \widehat{p})} 
    \def\tileperiodic{X^{mK}(\widehat{p})} 
    \def\itedeltameasure{\widehat{V}_{m}(\widehat{p})}

    For each $m \in \n$, it follows from Proposition~\ref{prop:properties cell decompositions}~\ref{item:prop:properties cell decompositions:iterate of cell decomposition} that $\mathbf{X}^{m}(\widehat{f}, \mathcal{C}) = \Tile{mK}$.
    Since $\widehat{f}(\mathcal{C}) \subseteq \mathcal{C}$, by Proposition~\ref{prop:cell decomposition: invariant Jordan curve} and Lemma~\ref{lem:good location n-tile has high level periodic point and preimage point}~\ref{item:lem:good location n-tile has high level periodic point and preimage point:periodic point}, there exists a constant $N \in \n$ depending only on $f$ and $\mathcal{C}$ such that for each integer $\ell \geqslant N$, each $m \in \n$, and each $X^{mK} \in \Tile{mK}$, there exists a fixed point of $f^{mK + \ell}$ in $\inte{X^{mK}}$.
    
    For each $m \in \n$, each $k \in \zeroton[K - 1]$, and each $\widehat{p} \in \periodorbit[m][\widehat{f}]$, let $\tileperiodic \in \Tile{mK}$ be an $mK$-tile that contains $\widehat{p}$ and let $\ptperiodic$ be a fixed point of $f^{mK + k + N}$ in $\inte{\tileperiodic}$.
    By \cite[Lemma~6.3]{li2015weak}, there exists $N_0 \in \n$ such that for each integer $n \geqslant N_{0}$ and each $n$-tile $X^n \in \Tile{n}$, the number of fixed points of $f^{n}$ contained in $X^{n}$ is at most $1$.
    This implies that for each integer $m \geqslant N_{0} / K$ and each $k \in \zeroton[K - 1]$, the map $\widehat{p} \mapsto \ptperiodic$ from $\periodorbit[m][\widehat{f}]$ to $\periodorbit[mK + k + N]$ is injective.

    We claim that there exists $n_{0} \in \n$ such that for each integer $m \geqslant n_{0}$, each $k \in \zeroton[K - 1]$, and each $\widehat{p} \in \periodorbit[m][\widehat{f}]$ with $\itedeltameasure \in \iteopenset$, it follows that $\deltameasure[mK + k + N]{\ptperiodic} \in \mathcal{G}$, where we define $\widehat{V}_{\ell}(x) \define \frac{1}{\ell} \sum_{i = 0}^{\ell - 1} \delta_{\widehat{f}^{i}(x)}$ for each $\ell \in \n$ and each $x \in S^2$.
    Indeed, by Lemma~\ref{lem:distortion lemma for continuous function}, there exists a sufficiently large $n_{0} \in \n$ such that for each integer $m \geqslant n_{0}$,
    \[
        D_{mK}(\vecfun) + (K + N)(\norm{\vecavg} + \norm{\vecfun}) \leqslant mK \varepsilon.
    \]
    Since $X^{mK}(\widehat{p})$ contains $\widehat{p}$ and $\ptperiodic$, we have $S_{mK}^{f} \vecfun(\ptperiodic) \geqslant S_{mK}^{f} \vecfun(\widehat{p}) - D_{mK}(\vecfun)$. 
    Note that $\itedeltameasure \in \iteopenset$ means that $m^{-1} S_{m}^{\widehat{f}} \itevecfun(\widehat{p}) = m^{-1} S_{mK}^{f} \vecfun(\widehat{p}) > K \vecavg + K \varepsilon$.
    Therefore, \[
        \begin{split}
            S_{mK + k + N}^{f} \vecfun(\ptperiodic) 
            &\geqslant S_{mK}^{f} \vecfun(\ptperiodic) - (K + N) \norm{\vecfun}  \\
            &\geqslant S_{mK}^{f} \vecfun(\widehat{p}) - D_{mK}(\vecfun) - (K + N) \norm{\vecfun} \\
            &> mK \vecavg + mK \varepsilon - D_{mK}(\vecfun) - (K + N) \norm{\vecfun}  \\
            &\geqslant (mK + k + N)\vecavg + mK \varepsilon - D_{mK}(\vecfun) - (K + N)( \norm{\vecavg} + \norm{\vecfun} )  \\
            &\geqslant (mK + k + N)\vecavg.
        \end{split}
    \]
    This implies $\deltameasure[mK + k + N]{\ptperiodic} \in \mathcal{G}$.
    
    We now prove the lower bound.    

    For each integer $n \geqslant N$, we set $m \define \lfloor (n - N) / K \rfloor$ and write $n = mK + k + N$ for some integer $k \in \zeroton[K - 1]$.
    By the arguments above, for each sufficiently large $n \in \n$ that satisfies $m \geqslant \max \{ N_0 / K, n_0 \}$, we have
    \[
        \begin{split}
            \sum_{ \substack{ p \in \periodorbit[n] \\ \deltameasure[n]{p} \in \mathcal{G} } } w_{n}(p) \myexp[\big]{ S_{n}^{f} \potential(p) } 
            &\geqslant \sum_{ \substack{ p \in \periodorbit[mK + k + N] \\ \deltameasure[mK + k + N]{p} \in \mathcal{G} } }  \myexp[\big]{ S_{mK + k + N}^{f} \potential(p) } \\
            &\geqslant \sum_{ \substack{ \widehat{p} \in \periodorbit[m][\widehat{f}] \\ \itedeltameasure \in \iteopenset } } \myexp[\big]{ S_{mK + k + N}^{f} \potential(\ptperiodic) },
        \end{split}
    \]
    where $\{ w_{j} \}_{j \in \n}$ is an arbitrary sequence of real-valued functions on $S^2$ with $w_{j}(x) \in \bigl[ 1, \deg_{f^{j}}(x) \bigr]$ for each $j \in \n$ and each $x \in S^{2}$.
    Then by Lemma~\ref{lem:distortion_lemma}, 
    \[
        \begin{split}
            \sum_{ \substack{ p \in \periodorbit[n] \\ \deltameasure[n]{p} \in \mathcal{G} } } w_{n}(p) \myexp[\big]{ S_{n}^{f} \potential(p) } 
            &\geqslant e^{-(K + N)\uniformnorm{\potential}} \sum_{ \substack{ \widehat{p} \in \periodorbit[m][\widehat{f}] \\ \itedeltameasure \in \iteopenset } } \myexp[\big]{ S_{mK}^{f} \potential(\ptperiodic) }  \\
            &\geqslant e^{-C} \sum_{ \substack{ \widehat{p} \in \periodorbit[m][\widehat{f}] \\ \itedeltameasure \in \iteopenset } } \myexp[\big]{ S_{mK}^{f} \potential(\widehat{p}) }
            = e^{-C} \sum_{ \substack{ \widehat{p} \in \periodorbit[m][\widehat{f}] \\ \itedeltameasure \in \iteopenset } } \myexp[\big]{ S_{m}^{\widehat{f}} \itepotential(\widehat{p}) },
        \end{split}
    \]
    where $C \define (K + N)\uniformnorm{\potential} + \Cdistortion$ and $C_{1} \geqslant 0$ is the constant defined in \eqref{eq:const:C_1} in Lemma~\ref{lem:distortion_lemma} that depends only on $f$, $\mathcal{C}$, $d$, $\phi$, and $\holderexp$.
    Thus by \eqref{eq:def:Periodic points distribution}, we have
    \begin{align*}
        \frac{1}{n} \log \Omega_{n}(\mathcal{G}) 
            &= \frac{1}{n} \log \sum_{ \substack{ p \in \periodorbit[n] \\ \deltameasure[n]{p} \in \mathcal{G} } } w_{n}(p) \myexp[\big]{ S_{n}^{f} \potential(p) }
                - \frac{1}{n} \log \sum_{ p' \in \periodorbit } w_{n}(p') \myexp[\big]{ S_{n}^{f} \potential(p') } \\
            &\geqslant \frac{1}{n} \log \sum_{ \substack{ \widehat{p} \in \periodorbit[m][\widehat{f}] \\ \itedeltameasure \in \iteopenset } } \myexp[\big]{ S_{m}^{\widehat{f}} \itepotential(\widehat{p}) } 
                - \frac{1}{n} \log \sum_{ p' \in \periodorbit } w_{n}(p') \myexp[\big]{ S_{n}^{f} \potential(p') } - \frac{C}{n} \\
            &= \frac{1}{n} \log \widehat{\Omega}_{m}(\iteopenset) 
                + \frac{1}{n} \log \sum_{ \widehat{p}' \in \periodorbit[m][\widehat{f}] } \myexp[\big]{ S_{m}^{\widehat{f}} \itepotential(\widehat{p}') }  \\
                &\quad - \frac{1}{n} \log \sum_{ p' \in \periodorbit } w_{n}(p') \myexp[\big]{ S_{n}^{f} \potential(p') } - \frac{C}{n},
    \end{align*}
    where $\sequen[\big]{\widehat{\Omega}_{j}}[j]$ is defined by setting $w_{j}(x) = 1$ for each $j \in \n$ and each $x \in S^2$ and replacing $f$ with $\widehat{f}$ and $\potential$ with $\itepotential$ in the definition of $\sequen{\Omega_{j}}[j]$ (recall \eqref{eq:def:Periodic points distribution}).
    Since $\widehat{f}$ has an $\widehat{f}$-invariant Jordan curve $\mathcal{C} \subseteq S^{2}$ with $\post{\widehat{f}} \subseteq \mathcal{C}$, Proposition~\ref{prop:lower bound for fundamental open sets} holds for $\widehat{f}$.
    Therefore, by Proposition~\ref{prop:characterization of pressure weighted periodic points}, we get 
    \[
        \begin{split}
            \liminf_{n \to +\infty} \frac{1}{n} \log \Omega_{n}(\mathcal{G}) 
            &\geqslant \frac{1}{K} \liminf_{m \to +\infty} \frac{1}{m} \log \widehat{\Omega}_{m}(\iteopenset) + \frac{1}{K} P(\widehat{f}, \itepotential) - P(f, \potential)  \\
            &= \frac{1}{K} \liminf_{m \to +\infty} \frac{1}{m} \log \widehat{\Omega}_{m}(\iteopenset) 
            \geqslant \frac{1}{K} \itefreeenergy(\mu) = \freeenergy(\mu).
        \end{split}
    \]

    \smallskip
    \emph{Case 3 (Iterated preimages):} $\xi_{n} = \Omega_{n}(x_{n})$ for each $n \in \n$ (recall \eqref{eq:def:Iterated preimages distribution}), where $\{ x_{n} \}_{n \in \n}$ is an arbitrary sequence of points in $S^{2}$.
    \smallskip

    \def\ptpreimage{y(k, \widehat{y})} 
    \def\tilepreimage{X^{mK}(\widehat{y})} 
    \def\itedeltameasure{\widehat{V}_{m}(\widehat{y})}

    \def\preimageset{f^{-mK - k - N}(x_{mK + k + N})}
    \def\itepreimageset{\widehat{f}^{-m}(x_0)}

    By Lemma~\ref{lem:good location n-tile has high level periodic point and preimage point}~\ref{item:lem:good location n-tile has high level periodic point and preimage point:preimage point}, there exists a constant $N \in \n$ depending only on $f$ and $\mathcal{C}$ such that for each integer $\ell \geqslant N$, each $m \in \n$, and each $X^{mK} \in \Tile{mK}$, there exists a preimage of $x_{mK + \ell}$ under $f^{mK + \ell}$ in $\inte{X^{mK}}$.

    We fix a point $x_0 \in S^{2} \setminus \post{f}$. Note that $\deg_{f^{n}}(y) = 1$ for each $n \in \n$ and each $y \in f^{-n}(x_{0})$.

    For each $m \in \n$, each $k \in \zeroton[K - 1]$, and each $\widehat{y} \in \itepreimageset$, let $\tilepreimage \in \Tile{mK}$ be an $mK$-tile that contains $\widehat{y}$ and let $\ptpreimage$ be a preimage of $x_{mK + k + N}$ under $f^{mK + k + N}$ in $\inte{\tilepreimage}$.
    By Proposition~\ref{prop:properties cell decompositions}~\ref{item:prop:properties cell decompositions:cellular}, for each $n \in \n$ and each $X^n \in \Tile{n}$, $f^{n}|_{X^{n}}$ is a homeomorphism of $X^{n}$ onto $f^{n}(X^n)$.
    This implies that for each $m \in \n$ and each $k \in \zeroton[K - 1]$, the map $\widehat{y} \mapsto \ptpreimage$ from $\itepreimageset$ to $\preimageset$ is injective.

    By the same reasoning as in Case~2, there exists $n_{0} \in \n$ such that for each integer $m \geqslant n_{0}$, each $k \in \zeroton[K - 1]$, and each $\widehat{y} \in \itepreimageset$ with $\itedeltameasure \in \iteopenset$, it follows that $\deltameasure[mK + k + N]{\ptpreimage} \in \mathcal{G}$.
    
    We now prove the lower bound.
    The proof is essentially the same as in Case~2, and we retain this proof for the convenience of the reader.

    For each integer $n \geqslant N$, we set $m \define \lfloor (n - N) / K \rfloor$ and write $n = mK + k + N$ for some integer $k \in \zeroton[K - 1]$.
    By the arguments above, for each sufficiently large $n \in \n$ that satisfies $m \geqslant n_0$, we have
    \[
        \begin{split}
            \sum_{ \substack{ y \in f^{-n}(x_{n}) \\ \deltameasure[n]{y} \in \mathcal{G} } } w_{n}(y) \myexp[\big]{ S_{n}^{f} \potential(y) } 
            &\geqslant \sum_{ \substack{ y \in \preimageset \\ \deltameasure[mK + k + N]{y} \in \mathcal{G} } }  \myexp[\big]{ S_{mK + k + N}^{f} \potential(y) } \\
            &\geqslant \sum_{ \substack{ \widehat{y} \in \itepreimageset \\ \itedeltameasure \in \iteopenset } } \myexp[\big]{ S_{mK + k + N}^{f} \potential(\ptpreimage) },
        \end{split}
    \]
    where $\{ w_{j} \}_{j \in \n}$ is an arbitrary sequence of real-valued functions on $S^2$ with $w_{j}(x) \in \bigl[ 1, \deg_{f^{j}}(x) \bigr]$ for each $j \in \n$ and each $x \in S^{2}$.
    Then by Lemma~\ref{lem:distortion_lemma}, 
    \[
        \begin{split}
            \sum_{ \substack{ y \in f^{-n}(x_{n}) \\ \deltameasure[n]{y} \in \mathcal{G} } } w_{n}(y) \myexp[\big]{ S_{n}^{f} \potential(y) } 
            &\geqslant e^{-(K + N)\uniformnorm{\potential}}  \sum_{ \substack{ \widehat{y} \in \itepreimageset \\ \itedeltameasure \in \iteopenset } } \myexp[\big]{ S_{mK}^{f} \potential(\ptpreimage) }  \\
            &\geqslant e^{-C} \sum_{ \substack{ \widehat{y} \in \itepreimageset \\ \itedeltameasure \in \iteopenset } } \myexp[\big]{ S_{mK}^{f} \potential(\widehat{y}) }  
            = e^{-C} \sum_{ \substack{ \widehat{y} \in \itepreimageset \\ \itedeltameasure \in \iteopenset } } \myexp[\big]{ S_{m}^{\widehat{f}} \itepotential(\widehat{y}) },
        \end{split}
    \]
    where the constant $C$ is the same as in Case~2.
    Thus by \eqref{eq:def:Iterated preimages distribution}, we have
    \begin{align*}
        \frac{1}{n} \log \Omega_{n}(x_{n})(\mathcal{G}) 
        &= \frac{1}{n} \log \sum_{ \substack{ y \in f^{-n}(x_{n}) \\ \deltameasure[n]{y} \in \mathcal{G} } } w_{n}(y) \myexp[\big]{ S_{n}^{f} \potential(y) } 
            - \frac{1}{n} \log \sum_{ y' \in f^{-n}(x_{n}) } w_{n}(y') \myexp[\big]{ S_{n}^{f} \potential(y') } \\
        &\geqslant \frac{1}{n} \log \sum_{ \substack{ \widehat{y} \in \itepreimageset \\ \itedeltameasure \in \iteopenset } } \myexp[\big]{ S_{m}^{\widehat{f}} \itepotential(\widehat{y}) }
            - \frac{1}{n} \log \sum_{ y' \in f^{-n}(x_{n}) } w_{n}(y') \myexp[\big]{ S_{n}^{f} \potential(y') } - \frac{C}{n} \\
        &= \frac{1}{n} \log \widehat{\Omega}_{m}(x_{0})(\iteopenset) 
            + \frac{1}{n} \log \sum_{ \widehat{y}' \in \itepreimageset } \myexp[\big]{ S_{m}^{\widehat{f}} \itepotential(\widehat{y}') }  \\
            &\quad - \frac{1}{n} \log \sum_{ y' \in f^{-n}(x_{n}) } w_{n}(y') \myexp[\big]{ S_{n}^{f} \potential(y') } - \frac{C}{n},
    \end{align*}
    
    where $\sequen[\big]{\widehat{\Omega}_{j}(x_{0})}[j]$ is defined by setting $w_{j} = \indicator{S^2}$ and $x_{j} = x_{0}$ for each $j \in \n$ and replacing $f$ with $\widehat{f}$ and $\potential$ with $\itepotential$ in the definition of $\sequen{\Omega_{j}(x_{j})}[j]$ (recall \eqref{eq:def:Iterated preimages distribution}).
    Since $\widehat{f}$ has an $\widehat{f}$-invariant Jordan curve $\mathcal{C} \subseteq S^{2}$ with $\post{\widehat{f}} \subseteq \mathcal{C}$, Proposition~\ref{prop:lower bound for fundamental open sets} holds for $\widehat{f}$.
    Therefore, by Proposition~\ref{prop:characterization of pressure iterated preimages arbitrary weight}, we get 
    \[
        \begin{split}
            \liminf_{n \to +\infty} \frac{1}{n} \log \Omega_{n}(x_{n})(\mathcal{G}) 
            &\geqslant \frac{1}{K} \liminf_{m \to +\infty} \frac{1}{m} \log \widehat{\Omega}_{m}(x_{0})(\iteopenset) + \frac{1}{K} P(\widehat{f}, \itepotential) - P(f, \potential)  \\
            &= \frac{1}{K} \liminf_{m \to +\infty} \frac{1}{m} \log \widehat{\Omega}_{m}(x_{0})(\iteopenset)
            \geqslant \frac{1}{K} \itefreeenergy(\mu) = \freeenergy(\mu).
        \end{split}
    \]

    \smallskip
        
    The proof is complete.
\end{proof}

\subsubsection{End of proof of the lower bound}%
\label{ssub:End of proof of the lower bound}

\begin{proof}[Proof of Proposition~\ref{prop:lower bound for open sets}]
    Let $\mathcal{G}$ be a non-empty open subset of $\probsphere$.
    Since subsets of $\probsphere$ of the form $\bigl\{ \mu \in \probsphere \describe \int \! \vecfun \,\mathrm{d}\mu > \vecavg \bigr\}$ with $\ell \in \n$, $\vecfun \in \multispace$, $\vecavg \in \real^{\ell}$ constitute a base of the weak$^{*}$-topology of $\probsphere$, we can write $\mathcal{G}$ as a union $\mathcal{G} = \bigcup_{\lambda} \mathcal{G}_{\lambda}$ of sets of this form.
    For each $\mathcal{G}_{\lambda}$, it follows from Proposition~\ref{prop:lower bound for fundamental open sets} that\[
        \liminf_{n \to +\infty} \frac{1}{n} \log \xi_{n}(\mathcal{G_{\lambda}}) \geqslant \sup_{ \mathcal{G}_{\lambda} } \freeenergy,
    \] 
    for each sequence $\sequen{\xi_{n}} \in \bigl\{ \sequen{\birkhoffmeasure}, \sequen{\Omega_{n}}, \sequen{\Omega_{n}(x_{n})} \bigr\}$.
    Then by Remark~\ref{rem:rate function lower semi-continuous regularization}, we get \[
        \liminf_{n \to +\infty} \frac{1}{n} \log \xi_{n}(\mathcal{G}) \geqslant \sup_{\lambda} \sup_{ \mathcal{G}_{\lambda} } \freeenergy = \sup_{ \mathcal{G} } \freeenergy = - \inf_{\mathcal{G}} \ratefun
    \]
    and complete the proof.
\end{proof}

\subsection{Large deviation upper bound}%
\label{sub:Large deviation upper bound}

In this subsection, we prove the upper bound \eqref{eq:level-2 LDP:upper bound} for all closed sets, with the main result being Proposition~\ref{prop:upper bound for closed sets}.
Based on a preliminary result in Section~\ref{ssub:Construction of suitable invariant measures}, we prove upper bounds for certain fundamental closed subsets of $\probsphere$ in Section~\ref{ssub:Upper bound for fundamental closed sets}.
Finally, in Section~\ref{ssub:End of proof of the upper bound} we establish Proposition~\ref{prop:upper bound for closed sets}.  

\begin{proposition}    \label{prop:upper bound for closed sets}
    Let $f$, $d$, $\phi$ satisfy the Assumptions in Section~\ref{sec:The Assumptions}.
    Then for each sequence $\sequen{\xi_{n}} \in \bigl\{ \sequen{\birkhoffmeasure}, \sequen{\Omega_{n}}, \sequen{\Omega_{n}(x_{n})} \bigr\}$ (as defined in Theorem~\ref{thm:level-2 large deviation principle}), we have
    \[
        \limsup_{n \to +\infty} \frac{1}{n} \log \xi_{n}(\mathcal{K}) \leqslant - \inf_{\mathcal{K}} \ratefun  \quad \text{for all closed } \mathcal{K} \subseteq \probsphere,
    \]
    where $\ratefun \colon \probsphere \mapping [0, +\infty]$ is defined in \eqref{eq:def:rate function}. 
\end{proposition}

\subsubsection{Construction of suitable invariant measures}%
\label{ssub:Construction of suitable invariant measures}

We use the notations as introduced in the beginning of Section~\ref{sec:Entropy density}.

\def \pair#1{ \mathbf{P}^n(#1) }  \def\unionpair#1{P^n(#1)} 
\begin{definition}    \label{def:pair for vector observation} 
    Let $f$, $\mathcal{C}$, $e^0$ satisfy the Assumptions in Section~\ref{sec:The Assumptions}. 
    Consider $\ell \in \n$ and $\vecfun \in \multispace$.    
    For each integer $n \in \n$ and each $\vecavg \in \real^{\ell}$ we define
    \[
        \pair{\vecavg} \define \bigl\{ P^n \in \mathbf{P}^n(f,\mathcal{C}, e^0) \describe \exists x \in P^n \text{ s.t. } n^{-1} S_n \vecfun(x) \geqslant \vecavg \bigr\}.
    \]
    Here $\mathbf{P}^n(f,\mathcal{C}, e^0)$ is the set of $n$-pairs (recall Definition~\ref{def:pair structures}).
\end{definition}

We first establish a generalization of \cite[Lemma~7.15~(ii)]{shi2023thermodynamic}. 

\begin{lemma}    \label{lem:estimate for Birkhoff sum on pair for vector}
    Let $f$, $\mathcal{C}$, $e^0$ satisfy the Assumptions in Section~\ref{sec:The Assumptions}.
    Consider $\ell \in \n$, $\vecfun \in \multispace$, and $\vecavg \in \real^{\ell}$.
    Then for each $n \in \n$ and each $x \in \bigcup \pair{\vecavg}$ we have $S_n \vecfun(x) \geqslant  n \vecavg - 2 D_{n}(\vecfun)$.
    Here $D_{n}(\vecfun)$ is defined in \eqref{eq:def:distortion of potential on tiles}.
\end{lemma}
\begin{proof}
    Let $n \in \n$ and $P^n =  X^n_{\black} \cup X^n_{\white} \in \pair{\vecavg}$ be arbitrary. 
    It suffices to show that $S_n \vecfun(x) \geqslant  n \vecavg - 2 D_{n}(\vecfun)$ for each $x \in P^n$.
    By Definition~\ref{def:pair structures}, there exists $e^{n} \in \Edge{n}$ with $f^{n}(e^{n}) = e^{0}$ such that $e^n \subseteq X^n_{\black} \cap X^n_{\white}$.
    We fix an arbitrary point $x_{e} \in e^n$. 
    By the definition of $\pair{\vecavg}$, there exists $x_{0} \in P^n = X^n_{\black} \cup X^n_{\white}$ such that $\frac{1}{n} S_n \vecfun(x_{0}) \geqslant \vecavg$.
    Since $x_e \in X^n_{\black} \cap X^n_{\white}$, we have $S_n \vecfun(x_e) \geqslant S_n \vecfun(x_{0}) - D_{n}(\vecfun)$.
    Then for each $x \in P^n = X^n_{\black} \cup X^n_{\white}$,
    \begin{equation*}
        S_n\vecfun(x) \geqslant S_n\vecfun(x_e) - D_{n}(\vecfun)
        \geqslant S_n \vecfun(x_{0}) - 2D_{n}(\vecfun) 
        \geqslant n \vecavg - 2 D_{n}(\vecfun).
    \end{equation*}
    The proof is complete.
\end{proof}

The following lemma is analog to \cite[Proposition~7.16]{shi2023thermodynamic}. 
The proof is essentially the same, and we retain this proof for the convenience of the reader.

\begin{proposition} \label{prop:construction of suitable invariant measures for vector continuous function}
    Let $f$, $\mathcal{C}$, $d$, $\potential$, $\holderexp$, $\mu_{\potential}$, $e^{0}$ satisfy the Assumptions in Section~\ref{sec:The Assumptions}.
    We assume in addition that $f(\mathcal{C}) \subseteq \mathcal{C}$.
    Consider $n \in \n$, $\ell \in \n$, $\vecfun \in \multispace$, and $\vecavg \in \real^{\ell}$.
    Suppose that for each $\colour \in \colours$ there exists $P^{n}_{\colour} \in \pair{\vecavg}$ such that $P^{n}_{\colour} \subseteq \inte[\big]{X^{0}_{\colour}}$.
    Then there exists a measure $\mu \in \invmea$ such that
    \[
        \int \! \vecfun \,\mathrm{d}\mu \geqslant \vecavg - \frac{2 D_{n}(\vecfun) }{n}     \quad \text{ and }\quad
        \mu_{\phi} \parentheses[\Big]{ \bigcup \pair{\vecavg} }  \leqslant C \myexp{ (P_{\mu}(f, \phi) - P(f,\phi))n},
    \] 
    where $D_{n}(\vecfun)$ is defined in \eqref{eq:def:distortion of potential on tiles} and $C = 2C_{\mu_{\phi}}\myexp[\big]{ \Cdistortion }$.
    Here $C_{\mu_{\phi}}$ is the constant from Proposition~\ref{prop:equilibrium state is gibbs measure} and $C_1 \geqslant 0$ is the constant defined in \eqref{eq:const:C_1} in Lemma~\ref{lem:distortion_lemma}.
\end{proposition}
\begin{proof}
    Denote $\unionpair{\vecavg} \define \bigcup \pair{\vecavg}$.
    Note that the subsystem $F \define f^{n}|_{ \unionpair{\vecavg} } \in \subsystem[f^{n}]$ is strongly primitive (recall Definition~\ref{def:primitivity of subsystem}).
    We set $\limitmap \define F|_{\limitset}$, where $\limitset \define \limitset(F, \mathcal{C})$ is the tile maximal invariant set associated with $F$ with respect to $\mathcal{C}$.
    Then it follows from Propositions~\ref{prop:subsystem:preliminary properties}~\ref{item:subsystem:properties:limitset forward invariant} and \cite[Proposition~5.20]{shi2023thermodynamic} that $F(\limitset) \subseteq \limitset$ and $\limitset \setminus \mathcal{C} \ne \emptyset$.

    Let $y_{0} \in \limitset \setminus \mathcal{C}$ be arbitrary. 
    By \cite[Theorem~1.1 and Proposition~6.20]{shi2023thermodynamic}, we have
    \begin{equation}     \label{eq:temp:prop:construction of suitable invariant measures for vector continuous function:variational principle}
    \begin{split}
        \sup_{\nu \in \mathcal{M}(\limitset, \limitmap)}  \biggl\{ h_{\nu}(\limitmap) + \int \! S_n\phi \,\mathrm{d}\nu \biggr\} 
        = \lim_{m \to +\infty} \frac{1}{m} \log \!\!\! \sum_{x\in \limitmap^{-m}(y_0)} \!\!\! e^{\sum\limits_{k=0}^{m-1}S_{n}\phi(f^{nk}(x))}.
    \end{split}
    \end{equation}
    For the summand inside the logarithm in \eqref{eq:temp:prop:construction of suitable invariant measures for vector continuous function:variational principle}, we have
    \begin{align}     \label{eq:temp:prop:construction of suitable invariant measures for vector continuous function:summation of preimages}
        \sum_{x\in \limitmap^{-m}(y_0)} \!\!\! e^{\sum\limits_{k = 0}^{m - 1}S_{n}\phi( f^{nk}(x) )}
        &= \sum_{x\in \limitmap^{-m}(y_0)} \!\!\!\!\! e^{S_n\phi (f^{n(m-1)}(x) )} \cdots e^{S_n\phi (f^{n}(x))} e^{S_n\phi(x)}   \notag   \\
        &= \sum_{y_1\in \limitmap^{-1}(y_0)} \!\!\!\!\! e^{S_n\phi(y_1)} \sum_{y_2\in \limitmap^{-1}(y_1)} \!\!\!\!\! e^{S_n\phi(y_2)} \cdots \sum_{y_m\in \limitmap^{-1}(y_{m-1})} \!\!\!\!\!\!\!\!\! e^{S_n\phi(y_m)}.
    \end{align}

    \smallskip
    
    \emph{Claim.} For each point $y \in \limitset \setminus \mathcal{C}$, we have $\card[\big]{\limitmap^{-1}(y)} = \card{\pair{\vecavg}}$, and each $n$-pair $P^n \in \pair{\vecavg}$ contains exactly one preimage $x \in \limitmap^{-1}(y)$, which satisfies $x \in \limitset \setminus \mathcal{C}$.
    
    \smallskip

    To establish this Claim, we consider an arbitrary point $y\in \limitset \setminus \mathcal{C}$. 
    Without loss of generality we may assume that $y \in \inte[\big]{X^0_{\black}}$. Then by Proposition~\ref{prop:properties cell decompositions}, we have $\card{f^{-n}(y)} = (\deg f)^n = \card{\mathbf{X}^n_{\black}}$, and each black $n$-tile $X^n_{\black} \in \mathbf{X}^n_{\black}$ contains exactly one preimage $x \in f^{-n}(y)$, which satisfies $x \in \inte{X^n_{\black}}$. 
    Thus each $n$-pair $P^n \in \pair{\vecavg}$ contains exactly one preimage $x \in f^{-n}(y) \cap \unionpair{\vecavg}$, which satisfies $x \in \inte{P^n}$, and we have\[
        \card{f^{-n}(y) \cap \unionpair{\vecavg}} = \card{\pair{\vecavg}}.
    \]
    Let preimage $x \in f^{-n}(y) \cap \unionpair{\vecavg}$ be arbitrary. Noting that $f^{-n}(y) \cap \unionpair{\vecavg} = \bigl( f^{n}|_{\unionpair{\vecavg}} \bigr)^{-1}(y) = F^{-1}(y)$ and $y \in \limitset \setminus \mathcal{C}$, by Proposition~\ref{prop:subsystem:preliminary properties}~\ref{item:subsystem:properties invariant Jordan curve:backward invariant limitset outside invariant Jordan curve}, we have $x \in \limitset \setminus \mathcal{C}$. 
    Since $\limitmap^{-1}(y) = f^{-n}(y) \cap \limitset = f^{-n}(y) \cap \unionpair{\vecavg}$, the claim follows.

    \smallskip

    By the claim, we know that all the preimages $y_i$ in the summation in \eqref{eq:temp:prop:construction of suitable invariant measures for vector continuous function:summation of preimages} belong to $\limitset \setminus \mathcal{C}$. 
    Moreover, for each point $y \in \limitset \setminus \mathcal{C}$, every $n$-pair $P^n \in \pair{\vecavg}$ contains exactly one preimage $x \in \limitmap^{-1}(y)$, and every preimage $x \in \limitmap^{-1}(y)$ is contained in a unique $n$-pair $P^n \in \pair{\vecavg}$. Thus we get the first two inequalities of the following:

    \begin{align*}
        &\sum_{x\in \limitmap^{-m}(y_0)} \myexp[\bigg]{ \sum_{k=0}^{m-1}S_{n}\phi\bigl( f^{nk}(x) \bigr) }  \\
        &\qquad=\sum_{y_1\in \limitmap^{-1}(y_0)} e^{S_n\phi(y_1)} \sum_{y_2\in \limitmap^{-1}(y_1)} e^{S_n\phi(y_2)} \cdots \sum_{y_m\in \limitmap^{-1}(y_{m-1})} e^{S_n\phi(y_m)}  \\
        &\qquad\geqslant  \biggl( \inf_{y \in \limitset \setminus \mathcal{C}} \sum_{x\in \limitmap^{-1}(y)} e^{S_n\phi(x)} \biggr)^{m}   \\
        &\qquad\geqslant \biggl( \sum_{P^n \in \pair{\vecavg}} \inf_{x \in P^n} e^{S_n\phi(x)} \biggr)^m \\
        &\qquad\geqslant \biggl( \frac{e^{P(f,\phi)n}}{2C_{\mu_{\phi}}e^{\Cdistortion }} \sum_{P^n \in \pair{\vecavg}} \mu_{\phi}(P^n) \biggr)^{m}  \\
        &\qquad= \biggl( \frac{e^{P(f,\phi)n}}{2C_{\mu_{\phi}}e^{\Cdistortion }} \mu_{\phi} (\unionpair{\vecavg}) \biggr)^{m}.
    \end{align*}
    The last inequality follows from \cite[Lemmma~7.15~(i)]{shi2023thermodynamic} and the last equality follows from Lemma~\ref{lem:pairs are disjoint} and Theorem~\ref{thm:properties of equilibrium state}~\ref{item:thm:properties of equilibrium state:edge measure zero}, where $C_1 \geqslant 0$ is the constant defined in \eqref{eq:const:C_1} in Lemma~\ref{lem:distortion_lemma}.
    Taking logarithms of both sides, dividing by $m$, and plugging the result into the previous inequality, we get
    \begin{equation*}
        \begin{split}
            \lim_{m \to +\infty} & \frac{1}{m} \log \biggl( \sum_{x\in \limitmap^{-m}(y_0)}  \myexp[\Big]{ \sum_{k=0}^{m-1}S_{n}\phi\bigl( f^{nk}(x) \bigr) } \biggr) \\
            &\geqslant \log\bigl(\mu_{\phi}(\unionpair{\vecavg})\bigr) + P(f, \phi)n - \bigl( \Cdistortion + \log(2 C_{\mu_{\phi}}) \bigr).
        \end{split}
    \end{equation*}
    Plugging this inequality into \eqref{eq:temp:prop:construction of suitable invariant measures for vector continuous function:variational principle} yields
    \begin{equation}    \label{eq:temp:prop:construction of suitable invariant measures for vector continuous function:supremum measure-theoretic pressure inequality} 
        \sup_{\nu \in \mathcal{M}(\limitset, \limitmap)} \biggl\{ h_{\nu}(\limitmap) + \int \! S_n\phi \,\mathrm{d}\nu \biggr\}  
        \geqslant \log \mathopen{} \bigl(\mu_{\phi}(\unionpair{\vecavg})\bigr) + P(f, \phi)n - \bigl( \Cdistortion + \log(2 C_{\mu_{\phi}}) \bigr).
    \end{equation}
    
    By Theorem~\ref{thm:existence uniqueness and properties of equilibrium state}, there exists an equilibrium state $\widehat{\mu} \in \mathcal{M}(\limitset, \limitmap) \subseteq \mathcal{M}(S^{2}, f^{n})$ which attains the supremum in \eqref{eq:temp:prop:construction of suitable invariant measures for vector continuous function:supremum measure-theoretic pressure inequality}. 
    Denote $\mu \define \frac{1}{n} \sum_{i = 0}^{n - 1} f_{*}^{i}  \widehat{\mu}$. 
    Then $\mu \in \mathcal{M}(S^2, f)$ and we have\[
        \int \! \phi \,\mathrm{d}\mu 
        = \frac{1}{n} \int \sum_{i = 0}^{n - 1} \phi \,\mathrm{d} f_{*}^{i} \widehat{\mu}
        = \frac{1}{n} \int \sum_{i = 0}^{n - 1} \phi \circ f^{i} \,\mathrm{d}\widehat{\mu} 
        = \frac{1}{n} \int \! S_n\phi \,\mathrm{d}\widehat{\mu}.
    \]
    By Lemma~\ref{lem:estimate for Birkhoff sum on pair for vector}, we have $S_n \vecfun(x) \geqslant n\vecavg - 2 D_{n}(\vecfun)$ for each $x \in \unionpair{\vecavg}$.
    Noting that $\supp{\widehat{\mu}} \subseteq \limitset \subseteq \unionpair{\vecavg}$, we get\[
        \int \! \vecfun \,\mathrm{d}\mu = \frac{1}{n}\int \! S_n \vecfun \,\mathrm{d}\widehat{\mu} \geqslant \vecavg - \frac{2 D_{n}(\vecfun) }{n}.
    \]
    By Lemma~\ref{lem:entropy of push-forward average}, we have
    \[
        n h_{\mu}(f) = h_{\widehat{\mu}}(f^{n}) =  h_{\widehat{\mu}}(\limitmap).
    \]
    Then
    \begin{equation*}
        \begin{split}
            n\left( h_{\mu}(f) + \int \! \phi \,\mathrm{d}\mu \right) 
            &= h_{\widehat{\mu}}(\limitmap) + \int \! S_n\phi \,\mathrm{d}\widehat{\mu} \\
            &\geqslant \log \mathopen{}\bigl(\mu_{\phi}(\unionpair{\vecavg})\bigr) + P(f, \phi)n - \bigl( \Cdistortion + \log(2 C_{\mu_{\phi}}) \bigr),
        \end{split}
    \end{equation*}
    i.e., $\log (\mu_{\phi}(\unionpair{\vecavg})) \leqslant \left(P_{\mu}(f, \phi) - P(f, \phi) \right)n + \Cdistortion + \log(2 C_{\mu_{\phi}})$.
    The proof is complete.
\end{proof}

\subsubsection{Upper bound for fundamental closed sets}%
\label{ssub:Upper bound for fundamental closed sets}

We first prove the following result under the additional assumption that there exists an $f$-invariant Jordan curve $\mathcal{C} \subseteq S^2$ with $\post{f} \subseteq \mathcal{C}$ and then for the general case.

\begin{proposition}    \label{prop:upper bound for fundamental closed sets}
    Let $f$, $\mathcal{C}$, $d$, $\phi$, $\holderexp$ satisfy the Assumptions in Section~\ref{sec:The Assumptions}.
    Consider $\ell \in \n$, $\vecfun \in \multispace$, and $\vecavg \in \real^{\ell}$.
    Let $\mathcal{K} \subseteq \probsphere$ be a non-empty closed set of the form\[
        \mathcal{K} \define \biggl\{ \mu \in \probsphere \describe \int \! \vecfun \,\mathrm{d}\mu \geqslant \vecavg \biggr\}.
    \]
    Then for each $\varepsilon > 0$ and each sequence $\sequen{\xi_{n}} \in \bigl\{ \sequen{\birkhoffmeasure}, \sequen{\Omega_{n}}, \sequen{\Omega_{n}(x_{n})} \bigr\}$ (as defined in Theorem~\ref{thm:level-2 large deviation principle}), we have
    \begin{equation}    \label{eq:prop:upper bound for fundamental closed sets:upper bound for fundamental closed sets}
        \limsup_{n \to +\infty} \frac{1}{n} \log \xi_{n}(\mathcal{K}) \leqslant \sup \biggl\{ \freeenergy(\mu) \describe \mu \in \probsphere, \, \int \! \vecfun \,\mathrm{d}\mu > \vecavg - \varepsilon \biggr\},
    \end{equation}
    where $\freeenergy \colon \probsphere \mapping [-\infty, 0]$ is defined in \eqref{eq:def:free energy}.
\end{proposition}
\begin{proof}[Proof of Proposition~\ref{prop:upper bound for fundamental closed sets} under an additional assumption]
    We assume in addition that there exists an $f$-invariant Jordan curve $\mathcal{C} \subseteq S^2$ with $\post{f} \subseteq \mathcal{C}$.

    Let $\varepsilon > 0$ and $\sequen{\xi_{n}} \in \bigl\{ \sequen{\birkhoffmeasure}, \sequen{\Omega_{n}}, \sequen{\Omega_{n}(x_{n})} \bigr\}$ be arbitrary.
    We may assume without loss of generality that the set $\{ n \in \n \describe \xi_{n}(\mathcal{K}) > 0 \}$ is unbounded, because otherwise \eqref{eq:prop:upper bound for fundamental closed sets:upper bound for fundamental closed sets} holds trivially.
    Then it follows from the definition of $\sequen{\birkhoffmeasure}$, $\sequen{\Omega_{n}}$, and $\sequen{\Omega_{n}(x_{n})}$ that for each $n_{0} \in \n$ there exists an integer $n \geqslant n_{0}$ and a point $x \in S^{2}$ such that $S_{n} \vecfun(x) \geqslant n\vecavg$.

    We first show that there exists an ergodic measure $\mu_{0} \in \invmea$ such that $\int \! \vecfun \,\mathrm{d}\mu_{0} \geqslant \vecavg - \varepsilon / 4$.
    Let $n_f \in \n$ be the constant from Definition~\ref{def:primitivity of subsystem}, which depends only on $f$ and $\mathcal{C}$.
    Then by Proposition~\ref{prop:cell decomposition: invariant Jordan curve} and Lemma~\ref{lem:good location n-tile has high level periodic point and preimage point}~\ref{item:lem:good location n-tile has high level periodic point and preimage point:periodic point}, for each $n \in \n_0$ and each $X^n \in \Tile{n}$, there exists a fixed point of $f^{n + n_{f}}$ in $\inte{X^n}$.
    By Lemma~\ref{lem:distortion lemma for continuous function}, there exists a sufficiently large $n_{0} \in \n$ such that for each integer $n \geqslant n_{0}$,
    \[
        \frac{n_{f}(\norm{\vecavg} + \norm{\vecfun}) + D_{n}(\vecfun)}{n + n_{f}} \leqslant \frac{\varepsilon}{4}.
    \]
    Then by the argument in the beginning of the proof, there exists an integer $n \geqslant n_{0}$ and a point $x \in S^{2}$ such that $S_{n} \vecfun(x) \geqslant n\vecavg$.
    We pick an $n$-tile $X^n \in \Tile{n}$ such that $x \in X^n$.
    Thus there exists a fixed point $p \in X^n$ of $f^{n + n_{f}}$.
    Noting that $S_n \vecfun(p) \geqslant S_n \vecfun(x) - D_{n}(\vecfun)$, we have\[
        \begin{split}
            S_{n + n_{f}} \vecfun(p) 
            &\geqslant S_{n} \vecfun(p) - n_{f} \norm{\vecfun} \geqslant S_n \vecfun(x) - n_{f} \norm{\vecfun} - D_{n}(\vecfun) \\
            &\geqslant n \vecavg - n_{f} \norm{\vecfun} - D_{n}(\vecfun)
            \geqslant (n + n_{f}) \vecavg - n_{f} \norm{\vecavg} - n_{f} \norm{\vecfun} - D_{n}(\vecfun).
        \end{split}
    \]
    This implies \[
        \int \! \vecfun \,\mathrm{d}\deltameasure[n + n_{f}]{p} = \frac{1}{n + n_{f}} S_{n + n_{f}} \vecfun(p) \geqslant \vecavg - \frac{n_{f}(\norm{\vecavg} + \norm{\vecfun}) + D_{n}(\vecfun)}{n + n_{f}}
        \geqslant \vecavg - \frac{\varepsilon}{4}.
    \]
    Set $\mu_{0} \define \deltameasure[n + n_{f}]{p}$. 
    Then $\mu_{0} \in \invmea$ is an ergodic measure with $\int \! \vecfun \,\mathrm{d}\mu_{0} \geqslant \vecavg - \varepsilon / 4$.

    Fix an arbitrary $0$-edge $e^{0} \in \Edge{0}$.
    By Lemma~\ref{lem:pair in the interior}, there exists $N \in \n$ such that for each integer $n \geqslant N$ and each $\colour \in \colours$, there exists $P^{n}_{\colour} \in \Pair{n}$ such that $P^{n}_{\colour} \subseteq \inte[\big]{X^0_{\colour}}$ and \[
        \sup_{x \in P^{n}_{\colour}} \norm[\bigg]{ \frac{1}{n} S_n \vecfun(x) - \int \! \vecfun \,\mathrm{d} \mu_{0} } \leqslant \frac{\varepsilon}{4}.
    \]
    Then for each $x \in P^{n}_{\colour}$,\[
        \frac{1}{n} S_n \vecfun(x) \geqslant \int \! \vecfun \,\mathrm{d} \mu_{0} - \frac{\varepsilon}{4} \geqslant \vecavg - \frac{\varepsilon}{2}. 
    \]
    This implies that for each integer $n \geqslant N$ and each $\colour \in \colours$, there exists $P^{n}_{\colour} \in \pair{\vecavg - \varepsilon / 2}$ such that $P^{n}_{\colour} \subseteq \inte[\big]{X^0_{\colour}}$.
    Therefore, it follows from Proposition~\ref{prop:construction of suitable invariant measures for vector continuous function} that for each integer $n \geqslant N$, there exists a measure $\mu_{n} \in \invmea$ such that
    \begin{equation}    \label{eq:temp:prop:upper bound for fundamental closed sets:bound for invariant measure}
        \int \! \vecfun \,\mathrm{d}\mu_{n} \geqslant \vecavg - \frac{\varepsilon}{2} - \frac{2 D_{n}(\vecfun) }{n}     \quad \text{ and }\quad
        \mu_{\potential} ( \unionpair{\vecavg - \varepsilon / 2} ) \leqslant C \myexp{ \freeenergy(\mu_{n}) n },
    \end{equation}
    where $C = 2C_{\mu_{\potential}} \myexp[\big]{ \Cdistortion }$.
    Here $C_{\mu_{\potential}}$ is the constant from Proposition~\ref{prop:equilibrium state is gibbs measure} and $C_1 \geqslant 0$ is the constant defined in \eqref{eq:const:C_1} in Lemma~\ref{lem:distortion_lemma} that depends only on $f$, $\mathcal{C}$, $d$, $\phi$, and $\holderexp$.

    We split the rest of the proof into three cases according to the type of the sequence $\sequen{\xi_{n}}$.

    \def \approximation #1{ \mathbf{T}^{#1} }
    \smallskip
    \emph{Case 1 (Birkhoff averages):} $\xi_{n} = \birkhoffmeasure = (V_{n})_{*}(\mu_{\potential})$ for each $n \in \n$ (recall \eqref{eq:def:delta measure for orbit}).
    \smallskip

    For each integer $n \geqslant N$, since $\{ x \in S^2 \describe \deltameasure{x} \in \mathcal{K} \} \subseteq \unionpair{\vecavg - \varepsilon / 2}$, we have
    \[
            \frac{1}{n} \log \birkhoffmeasure(\mathcal{K}) 
            = \frac{1}{n} \log \mu_{\potential} ( \{ x \in S^2 \describe \deltameasure{x} \in \mathcal{K} \} ) 
            \leqslant \frac{1}{n} \log \mu_{\potential}( \unionpair{\vecavg - \varepsilon / 2} ).
    \]
    Thus by \eqref{eq:temp:prop:upper bound for fundamental closed sets:bound for invariant measure} and Lemma~\ref{lem:distortion lemma for continuous function}, for each sufficiently large integer $n \geqslant N$ that satisfies $2 D_{n}(\vecfun) / n < \varepsilon / 2$,
    \[
        \frac{1}{n} \log \birkhoffmeasure(\mathcal{K}) 
        \leqslant \freeenergy(\mu_{n}) + \frac{\log C}{n}  
        \leqslant \sup \biggl\{ \freeenergy(\mu) \describe \mu \in \probsphere, \, \int \! \vecfun \,\mathrm{d}\mu > \vecavg - \varepsilon \biggr\} + \frac{\log C}{n}.
    \]    
    Then by letting $n \to +\infty$, the desired inequality follows.

    \smallskip
    \emph{Case 2 (Periodic points):} $\xi_{n} = \Omega_{n}$ for each $n \in \n$ (recall \eqref{eq:def:Periodic points distribution}). 
    \smallskip

    For each $n \in \n$ and each $p \in \periodorbit[n]$, let $X^n(p) \in \Tile{n}$ be an $n$-tile that contains $p$.
    In particular, if $p \in \periodorbit[n]$ satisfies $\deltameasure{p} \in \mathcal{K}$, then $X^{n}(p) \subseteq \unionpair{\vecavg} \subseteq \unionpair{\vecavg - \varepsilon / 2}$.
    By \cite[Lemma~6.3]{li2015weak}, there exists $N_0 \in \n$ such that for each integer $n \geqslant N_{0}$ and each $n$-tile $X^n \in \Tile{n}$, the number of fixed points of $f^{n}$ contained in $X^{n}$ is at most $1$.
    This implies that for each integer $n \geqslant N_{0}$, the map $p \mapsto X^n(p)$ from $\periodorbit[n]$ to $\Tile{n}$ is injective.

    \def\tileflower{\mathbf{X}^{n}(f, \mathcal{C}, p')}
    Let $n \in \n$ be arbitrary. 
    Consider $p' \in \periodorbit[n] \cap \vertex{n}$, where $\vertex{n} = \Vertex{n}$ is the set of $n$-vertices.
    We set $\tileflower \define \{X \in \Tile{n} \describe p' \in X \}$.
    By Remark~\ref{rem:flower structure}, we have $\cflower{n}{p'} = \bigcup \tileflower$ and $\card{\tileflower} = 2 \deg_{f^{n}}(p')$, where $\flower{n}{p'}$ is defined in \eqref{eq:n-flower} and $\cflower{n}{p'}$ is the closure of $\flower{n}{p'}$.
    In particular, if $p' \in \periodorbit[n]$ satisfies $\deltameasure{p} \in \mathcal{K}$, then $\cflower{n}{p'} = \bigcup \tileflower \subseteq \unionpair{\vecavg} \subseteq \unionpair{\vecavg - \varepsilon / 2}$.
    Moreover, if $n \geqslant N_0$, then $X^n(p) \notin \tileflower$ for every $p \in \periodorbit[n] \setminus \Vertex{n}$.

    We are now ready to establish the desired upper bound.
    Let $\{ w_{j} \}_{j \in \n}$ be an arbitrary sequence of real-valued functions on $S^2$ with $w_{j}(x) \in \bigl[ 1, \deg_{f^{j}}(x) \bigr]$ for each $j \in \n$ and each $x \in S^{2}$.
    For each integer $n \geqslant N_0$, we have
    \begin{align*}
        &\sum_{ \substack{ p \in \periodorbit[n] \\ \deltameasure[n]{p} \in \mathcal{K} } }  w_{n}(p) \myexp{ S_{n}\potential(p) } \\
        &\qquad \leqslant \sum_{ \substack{ p' \in \periodorbit[n] \cap \vertex{n} \\ \deltameasure[n]{p'} \in \mathcal{K} } }  \deg_{f^{n}}(p') \myexp{ S_{n}\potential(p') }
            + \sum_{ \substack{ p \in \periodorbit[n] \setminus \vertex{n} \\ \deltameasure[n]{p} \in \mathcal{K} } }  \myexp{ S_{n}\potential(p) }  \\
        &\qquad \leqslant \sum_{ \substack{ p' \in \periodorbit[n] \cap \vertex{n} \\ \deltameasure[n]{p'} \in \mathcal{K} } }  \sum_{ X^{n} \in \tileflower } \myexp{ S_{n}\potential(X^n) }
            + \sum_{ \substack{ p \in \periodorbit[n] \setminus \vertex{n} \\ \deltameasure[n]{p} \in \mathcal{K} } }  \myexp{ S_{n}\potential(X^{n}(p)) }  \\
        &\qquad \leqslant \sum_{ \substack{ X^n \subseteq \unionpair{\vecavg - \varepsilon / 2} \\ X^n \in \Tile{n} } } \myexp{ S_{n}\potential(X^{n}) },
    \end{align*}
    where we write $S_{n}\potential(X^n) \define \sup_{x \in X^n} S_{n}\potential(x)$ for each $n \in \n$ and each $X^n \in \Tile{n}$.
    Then it follows from Proposition~\ref{prop:equilibrium state is gibbs measure} and Theorem~\ref{thm:properties of equilibrium state}~\ref{item:thm:properties of equilibrium state:edge measure zero} that for each integer $n \geqslant N_0$,
    \[
        \begin{split}
            \sum_{ \substack{ p \in \periodorbit[n] \\ \deltameasure[n]{p} \in \mathcal{K} } }  w_{n}(p) \myexp{ S_{n}\potential(p) }
            &\leqslant C_{\mu_{\potential}} e^{n P(f, \potential)} \sum_{ \substack{ X^n \subseteq \unionpair{\vecavg - \varepsilon / 2} \\ X^n \in \Tile{n} } }  \mu_{\potential}(X^{n})  \\
            &= C_{\mu_{\potential}} e^{n P(f, \potential)} \mu_{\potential}(\unionpair{\vecavg - \varepsilon / 2}).
        \end{split}
    \]
    By \eqref{eq:temp:prop:upper bound for fundamental closed sets:bound for invariant measure} and Lemma~\ref{lem:distortion lemma for continuous function}, for each sufficiently large integer $n \geqslant \max \{ N_0, N \}$ that satisfies $2 D_{n}(\vecfun) / n < \varepsilon / 2$,
    \[
        \begin{split}
            \frac{1}{n} \log \mu_{\potential}(\unionpair{\vecavg - \varepsilon / 2})
            &\leqslant \freeenergy(\mu_{n}) + \frac{ \log C }{n}  \\
            &\leqslant \sup \biggl\{ \freeenergy(\mu) \describe \mu \in \probsphere, \, \int \! \vecfun \,\mathrm{d}\mu > \vecavg - \varepsilon \biggr\} + \frac{\log C}{n},
        \end{split}
    \]
    and therefore
    \[
        \begin{split}
            \frac{1}{n} \log \Omega_{n}(\mathcal{K}) 
            &\leqslant - \frac{1}{n} \log \sum_{ p \in \periodorbit[n] } w_{n}(p) \myexp{ S_{n}\potential(p) }  \\
            &\quad + \sup \biggl\{ \freeenergy(\mu) \describe \mu \in \probsphere, \, \int \! \vecfun \,\mathrm{d}\mu > \vecavg - \varepsilon \biggr\} + P(f, \potential) + \frac{ \log (C C_{\mu_{\potential}}) }{n}.
        \end{split}
    \]
    Note that as $n \to +\infty$, the first term of the right hand side in the equation above converges to $- P(f, \potential)$ by Proposition~\ref{prop:characterization of pressure weighted periodic points}.
    Therefore, by letting $n \to +\infty$, the desired inequality follows.

    \smallskip
    \emph{Case 3 (Iterated preimages):} $\xi_{n} = \Omega_{n}(x_{n})$ for each $n \in \n$ (recall \eqref{eq:def:Iterated preimages distribution}), where $\{ x_{n} \}_{n \in \n}$ is an arbitrary sequence of points in $S^{2}$.
    \smallskip

    \def\preimage{f^{-n}(x_{n})}
    For each $n \in \n$ and each $y \in \preimage$, let $X^n(y) \in \Tile{n}$ be an $n$-tile that contains $y$.
    In particular, if $y \in \preimage$ satisfies $\deltameasure{y} \in \mathcal{K}$, then $X^{n}(y) \subseteq \unionpair{\vecavg} \subseteq \unionpair{\vecavg - \varepsilon / 2}$.
    By Proposition~\ref{prop:properties cell decompositions}~\ref{item:prop:properties cell decompositions:cellular}, for each $n \in \n$ and each $X^n \in \Tile{n}$, $f^{n}|_{X^{n}}$ is a homeomorphism of $X^{n}$ onto $f^{n}(X^n)$.
    This implies that for each integer $n \in \n$ the map $y \mapsto X^n(y)$ from $\preimage$ to $\Tile{n}$ is injective.

    Let $\{ w_{j} \}_{j \in \n}$ be an arbitrary sequence of real-valued functions on $S^2$ with $w_{j}(x) \in \bigl[ 1, \deg_{f^{j}}(x) \bigr]$ for each $j \in \n$ and each $x \in S^{2}$.
    By similar arguments as in Case~2, for each sufficiently large integer $n \geqslant N$ that satisfies $2 D_{n}(\vecfun) / n < \varepsilon / 2$, \[
        \begin{split}
            &\frac{1}{n} \log \sum_{ \substack{ y \in \preimage \\ \deltameasure[n]{y} \in \mathcal{K} } }  w_{n}(y) \myexp{ S_{n}\potential(y) } \\
            &\qquad \leqslant \sup \biggl\{ \freeenergy(\mu) \describe \mu \in \probsphere, \, \int \! \vecfun \,\mathrm{d}\mu > \vecavg - \varepsilon \biggr\} + P(f, \potential) + \frac{ \log (C C_{\mu_{\potential}}) }{n},
        \end{split}
    \]
    and therefore
    \[
        \begin{split}
            \frac{1}{n} \log \Omega_{n}(x_{n})(\mathcal{K}) 
            &\leqslant - \frac{1}{n} \log \sum_{ y \in \preimage } w_{n}(y) \myexp{ S_{n}\potential(y) }  \\
            &\quad + \sup \biggl\{ \freeenergy(\mu) \describe \mu \in \probsphere, \, \int \! \vecfun \,\mathrm{d}\mu > \vecavg - \varepsilon \biggr\} + P(f, \potential) + \frac{ \log (C C_{\mu_{\potential}}) }{n}.
        \end{split}
    \]
    Note that as $n \to +\infty$, the first term of the right hand side in the equation above converges to $- P(f, \potential)$ by Proposition~\ref{prop:characterization of pressure iterated preimages arbitrary weight}.
    Therefore, by letting $n \to +\infty$, the desired inequality follows.

    \smallskip
        
    The proof is complete.
\end{proof}

We now prove the general case.
\begin{proof}[Proof of Proposition~\ref{prop:upper bound for fundamental closed sets}]
    By Lemma~\ref{lem:invariant_Jordan_curve}, we can find a sufficiently high iterate $\widehat{f} \define f^{K}$ of $f$ that has an $\widehat{f}$-invariant Jordan curve $\mathcal{C} \subseteq S^{2}$ with $\post{\widehat{f}} = \post{f} \subseteq \mathcal{C}$. 
    Then $\widehat{f}$ is also an expanding Thurston map (recall Remark~\ref{rem:Expansion_is_independent}). 

    \def\itevecfun{\vecfun[\Phi]}  \def\itevecavg{K\vecavg}  \def\itepotential{\widehat{\potential}}  \def\itefreeenergy{\widehat{F}_{\itepotential}}  \def\iteequstate{\widehat{\mu}_{\itepotential}}  \def\iteinvmea{\mathcal{M}(S^2, \widehat{f})}

    Denote $\itevecfun \define S_{K}^{f} \vecfun$ and $\itepotential \define S_K^f \potential$.
    We define $\itefreeenergy \colon \probsphere \mapping [-\infty, 0]$ by\[
        \itefreeenergy(\nu) \define 
        \begin{cases}
            h_{\nu}(\widehat{f}) + \int\! \itepotential \,\mathrm{d}\nu - P(\widehat{f}, \itepotential) & \mbox{if } \nu \in \iteinvmea; \\
            -\infty & \mbox{if } \nu \in \probsphere \setminus \iteinvmea.
        \end{cases}
    \]
    Note that for each $\mu \in \invmea$, we have $P(\widehat{f}, \itepotential) = K P(f, \potential)$, $h_{\mu}(\widehat{f}) = K h_{\mu}(f)$, and $\int \! \itepotential \,\mathrm{d}\mu = K \int \! \potential \,\mathrm{d}\mu$ (recall \eqref{eq:def:topological pressure} and \eqref{eq:measure-theoretic entropy well-behaved under iteration}).
    Then for each $\mu \in \invmea$, we have $\itefreeenergy(\mu) = K \freeenergy(\mu)$ since $\invmea \subseteq \iteinvmea$.
    Let $\iteequstate$ be the unique equilibrium state for the map $\widehat{f}$ and the potential $\itepotential$ (recall Theorem~\ref{thm:properties of equilibrium state}~\ref{item:thm:properties of equilibrium state:existence and uniqueness}).
    Since $P_{\mu_{\potential}}(\widehat{f}, \itepotential) = K P_{\mu_{\potential}}(f, \potential) = K P(f, \potential) = P(\widehat{f}, \itepotential)$, it follows from the uniqueness of the equilibrium state that $\iteequstate = \mu_{\potential}$.

    \def\itecloseset#1{\widehat{\mathcal{K}}_{#1}}
    
    For each $\delta > 0$, let $\itecloseset{\delta} \subseteq \probsphere$ be the closed set defined by\[
        \itecloseset{\delta} \define \biggl\{ \mu \in \probsphere \describe \int \! \itevecfun \,\mathrm{d} \mu \geqslant K \vecavg - K \delta \biggr\}.
    \]

    Let $\varepsilon > 0$ be arbitrary.
    We claim that
    \begin{equation}     \label{eq:temp:prop:upper bound for fundamental closed sets:sup freeenergy equal}
        \sup \biggl\{ \itefreeenergy(\widehat{\mu}) \describe \widehat{\mu} \in \probsphere, \, \int \! \itevecfun \,\mathrm{d}\widehat{\mu} > \itevecavg - K \varepsilon \biggr\}
        = K \sup \biggl\{ \freeenergy(\mu) \describe \mu \in \probsphere, \, \int \! \vecfun \,\mathrm{d}\mu > \vecavg - \varepsilon \biggr\}.
    \end{equation}
    By the definitions of $\freeenergy$ and $\itefreeenergy$, it suffices to show that
    \[
        \sup \biggl\{ \itefreeenergy(\widehat{\mu}) \describe \widehat{\mu} \in \iteinvmea, \, \int \! \itevecfun \,\mathrm{d}\widehat{\mu} > \itevecavg - K \varepsilon \biggr\}
        = K \sup \biggl\{ \freeenergy(\mu) \describe \mu \in \invmea, \, \int \! \vecfun \,\mathrm{d}\mu > \vecavg - \varepsilon \biggr\}.
    \]
    To see this, we consider arbitrary $\widehat{\mu} \in \iteinvmea$ satisfying $\int \! \itevecfun \,\mathrm{d}\widehat{\mu} > \itevecavg - K \varepsilon$.
    Define $\mu \define \frac{1}{K} \sum_{j = 0}^{K - 1} f_{*}^{j} \widehat{\mu}$.
    Then $\int \! \vecfun \,\mathrm{d}\mu = \frac{1}{K} \int \! \itevecfun \,\mathrm{d}\widehat{\mu} > \vecavg - \varepsilon$ and it follows from Lemma~\ref{lem:entropy of push-forward average} that $\mu \in \invmea$ and $h_{\widehat{\mu}}(\widehat{f}) = K h_{\mu}(f)$.
    Thus we have $\itefreeenergy(\widehat{\mu}) = K \freeenergy(\mu)$ and
    \[
        \sup \biggl\{ \itefreeenergy(\widehat{\mu}) \describe \widehat{\mu} \in \iteinvmea, \, \int \! \itevecfun \,\mathrm{d}\widehat{\mu} > \itevecavg - K \varepsilon \biggr\}
        \leqslant K \sup \biggl\{ \freeenergy(\mu) \describe \mu \in \invmea, \, \int \! \vecfun \,\mathrm{d}\mu > \vecavg - \varepsilon \biggr\}.
    \]
    The other direction follows immediately from the facts that $\invmea \subseteq \iteinvmea$ and $K \freeenergy(\mu) = \itefreeenergy(\mu)$ for each $\mu \in \invmea$.

    We split the proof into three cases according to the type of the sequence $\sequen{\xi_{n}}$.

    \smallskip
    \emph{Case 1 (Birkhoff averages):} $\xi_{n} = \birkhoffmeasure = (V_{n})_{*}(\mu_{\potential})$ for each $n \in \n$ (recall \eqref{eq:def:delta measure for orbit}).
    \smallskip

    \def\itebirkhoffmeasure#1{\widehat{\Sigma}_{#1}}

    For each integer $k \in \zeroton[K - 1]$ and each integer $m \in \n$ that satisfies $(\norm{\vecavg} + \norm{\vecfun}) / m \leqslant \varepsilon / 2$, we have \[
        \begin{split}
            \bigl\{ x \in S^2 \describe S_{mK - k}^{f}\vecfun(x) \geqslant (mK - k) \vecavg \bigr\} 
            &\subseteq \bigl\{ x \in S^2 \describe S_{mK}^{f} \vecfun(x) \geqslant (mK - k) \vecavg - k \norm{\vecfun} \bigr\} \\
            &\subseteq \bigl\{ x \in S^2 \describe S_{mK}^{f}\vecfun(x) \geqslant mK \vecavg - K (\norm{\vecavg} + \norm{\vecfun}) \bigr\} \\
            &\subseteq \bigl\{ x \in S^2 \describe m^{-1} S_m^{\widehat{f}}\itevecfun(x) \geqslant K \vecavg - K \varepsilon / 2 \bigr\}.
        \end{split}
    \]
    For each $n \in \n$, we can write $n = mK - k$ for some integer $k \in \zeroton[K - 1]$ and $m \in \n$.
    Then for each sufficiently large $n \in \n$, we have
    \[
        \begin{split}
            \frac{1}{n} \log \birkhoffmeasure(\mathcal{K}) 
            &= \frac{1}{n} \log \mu_{\potential} \bigl( \bigl\{ x \in S^2 \describe S_{n}^{f} \vecfun(x) \geqslant n \vecavg \bigr\} \bigr)  \\
            &\leqslant \frac{1}{n} \log \iteequstate \bigl( \bigl\{ x \in S^2 \describe m^{-1} S_m^{\widehat{f}}\itevecfun(x) \geqslant K \vecavg - K \varepsilon / 2 \bigr\} \bigr) \\
            &= \frac{1}{n} \log \itebirkhoffmeasure{m}(\itecloseset{\varepsilon / 2})
            \leqslant \frac{1}{m K} \log \itebirkhoffmeasure{m}(\itecloseset{\varepsilon / 2}),
        \end{split}
    \]
    where $m = \lceil n / K \rceil$ and $\sequen[\big]{\itebirkhoffmeasure{j}}[j]$ is defined by replacing $f$ with $\widehat{f}$ and $\potential$ with $\itepotential$ in the definition of $\sequen{\birkhoffmeasure[j]}[j]$.
    Since $\widehat{f}$ has an $\widehat{f}$-invariant Jordan curve $\mathcal{C} \subseteq S^{2}$ with $\post{\widehat{f}} \subseteq \mathcal{C}$, Proposition~\ref{prop:upper bound for fundamental closed sets} holds for $\widehat{f}$.
    Therefore, by \eqref{eq:temp:prop:upper bound for fundamental closed sets:sup freeenergy equal}, we get
    \[
        \begin{split}
            \limsup_{n \to +\infty} \frac{1}{n} \log \birkhoffmeasure(\mathcal{K}) 
            &\leqslant \frac{1}{K} \limsup_{m \to +\infty} \frac{1}{m} \log \itebirkhoffmeasure{m}(\itecloseset{\varepsilon / 2})  \\ 
            &\leqslant \frac{1}{K} \sup \biggl\{ \itefreeenergy(\widehat{\mu}) \describe \widehat{\mu} \in \probsphere, \, \int \! \itevecfun \,\mathrm{d}\widehat{\mu} > \itevecavg - K \varepsilon \biggr\}  \\
            &= \sup \biggl\{ \freeenergy(\mu) \describe \mu \in \probsphere, \, \int \! \vecfun \,\mathrm{d}\mu > \vecavg - \varepsilon \biggr\}.
        \end{split}
    \]

    \smallskip
    \emph{Case 2 (Periodic points):} $\xi_{n} = \Omega_{n}$ for each $n \in \n$ (recall \eqref{eq:def:Periodic points distribution}). 
    \smallskip

    \def\pttileperiodic{\widehat{p}(X^n)} 
    \def\iteptperiodic{\widehat{p}(X^n(p))} 
    \def\itedeltameasure#1{\widehat{V}_{m}(#1)}
    \def\mconst{\lceil (n + N) / K \rceil + \lceil N / K \rceil}

    Let $N \define n_{f} \in \n$ be the constant from Definition~\ref{def:primitivity of subsystem}, which depends only on $f$ and $\mathcal{C}$.
    For each $n \in \n$ and each $X^{n} \in \Tile{n}$, it follows from Lemma~\ref{lem:strongly primitive:tile in interior tile for high enough level} that there exists $X^{\ell K} \in \Tile{\ell K}$ such that $X^{\ell K} \subseteq \inte{X^{n}}$, where $\ell \define \lceil (n + N) / K \rceil$.
    Since $\widehat{f}(\mathcal{C}) \subseteq \mathcal{C}$, it follows from Propositions~\ref{prop:properties cell decompositions}~\ref{item:prop:properties cell decompositions:iterate of cell decomposition} and \ref{prop:cell decomposition: invariant Jordan curve} that $X^{\ell K} \subseteq X^{0}_{\colour}$ for some $\colour \in \colours$.
    Define $m \define \ell + \lceil N / K \rceil$.
    Then by Lemma~\ref{lem:good location n-tile has high level periodic point and preimage point}~\ref{item:lem:good location n-tile has high level periodic point and preimage point:periodic point}, there exists a fixed point of $f^{mK}$ in $\inte{X^{\ell K}} \subseteq \inte{X^{n}}$.
    
    For each $n \in \n$ and each $X^{n} \in \Tile{n}$, we fix a fixed point of $f^{mK}$ in $\inte{X^{n}}$ and denote it by $\pttileperiodic$, where $m = \mconst$.
    Then for each $n \in \n$, the map $X^n \mapsto \pttileperiodic$ from $\Tile{n}$ to $\periodorbit[m][\widehat{f}]$ is injective.

    For each $n \in \n$ and each $p \in \periodorbit[n]$, let $X^n(p) \in \Tile{n}$ be an $n$-tile that contains $p$.
    By \cite[Lemma~6.3]{li2015weak}, there exists $N_0 \in \n$ such that for each integer $n \geqslant N_{0}$ and each $n$-tile $X^n \in \Tile{n}$, the number of fixed points of $f^{n}$ contained in $X^{n}$ is at most $1$.
    This implies that for each integer $n \geqslant N_{0}$, the map $p \mapsto X^n(p)$ from $\periodorbit[n]$ to $\Tile{n}$ is injective.        
    Therefore, for each integer $n \geqslant N_{0}$, the map $p \mapping \iteptperiodic$ from $\periodorbit[n]$ to $\periodorbit[m][\widehat{f}]$ is injective, where $m = \mconst$ and $\iteptperiodic \in \inte{X^{n}(p)}$.

    We claim that there exists $n_{0} \in \n$ such that for each integer $n \geqslant n_{0}$, each $p \in \periodorbit[n][f]$ with $\deltameasure[n]{p} \in \mathcal{K}$, and each $X^n \in \Tile{n}$ with $p \in X^{n}$, it follows that $\itedeltameasure{\pttileperiodic} \in \itecloseset{\varepsilon / 2}$, where $m = \mconst$ and $\widehat{V}_{\ell}(x) \define \frac{1}{\ell} \sum_{i = 0}^{\ell - 1} \delta_{\widehat{f}^{i}(x)}$ for each $\ell \in \n$ and each $x \in S^2$.
    Indeed, by Lemma~\ref{lem:distortion lemma for continuous function}, there exists a sufficiently large $n_{0} \in \n$ such that for each integer $n \geqslant n_{0}$,
    \[
        D_{n}(\vecfun) + 2(K + N)(\norm{\vecavg} + \norm{\vecfun}) \leqslant n \varepsilon / 2.
    \]
    Since $X^{n}$ contains $p$ and $\pttileperiodic$, we have $S_{n}^{f} \vecfun(\pttileperiodic) \geqslant S_{n}^{f} \vecfun(p) - D_{n}(\vecfun)$. 
    Set $m \define \mconst$.
    Note that $0 \leqslant mK - n \leqslant 2(N + K)$ and $\deltameasure[n]{p} \in \mathcal{K}$ means that $S_{n}^{f} \vecfun(p) \geqslant n \vecavg$.
    Therefore, 
    \begin{align*}
        S_{m}^{\widehat{f}} \itevecfun(\pttileperiodic) 
        &= S_{mK}^{f} \vecfun(\pttileperiodic)
        \geqslant S_{n}^{f} \vecfun(\pttileperiodic) - 2(K + N) \norm{\vecfun}  \\
        &\geqslant S_{n}^{f} \vecfun(\widehat{p}) - D_{n}(\vecfun) - 2(K + N) \norm{\vecfun} \\
        &\geqslant n \vecavg - D_{n}(\vecfun) - 2(K + N) \norm{\vecfun}  \\
        &\geqslant mK \vecavg - D_{n}(\vecfun) - 2(K + N)( \norm{\vecavg} + \norm{\vecfun} )  \\
        &\geqslant mK \vecavg - n \varepsilon / 2  \\
        &\geqslant mK \vecavg - mK \varepsilon / 2.
    \end{align*}
    This implies $\deltameasure[m]{\pttileperiodic} \in \itecloseset{\varepsilon / 2}$.
 
    \def\tileflower{\mathbf{X}^{n}(f, \mathcal{C}, p')}
    Let $n \in \n$ be arbitrary. 
    Consider $p' \in \periodorbit[n] \cap \vertex{n}$, where $\vertex{n} = \Vertex{n}$ is the set of $n$-vertices.
    We set $\tileflower \define \{X \in \Tile{n} \describe p' \in X \}$.
    By Remark~\ref{rem:flower structure}, we have $\cflower{n}{p'} = \bigcup \tileflower$ and $\card{\tileflower} = 2 \deg_{f^{n}}(p')$, where $\flower{n}{p'}$ is defined in \eqref{eq:n-flower} and $\cflower{n}{p'}$ is the closure of $\flower{n}{p'}$.
    Note that if $n \geqslant N_0$, then $X^n(p) \notin \tileflower$ for every $p \in \periodorbit[n] \setminus \Vertex{n}$.

    We are now ready to establish the desired upper bound.
    Let $\{ w_{j} \}_{j \in \n}$ be an arbitrary sequence of real-valued functions on $S^2$ with $w_{j}(x) \in \bigl[ 1, \deg_{f^{j}}(x) \bigr]$ for each $j \in \n$ and each $x \in S^{2}$.

    For each integer $n \in \n$, we set $m \define \mconst$.
    Note that $0 \leqslant mK - n \leqslant 2(N + K)$.
    By the arguments above, for each integer $n \geqslant \max\{N_0, n_{0}\}$, we have
    \begin{align*}
        &\sum_{ \substack{ p \in \periodorbit[n] \\ \deltameasure[n]{p} \in \mathcal{K} } }  w_{n}(p) \myexp[\big]{ S_{n}^{f}\potential(p) } \\
        &\qquad \leqslant \sum_{ \substack{ p' \in \periodorbit[n] \cap \vertex{n} \\ \deltameasure[n]{p'} \in \mathcal{K} } }  \deg_{f^{n}}(p') \myexp[\big]{ S_{n}^{f}\potential(p') }
            + \sum_{ \substack{ p \in \periodorbit[n] \setminus \vertex{n} \\ \deltameasure[n]{p} \in \mathcal{K} } }  \myexp[\big]{ S_{n}^{f}\potential(p) }  \\
        &\qquad \leqslant \sum_{ \substack{ p' \in \periodorbit[n] \cap \vertex{n} \\ \deltameasure[n]{p'} \in \mathcal{K} } }  \sum_{ X^{n} \in \tileflower } e^{D_{n}(\potential)} \myexp[\big]{ S_{n}^{f}\potential(\pttileperiodic) }  \\
        &\qquad \qquad + \sum_{ \substack{ p \in \periodorbit[n] \setminus \vertex{n} \\ \deltameasure[n]{p} \in \mathcal{K} } }  e^{D_{n}(\potential)} \myexp[\big]{ S_{n}^{f}\potential(\iteptperiodic) }  \\
        &\qquad \leqslant e^{D_{n}(\potential)} \sum_{ \substack{ \widehat{p} \in \periodorbit[m][\widehat{f}] \\ \itedeltameasure{\widehat{p}} \in \itecloseset{\varepsilon / 2} } } \myexp[\big]{ S_{n}^{f} \potential(\widehat{p}) }  \\
        &\qquad \leqslant e^{ D_{n}(\potential) } e^{2(N + K)\uniformnorm{\potential}} \sum_{ \substack{ \widehat{p} \in \periodorbit[m][\widehat{f}] \\ \itedeltameasure{\widehat{p}} \in \itecloseset{\varepsilon / 2} } } \myexp[\big]{ S_{m}^{\widehat{f}} \itepotential(\widehat{p}) }.
    \end{align*}
    Then by Lemma~\ref{lem:distortion_lemma}, 
    \[
        \sum_{ \substack{ p \in \periodorbit[n] \\ \deltameasure[n]{p} \in \mathcal{K} } }  w_{n}(p) \myexp[\big]{ S_{n}^{f}\potential(p) } 
        \leqslant e^{C} \sum_{ \substack{ \widehat{p} \in \periodorbit[m][\widehat{f}] \\ \itedeltameasure{\widehat{p}} \in \itecloseset{\varepsilon / 2} } } \myexp[\big]{ S_{m}^{\widehat{f}} \itepotential(\widehat{p}) },
    \]
    where $C \define 2(N + K)\uniformnorm{\potential} + \Cdistortion$ and $C_{1} \geqslant 0$ is the constant defined in \eqref{eq:const:C_1} in Lemma~\ref{lem:distortion_lemma} that depends only on $f$, $\mathcal{C}$, $d$, $\phi$, and $\holderexp$.
    Thus by \eqref{eq:def:Periodic points distribution}, we have
    \[
        \begin{split}
            \frac{1}{n} \log \Omega_{n}(\mathcal{K}) 
            &= \frac{1}{n} \log \sum_{ \substack{ p \in \periodorbit[n] \\ \deltameasure[n]{p} \in \mathcal{K} } } w_{n}(p) \myexp[\big]{ S_{n}^{f} \potential(p) }
                - \frac{1}{n} \log \sum_{ p' \in \periodorbit } w_{n}(p') \myexp[\big]{ S_{n}^{f} \potential(p') } \\
            &\leqslant \frac{1}{n} \log \sum_{ \substack{ \widehat{p} \in \periodorbit[m][\widehat{f}] \\ \itedeltameasure{\widehat{p}} \in \itecloseset{\varepsilon / 2} } } \myexp[\big]{ S_{m}^{\widehat{f}} \itepotential(\widehat{p}) } 
                - \frac{1}{n} \log \sum_{ p' \in \periodorbit } w_{n}(p') \myexp[\big]{ S_{n}^{f} \potential(p') } + \frac{C}{n} \\
            &= \frac{1}{n} \log \widehat{\Omega}_{m}(\itecloseset{\varepsilon / 2}) 
                + \frac{1}{n} \log \sum_{ \widehat{p}' \in \periodorbit[m][\widehat{f}] } \myexp[\big]{ S_{m}^{\widehat{f}} \itepotential(\widehat{p}') }  \\
                &\quad - \frac{1}{n} \log \sum_{ p' \in \periodorbit } w_{n}(p') \myexp[\big]{ S_{n}^{f} \potential(p') } + \frac{C}{n},
        \end{split}
    \]
    where $\sequen[\big]{\widehat{\Omega}_{j}}[j]$ is defined by setting $w_{j}(x) = 1$ for each $j \in \n$ and each $x \in S^2$ and replacing $f$ with $\widehat{f}$ and $\potential$ with $\itepotential$ in the definition of $\sequen{\Omega_{j}}[j]$ (recall \eqref{eq:def:Periodic points distribution}).
    Since $\widehat{f}$ has an $\widehat{f}$-invariant Jordan curve $\mathcal{C} \subseteq S^{2}$ with $\post{\widehat{f}} \subseteq \mathcal{C}$, Proposition~\ref{prop:upper bound for fundamental closed sets} holds for $\widehat{f}$.
    Therefore, it follows from Proposition~\ref{prop:characterization of pressure weighted periodic points} and \eqref{eq:temp:prop:upper bound for fundamental closed sets:sup freeenergy equal} that 
    \[
        \begin{split}
            \limsup_{n \to +\infty} \frac{1}{n} \log \Omega_{n}(\mathcal{K}) 
            &\leqslant \frac{1}{K} \limsup_{m \to +\infty} \frac{1}{m} \log \widehat{\Omega}_{m}(\itecloseset{\varepsilon / 2}) + \frac{1}{K} P(\widehat{f}, \itepotential) - P(f, \potential)  \\
            &\leqslant \frac{1}{K} \sup \biggl\{ \itefreeenergy(\widehat{\mu}) \describe \widehat{\mu} \in \probsphere, \, \int \! \itevecfun \,\mathrm{d}\widehat{\mu} > \itevecavg - K \varepsilon \biggr\}  \\
            &= \sup \biggl\{ \freeenergy(\mu) \describe \mu \in \probsphere, \, \int \! \vecfun \,\mathrm{d}\mu > \vecavg - \varepsilon \biggr\}.
        \end{split}
    \]

    \smallskip
    \emph{Case 3 (Iterated preimages):} $\xi_{n} = \Omega_{n}(x_{n})$ for each $n \in \n$ (recall \eqref{eq:def:Iterated preimages distribution}), where $\{ x_{n} \}_{n \in \n}$ is an arbitrary sequence of points in $S^{2}$.
    \smallskip

    \def\pttilepreimage{\widehat{y}(X^n)} 
    \def\iteptpreimage{\widehat{y}(X^n(y))} 
    \def\itedeltameasure#1{\widehat{V}_{m}(#1)}
    \def\mconst{\lceil (n + N) / K \rceil}
    
    \def\preimageset{f^{-n}(x_{n})}
    \def\itepreimageset{\widehat{f}^{-m}(x_0)}

    We fix a point $x_0 \in S^{2}$.

    Let $N \define n_{f} \in \n$ be the constant from Definition~\ref{def:primitivity of subsystem}, which depends only on $f$ and $\mathcal{C}$.
    For each $n \in \n$ and each $X^{n} \in \Tile{n}$, it follows from Lemma~\ref{lem:good location n-tile has high level periodic point and preimage point}~\ref{item:lem:good location n-tile has high level periodic point and preimage point:preimage point} that there exists a preimage of $x_{0}$ under $f^{mK}$ in $\inte{X^{n}}$, where $m \define \mconst$.
    We fix a preimage of of $x_{0}$ under $f^{mK}$ in $\inte{X^{n}}$ and denote it by $\pttilepreimage$.
    Then for each $n \in \n$, the map $X^n \mapsto \pttilepreimage$ from $\Tile{n}$ to $\itepreimageset$ is injective.

    For each $n \in \n$ and each $y \in \preimageset$, let $X^n(y) \in \Tile{n}$ be an $n$-tile that contains $y$.    
    By Proposition~\ref{prop:properties cell decompositions}~\ref{item:prop:properties cell decompositions:cellular}, for each $n \in \n$ and each $X^n \in \Tile{n}$, $f^{n}|_{X^{n}}$ is a homeomorphism of $X^{n}$ onto $f^{n}(X^n)$.
    This implies that for each integer $n \in \n$, the map $y \mapsto \iteptpreimage$ from $\preimageset$ to $\itepreimageset$ is injective, where $m = \mconst$ and $\iteptpreimage \in \inte{X^{n}(y)}$.

    By similar arguments as in Case~2, there exists $n_{0} \in \n$ such that for each integer $n \geqslant n_{0}$, each $y \in \preimageset$ with $\deltameasure[n]{y} \in \mathcal{K}$, and each $X^n \in \Tile{n}$ with $y \in X^{n}$, it follows that $\itedeltameasure{\pttilepreimage} \in \itecloseset{\varepsilon / 2}$, where $m = \mconst$.
    
    \def\tileflower{\mathbf{X}^{n}(f, \mathcal{C}, y')}
    Let $n \in \n$ be arbitrary. 
    Consider $y' \in \preimageset \cap \vertex{n}$, where $\vertex{n} = \Vertex{n}$ is the set of $n$-vertices.
    We set $\tileflower \define \{X \in \Tile{n} \describe y' \in X \}$.
    By Remark~\ref{rem:flower structure}, we have $\cflower{n}{y'} = \bigcup \tileflower$ and $\card{\tileflower} = 2 \deg_{f^{n}}(y')$, where $\flower{n}{y'}$ is defined in \eqref{eq:n-flower} and $\cflower{n}{y'}$ is the closure of $\flower{n}{y'}$.
    Note that $X^n(y) \notin \tileflower$ for every $y \in \preimageset \setminus \Vertex{n}$.

    We now prove the upper bound.
    The proof is essentially the same as in Case~2, and we retain this proof for the convenience of the reader.

    Let $\{ w_{j} \}_{j \in \n}$ be an arbitrary sequence of real-valued functions on $S^2$ with $w_{j}(x) \in \bigl[ 1, \deg_{f^{j}}(x) \bigr]$ for each $j \in \n$ and each $x \in S^{2}$.
    For each integer $n \in \n$, we set $m \define \mconst$.
    Note that $0 \leqslant mK - n \leqslant N + K$.
    By the arguments above, for each integer $n \geqslant n_{0}$, we have
    \begin{align*}
        &\sum_{ \substack{ y \in \preimageset \\ \deltameasure[n]{y} \in \mathcal{K} } }  w_{n}(y) \myexp[\big]{ S_{n}^{f}\potential(y) } \\
            &\qquad \leqslant \sum_{ \substack{ y' \in \preimageset \cap \vertex{n} \\ \deltameasure[n]{y'} \in \mathcal{K} } }  w_{n}(y') \myexp[\big]{ S_{n}^{f}\potential(y') }
                + \sum_{ \substack{ y \in \preimageset \setminus \vertex{n} \\ \deltameasure[n]{y} \in \mathcal{K} } }  \myexp[\big]{ S_{n}^{f}\potential(y) }  \\
            &\qquad \leqslant \sum_{ \substack{ y' \in \preimageset \cap \vertex{n} \\ \deltameasure[n]{y'} \in \mathcal{K} } }  \sum_{ X^{n} \in \tileflower } e^{D_{n}(\potential)} \myexp[\big]{ S_{n}^{f}\potential(\pttilepreimage) }  \\
            &\qquad \qquad + \sum_{ \substack{ y \in \preimageset \setminus \vertex{n} \\ \deltameasure[n]{y} \in \mathcal{K} } }  e^{D_{n}(\potential)} \myexp[\big]{ S_{n}^{f}\potential(\iteptpreimage) }  \\
            &\qquad \leqslant e^{D_{n}(\potential)} \sum_{ \substack{ \widehat{y} \in \itepreimageset \\ \itedeltameasure{\widehat{y}} \in \itecloseset{\varepsilon / 2} } } \myexp[\big]{ S_{n}^{f} \potential(\widehat{y}) }  \\
            &\qquad \leqslant e^{ D_{n}(\potential) } e^{ 2(N + K)\uniformnorm{\potential} } \sum_{ \substack{ \widehat{y} \in \itepreimageset \\ \itedeltameasure{\widehat{y}} \in \itecloseset{\varepsilon / 2} } } \myexp[\big]{ S_{m}^{\widehat{f}} \itepotential(\widehat{y}) }.
    \end{align*}
    Then by Lemma~\ref{lem:distortion_lemma}, 
    \[
        \sum_{ \substack{ y \in \preimageset \\ \deltameasure[n]{y} \in \mathcal{K} } }  w_{n}(y) \myexp[\big]{ S_{n}^{f}\potential(y) } 
        \leqslant e^{C} \sum_{ \substack{ \widehat{y} \in \itepreimageset \\ \itedeltameasure{\widehat{y}} \in \itecloseset{\varepsilon / 2} } } \myexp[\big]{ S_{m}^{\widehat{f}} \itepotential(\widehat{y}) },
    \]
    where the constant $C$ is the same as in Case~2.
    Thus by \eqref{eq:def:Iterated preimages distribution}, we have
    \[
        \begin{split}
            \frac{1}{n} \log \Omega_{n}(x_{n})(\mathcal{K}) 
            &= \frac{1}{n} \log \sum_{ \substack{ y \in \preimageset \\ \deltameasure[n]{y} \in \mathcal{K} } } w_{n}(y) \myexp[\big]{ S_{n}^{f} \potential(y) }
                - \frac{1}{n} \log \sum_{ y' \in \preimageset } w_{n}(y') \myexp[\big]{ S_{n}^{f} \potential(y') } \\
            &\leqslant \frac{1}{n} \log \sum_{ \substack{ \widehat{y} \in \itepreimageset \\ \itedeltameasure{\widehat{y}} \in \itecloseset{\varepsilon / 2} } } \myexp[\big]{ S_{m}^{\widehat{f}} \itepotential(\widehat{y}) } 
                - \frac{1}{n} \log \sum_{ y' \in \preimageset } w_{n}(y') \myexp[\big]{ S_{n}^{f} \potential(y') } + \frac{C}{n} \\
            &= \frac{1}{n} \log \widehat{\Omega}_{m}(x_{0})(\itecloseset{\varepsilon / 2}) 
                + \frac{1}{n} \log \sum_{ \widehat{y}' \in \itepreimageset } \myexp[\big]{ S_{m}^{\widehat{f}} \itepotential(\widehat{y}') }  \\
                &\quad - \frac{1}{n} \log \sum_{ y' \in \preimageset } w_{n}(y') \myexp[\big]{ S_{n}^{f} \potential(y') } + \frac{C}{n},
        \end{split}
    \]
    where $\sequen[\big]{\widehat{\Omega}_{j}(x_{0})}[j]$ is defined by setting $w_{j} = \indicator{S^2}$ and $x_{j} = x_{0}$ for each $j \in \n$ and replacing $f$ with $\widehat{f}$ and $\potential$ with $\itepotential$ in the definition of $\sequen{\Omega_{j}(x_{j})}[j]$ (recall \eqref{eq:def:Iterated preimages distribution}).
    Since $\widehat{f}$ has an $\widehat{f}$-invariant Jordan curve $\mathcal{C} \subseteq S^{2}$ with $\post{\widehat{f}} \subseteq \mathcal{C}$, Proposition~\ref{prop:upper bound for fundamental closed sets} holds for $\widehat{f}$.
    Therefore, it follows from Proposition~\ref{prop:characterization of pressure iterated preimages arbitrary weight} and \eqref{eq:temp:prop:upper bound for fundamental closed sets:sup freeenergy equal} that 
    \[
        \begin{split}
            \limsup_{n \to +\infty} \frac{1}{n} \log \Omega_{n}(x_{n})(\mathcal{K}) 
            &\leqslant \frac{1}{K} \limsup_{m \to +\infty} \frac{1}{m} \log \widehat{\Omega}_{m}(x_{0})(\itecloseset{\varepsilon / 2}) + \frac{1}{K} P(\widehat{f}, \itepotential) - P(f, \potential)  \\
            &\leqslant \frac{1}{K} \sup \biggl\{ \itefreeenergy(\widehat{\mu}) \describe \widehat{\mu} \in \probsphere, \, \int \! \itevecfun \,\mathrm{d}\widehat{\mu} > \itevecavg - K \varepsilon \biggr\}  \\
            &= \sup \biggl\{ \freeenergy(\mu) \describe \mu \in \probsphere, \, \int \! \vecfun \,\mathrm{d}\mu > \vecavg - \varepsilon \biggr\}.
        \end{split}
    \]

    \smallskip

    The proof is complete.
\end{proof}

\subsubsection{End of proof of the upper bound}%
\label{ssub:End of proof of the upper bound}

\begin{proof}[Proof of Proposition~\ref{prop:upper bound for closed sets}]
    Let $\mathcal{K}$ be a closed subset of $\probsphere$.
    Let $\mathcal{G} \subseteq \probsphere$ be an open set containing $\mathcal{K}$.
    Since $\probsphere$ is metrizable and compact in the weak$^{*}$-topology (see for example, \cite[Theorems~6.4 and~6.5]{walters1982introduction}) and $\mathcal{K}$ is compact, we can choose $\varepsilon > 0$ and finitely many closed sets $\listings{\mathcal{K}}[s]$ of the form $\mathcal{K}_{j} = \bigl\{ \mu \in \probsphere \describe \int \! \vecfun_{j} \,\mathrm{d}\mu \geqslant \vecavg_{j} \bigr\}$ with $\ell_{j} \in \n$, $\vecfun_{j} \in \multispace[\ell_{j}]$, and $\vecavg_{j} \in \real^{\ell_{j}}$, so that $\mathcal{K} \subseteq \bigcup_{j = 1}^{s} \mathcal{K}_{j} \subseteq \bigcup_{j = 1}^{s} \mathcal{K}_{j}(\varepsilon) \subseteq \mathcal{G}$, where $\mathcal{K}_{j}(\varepsilon) \define \bigl\{ \mu \in \probsphere \describe \int \! \vecfun_{j} \,\mathrm{d}\mu > \vecavg_{j} - \varepsilon \bigr\}$.
    For each $j \in \oneton[s]$, it follows from Proposition~\ref{prop:upper bound for fundamental closed sets} that\[
        \limsup_{n \to +\infty} \frac{1}{n} \log \xi_{n}(\mathcal{K}_{j}) \leqslant \sup_{\mathcal{K}_{j}(\varepsilon)} \freeenergy.
    \]     
    for each sequence $\sequen{\xi_{n}} \in \bigl\{ \sequen{\birkhoffmeasure}, \sequen{\Omega_{n}}, \sequen{\Omega_{n}(x_{n})} \bigr\}$.
    Hence, \[
        \begin{split}
            \limsup_{n \to +\infty} \frac{1}{n} \log \xi_{n}(\mathcal{K}) 
            &\leqslant \limsup_{n \to +\infty} \frac{1}{n} \log \xi_{n} \biggl( \bigcup_{j = 1}^{s} \mathcal{K}_{j} \biggr)  
            \leqslant \max_{1 \leqslant j \leqslant s} \limsup_{n \to +\infty} \frac{1}{n} \log \xi_{n}(\mathcal{K}_{j})  \\
            &\leqslant \max_{1 \leqslant j \leqslant s} \sup_{\mathcal{K}_{j}(\varepsilon)} \freeenergy 
            \leqslant \sup_{\mathcal{G}} \freeenergy.
        \end{split}
    \]
    Since $\mathcal{G}$ is an arbitrary open set containing $\mathcal{K}$, it follows from Remark~\ref{rem:rate function lower semi-continuous regularization} that\[
        \limsup_{n \to +\infty} \frac{1}{n} \log \xi_{n}(\mathcal{K}) 
        \leqslant \inf_{\mathcal{G} \supseteq \mathcal{K}} \sup_{\mathcal{G}} \freeenergy 
        = \inf_{\mathcal{G} \supseteq \mathcal{K}} \sup_{\mathcal{G}} (- \ratefun)
        = - \inf_{\mathcal{K}} \ratefun,
    \]
    where the last equality is due to the lower semi-continuity of $\ratefun$.
\end{proof}

\subsection{Proof of large deviation principles}%
\label{sub:Proof of large deviation principles}

In this subsection, we finish the proof of Theorem~\ref{thm:level-2 large deviation principle} and its corollaries.

We record the following well-known lemma, sometimes known as the \emph{Portmanteau Theorem}, and refer the reader to \cite[Theorem~2.1]{billingsley2013convergence} for a proof.

\begin{lemma} \label{lem:portmanteau theorem}
    Let $(X, d)$ be a compact metric space, and $\mu$ and $\mu_{i}$, for $i \in \n$, be Borel probability measures on $X$.
    Then the following statements are equivalent:
    \begin{enumerate}
        \smallskip

        \item     \label{item:lem:portmanteau theorem:weak converge}
            $\mu_{i} \weakconverge \mu$ as $i \to +\infty$;

        \smallskip

        \item     \label{item:lem:portmanteau theorem:closed set}
            $\limsup\limits_{i \to +\infty} \mu_{i}(E) \leqslant \mu(E)$ for each closed set $E \subseteq X$;
        
        \smallskip

        \item     \label{item:lem:portmanteau theorem:open set}
            $\liminf\limits_{i \to +\infty} \mu_{i}(G) \geqslant \mu(G)$ for each open set $G \subseteq X$;
        
        \smallskip

        \item     \label{item:lem:portmanteau theorem:boundary}
            $\lim\limits_{i \to +\infty} \mu_{i}(B) \leqslant \mu(B)$ for each Borel set $B \subseteq X$ with $\mu(\partial B) = 0$.
    \end{enumerate}
\end{lemma}

\begin{proof}[Proof of Theorem~\ref{thm:level-2 large deviation principle}]
    \def\distribution{\sequen{\xi_{n}}}
    We fix arbitrary sequence $\sequen{\xi_{n}} \in \bigl\{ \sequen{\birkhoffmeasure}, \sequen{\Omega_{n}}, \sequen{\Omega_{n}(x_{n})} \bigr\}$.

    By Propositions~\ref{prop:lower bound for open sets} and~\ref{prop:upper bound for closed sets}, $\distribution$ satisfies a large deviation principle with the rate function $\ratefun$ as defined in \eqref{eq:def:rate function}.

    By Theorem~\ref{thm:uniqueness of the minimizer}, $\mu_{\potential}$ is the unique minimizer of the rate function $\ratefun$.

    To prove that $\distribution$ converges to $\delta_{\mu_{\potential}}$ in the weak$^*$ topology, by Lemma~\ref{lem:portmanteau theorem}~\ref{item:lem:portmanteau theorem:weak converge} and~\ref{item:lem:portmanteau theorem:closed set}, it suffices to show that $\limsup_{n \to +\infty} \xi_{n}(\mathcal{K}) \leqslant 0$ for each closed set $\mathcal{K} \subseteq \probsphere \setminus \{\mu_{\potential}\}$.
    Let $\mathcal{K}$ be a closed set in $\probsphere$ with $\mu_{\potential} \notin \mathcal{K}$.
    Indeed, since $\ratefun$ is lower semi-continuous, non-negative, and it vanishes precisely on $\{\mu_{\potential}\}$ by Theorem~\ref{thm:uniqueness of the minimizer}, the infimum of $\ratefun$ on $\mathcal{K}$ is attained at some point of $\mathcal{K}$, and thus $\inf_{\mathcal{K}} \ratefun > 0$.
    Therefore, it follows immediately from the large deviation upper bounds that
    \[
        \limsup_{n \to +\infty} \frac{1}{n} \log \xi_{n}(\mathcal{K}) \leqslant - \inf_{\mathcal{K}} \ratefun < 0
    \]
    and $\lim_{n \to +\infty} \xi_{n}(\mathcal{K}) = 0$.

    To prove the last statement of the theorem, let $\mathcal{G} \subseteq \probsphere$ be a convex and open set containing an invariant measure $\mu'$.
    Since the rate function $\ratefun$ is lower semi-continuous, and since it takes finite values precisely on the compact set $\invmea$ by \eqref{eq:def:rate function}, there exists $\mu \in \overline{\mathcal{G}} \cap \invmea$ such that $\ratefun(\mu) = \inf_{\overline{\mathcal{G}}} \ratefun$.
    For each $t \in (0, 1)$, put $\mu_{t} \define (1 - t)\mu + t\mu'$, and note that $\mu_{t} \in \invmea$ and $\mu_{t} \in \mathcal{G}$ (see for example, \cite[1.1, p.~38]{schaefer1971locally}).
    Since $\ratefun$ is convex (recall Remark~\ref{rem:rate function lower semi-continuous regularization}), we have
    \[
        \inf_{\mathcal{G}} \ratefun \leqslant \liminf_{t \to 0} \ratefun(\mu_{t}) \leqslant \ratefun(\mu) = \inf_{\overline{\mathcal{G}}} \ratefun.
    \]
    This shows that $\inf_{\mathcal{G}} \ratefun = \inf_{\overline{\mathcal{G}}} \ratefun$.
    Hence, $\mathcal{G}$ is a $\ratefun$-continuity set and the last assertion of the theorem follows immediately from \eqref{eq:property of I-continuity set} in Subsection~\ref{sub:Level-2 large deviation principles}.

    The proof is complete.
\end{proof}

We show that Corollary~\ref{coro:level-1 large deviation principle} follows from Theorem~\ref{thm:level-2 large deviation principle} and the general theory of large deviations.

\begin{proof}[Proof of Corollary~\ref{coro:level-1 large deviation principle}]
    The first assertion follows immediately from Theorem~\ref{thm:level-2 large deviation principle} and the contraction principle (see Subsection~\ref{sub:Level-2 large deviation principles}).

    We now consider an arbitrary interval $J \subseteq \real$ that intersects $(c_{\psi}, d_{\psi})$.
    We denote the rate function defined in \eqref{eq:coro:level-1 large deviation principle:rate function} by $K_{\psi} \colon \real \mapping \real$.
    This function is bounded on $[c_{\psi}, d_{\psi}]$ and constant equal to $+\infty$ on $\real \setminus [c_{\psi}, d_{\psi}]$.
    Furthermore, $K_{\psi}$ is convex on $\real$, and therefore continuous on $(c_{\psi}, d_{\psi})$.
    This implies $\inf_{\interior{J}} K_{\psi} = \inf_{\overline{J}} K_{\psi}$ since $J \cap (c_{\psi}, d_{\psi}) \ne \emptyset$.
    Then \eqref{eq:coro:level-1 large deviation principle:on closed interval} follows from \eqref{eq:level-2 LDP:lower bound} and \eqref{eq:level-2 LDP:upper bound}.
\end{proof}

We now prove the other corollary of Theorem~\ref{thm:level-2 large deviation principle}, as stated in Corollary~\ref{coro:measure-theoretic pressure infimum on local basis}, which gives a characterization of the rate function.

\begin{proof}[Proof of Corollary~\ref{coro:measure-theoretic pressure infimum on local basis}]
    \def\localbasis{G_{\mu}}
    Fix $\mu \in \invmea$ and a convex local basis $\localbasis$ at $\mu$.
    We show that \eqref{eq:coro:measure-theoretic pressure infimum on local basis} in Corollary~\ref{coro:measure-theoretic pressure infimum on local basis} holds.
    Since the rate function $\ratefun$ is lower semi-continuous (recall Remark~\ref{rem:rate function lower semi-continuous regularization}), we get 
    \[
        - \ratefun(\mu) = \inf_{\mathcal{G} \in \localbasis} \sup_{\mathcal{G}} (- \ratefun) = \inf_{\mathcal{G} \in \localbasis} \parentheses[\big]{- \inf_{\mathcal{G}} \ratefun}.
    \]
    Then it follows from \eqref{eq:equalities for rate function} in Theorem~\ref{thm:level-2 large deviation principle} that \[
        \begin{split}
            - \ratefun(\mu)
            &= \inf_{\mathcal{G} \in \localbasis} \set[\bigg]{ \lim_{n \to +\infty} \frac{1}{n} \log \mu_{\potential}(\set{x \in S^{2} \describe \deltameasure[n]{x} \in \mathcal{G}}) } \\
            &= \inf_{\mathcal{G} \in \localbasis} \set[\bigg]{ \lim_{n \to +\infty} \frac{1}{n} \log \sum_{ p \in \periodorbit, \deltameasure[n]{p} \in \mathcal{G} } \frac{w_{n}(p) \myexp{S_{n}\potential(p)}}{Z_{n}(\potential)} } \\
            &= \inf_{\mathcal{G} \in \localbasis} \set[\bigg]{ \lim_{n \to +\infty} \frac{1}{n} \log \sum_{ y \in f^{-n}(x_{n}), \deltameasure[n]{y} \in \mathcal{G} } \frac{w_{n}(y) \myexp{S_{n}\potential(y)}}{Z'_{n}(\potential)} },
        \end{split}
    \]
    where we write $Z_{n}(\potential) \define \sum_{ x \in \periodorbit } w_{n}(x) \myexp{ S_{n}\potential(x) }$ and $Z'_{n}(\potential) \define \sum_{ y \in f^{-n}(x_{n}) } w_{n}(y) \myexp{ S_{n}\potential(y) }$.
    Note that by Propositions~\ref{prop:characterization of pressure weighted periodic points} and \ref{prop:characterization of pressure iterated preimages arbitrary weight} we have $P(f, \potential) = \lim_{n \to +\infty} \frac{1}{n} \log Z_{n}(\potential) = \lim_{n \to +\infty} \frac{1}{n} \log Z'_{n}(\potential)$.
    Thus \eqref{eq:coro:measure-theoretic pressure infimum on local basis} holds.
\end{proof}

\subsection{Equidistribution with respect to the equilibrium state}%
\label{sub:Equidistribution with respect to the equilibrium state}

We finish this section with an equidistribution result, as a consequence of level-$2$ large deviation principles.

\begin{proof}[Proof of Theorem~\ref{thm:equidistribution results}]
    \def\openset{\mathcal{G}} \def\zplus{Z^{+}_{n}(\openset)}   \def\zminus{Z^{-}_{n}(\openset)}    \def\localbasis{G_{\mu_{\potential}}}    
    \def\neighborhood{\openset_{\mu}}  \def\neighborhoodi{\openset_{\mu_{i}}}   \def\zplusneighborhood{Z^{+}_{n}(\neighborhood)}
    Recall that $\deltameasure[n]{x} = \frac{1}{n} \sum_{i = 0}^{n - 1} \delta_{f^{i}(x)}$ for $x \in S^{2}$ and $n \in \n$ as defined in \eqref{eq:def:delta measure for orbit}.
    For each $n \in \n$ and each open set $\openset \subseteq \probmea{S^2}$, we write
    \[
        \begin{split}
            \zplus &\define \sum_{ y \in f^{-n}(x_{n}), \deltameasure[n]{y} \in \mathcal{G} } w_{n}(y) \myexp{S_{n}\potential(y)} \\
            \zminus &\define \sum_{ y \in f^{-n}(x_{n}), \deltameasure[n]{y} \notin \mathcal{G} } w_{n}(y) \myexp{S_{n}\potential(y)}.
        \end{split}
    \]

    Let $\localbasis$ be a convex local basis of $\probmea{S^2}$ at $\mu_{\potential}$.
    We fix an arbitrary convex open set $\openset \in \localbasis$.
 
    Recall that $\mu_{\potential}$ is the unique minimizer of the rate function $\ratefun$ by Theorem~\ref{thm:uniqueness of the minimizer}.
    Then it follows from Corollary~\ref{coro:measure-theoretic pressure infimum on local basis} that for each $\mu \in \probmea{S^2} \setminus \{\mu_{\potential}\}$, there exist numbers $a_{\mu} < P(f, \potential)$ and $N_{\mu} \in \n$ and an open neighborhood $\neighborhood \subseteq \probmea{S^{2}} \setminus \{\mu_{\potential}\}$ containing $\mu$ such that for each $n \geqslant N_{\mu}$,
    \begin{equation}    \label{eq:temp:thm:equidistribution results:upper bound for z plus}
        \zplusneighborhood \leqslant \myexp{n a_{\mu}}.
    \end{equation}
    Since $\probmea{S^{2}}$ is compact in the weak$^{*}$ topology by Alaoglu's theorem, so is $\probmea{S^{2}} \setminus \openset$.
    Thus there exists a finite set $\{ \mu_{i} \describe i \in I \} \subseteq \probmea{S^{2}} \setminus \openset$ (where $I$ is a finite index set) such that
    \begin{equation}    \label{eq:temp:thm:equidistribution results:cover complement of open set by neighborhood}
        \probmea{S^{2}} \setminus \openset \subseteq \bigcup_{i \in I} \neighborhoodi.
    \end{equation}
    Set $a \define \max\{ a_{\mu_{i}} \describe i \in I \}$.
    Note that $a < P(f, \potential)$.
    Applying Corollary~\ref{coro:measure-theoretic pressure infimum on local basis} with $\mu = \mu_{\potential}$ and noting that $\ratefun(\mu_{\potential}) = 0$ (recall \eqref{eq:def:rate function} and \eqref{eq:def:free energy} in Theorem~\ref{thm:level-2 large deviation principle}), we get
    \begin{equation}    \label{eq:temp:thm:equidistribution results:lower bound of limit of z plus on open set by pressure}
        P(f, \potential) \leqslant \lim_{n \to +\infty} \frac{1}{n} \log \zplus.
    \end{equation}
    Combining \eqref{eq:temp:thm:equidistribution results:lower bound of limit of z plus on open set by pressure} with Proposition~\ref{prop:characterization of pressure iterated preimages arbitrary weight}, we get that the equality holds in \eqref{eq:temp:thm:equidistribution results:lower bound of limit of z plus on open set by pressure}.
    So there exist numbers $b \in (a, P(f, \potential))$ and $N \geqslant \max\{N_{\mu_{i}} \describe i \in I\}$ such that for each integer $n \geqslant N$,
    \begin{equation}    \label{eq:temp:thm:equidistribution results:lower bound of z plus on open set by number less than pressure}
        \zplus \geqslant \myexp{n b}.
    \end{equation}

    We claim that every subsequential limit of $\{\nu_{n}\}_{n \in \n}$ in the weak$^{*}$ topology lies in the closure $\overline{\openset}$ of $\openset$.
    Assuming that the claim holds, then since $\openset \in \localbasis$ is arbitrary, we get that any subsequential limit of $\sequen{\nu_{n}}$ in the weak$^{*}$ topology is $\mu_{\potential}$, i.e., $\nu_{n} \weakconverge \mu_{\potential}$ as $n \to +\infty$.

    We now prove the claim.
    We first observe that for each $n \in \n$,\[
        \begin{split}
            \nu_{n} 
            &= \sum_{y \in f^{-n}(x_{n})} \frac{w_{n}(y) \myexp{S_{n}\potential(y)}}{\zplus + \zminus} \deltameasure{y}  \\
            &= \frac{\zplus}{\zplus + \zminus} \nu'_{n} + \sum_{ y \in f^{-n}(x_{n}), \deltameasure[n]{y} \notin \mathcal{G} } \frac{w_{n}(y) \myexp{S_{n}\potential(y)}}{\zplus + \zminus}  \deltameasure{y},
        \end{split}
    \]
    where \[
        \nu'_{n} \define \sum_{ y \in f^{-n}(x_{n}), \deltameasure[n]{y} \in \mathcal{G} } \frac{w_{n}(y) \myexp{S_{n}\potential(y)}}{\zplus} \deltameasure{y}.
    \]
    Note that since $a < b$, it follows from \eqref{eq:temp:thm:equidistribution results:cover complement of open set by neighborhood}, \eqref{eq:temp:thm:equidistribution results:upper bound for z plus}, and \eqref{eq:temp:thm:equidistribution results:lower bound of z plus on open set by number less than pressure} that\[
        0 \leqslant \lim_{n \to +\infty} \frac{\zminus}{\zplus} \leqslant \lim_{n \to +\infty} \frac{ \sum_{i \in I} \zplusneighborhood }{\zplus}  \leqslant \lim_{n \to +\infty} \frac{ \card{I} \myexp{na} }{\myexp{nb}} = 0.
    \]
    So $\lim_{n \to +\infty} \frac{\zplus}{\zplus + \zminus} = 1$, and that the total variation
    \[
        \begin{split}
            &\norm[\bigg]{ \sum_{ y \in f^{-n}(x_{n}), \deltameasure[n]{y} \notin \mathcal{G} } \frac{w_{n}(y) \myexp{S_{n}\potential(y)}}{\zplus + \zminus}  \deltameasure{y} }  \\
            &\qquad \leqslant \frac{ \sum\limits_{ y \in f^{-n}(x_{n}), \deltameasure[n]{y} \notin \mathcal{G} } w_{n}(y) \myexp{S_{n}\potential(y)} \norm{\deltameasure{y}} }{\zplus + \zminus} \\
            &\qquad \leqslant \frac{\zminus}{\zplus + \zminus} \converge 0
        \end{split}
    \]
    as $n \to +\infty$.
    Thus a measure is a subsequential limit of $\sequen{\nu_{n}}$ if and only if it is a subsequential limit of $\sequen{\nu'_{n}}$.
    Note that for each $n \in \n$, $\nu'_{n}$ is a convex combination of measures in the convex set $\openset$, so $\nu'_{n} \in \openset$.
    Hence each subsequential limit of $\sequen{\nu_{n}}$ lies in the closure $\overline{\openset}$ of $\openset$.
    The proof of the claim is complete now.

    By similar arguments as in the proof of the convergence of $\sequen{\nu_{n}}$ above, with Proposition~\ref{prop:characterization of pressure iterated preimages arbitrary weight} replaced by Proposition~\ref{prop:characterization of pressure weighted periodic points}, we get that $\mu_{n} \weakconverge \mu_{\potential}$ as $n \to +\infty$.
\end{proof}

\printbibliography

@misc{bonk2010expanding,
  title={Expanding Thurston Maps}, 
  author={Mario Bonk and Daniel Meyer},
  year={2010},
  eprint={1009.3647v1},
  archivePrefix={arXiv},
  primaryClass={math.DS}
}

@book{bonk2017expanding,
  title={Expanding {Thurston} Maps},
  author={Bonk, Mario and Meyer, Daniel},
  volume={225},
  year={2017},
  publisher={Amer. Math. Soc.},
  location={Providence, RI}
}

@article{das2021thermodynamic,
  title={Thermodynamic formalism for coarse expanding dynamical systems},
  author={Das, Tushar and Przytycki, Feliks and Tiozzo, Giulio and Urba{\'n}ski, Mariusz and Zdunik, Anna},
  journal={Communications in Mathematical Physics},
  shortjournal={{Comm.\ Math.\ Phys.}},
  volume={384},
  pages={165--199},
  year={2021},
  publisher={Springer}
}

@article{douady1993proof,
  title={A proof of {T}hurston's topological characterization of rational functions},
  author={Douady, Adrien and Hubbard, John H},
  journal={Acta Mathematica},
  shortjournal={{Acta Math.}},
  volume={171},
  pages={263--297},
  year={1993}
}

@book{folland2013real,
  title={Real Analysis: Modern Techniques and Their Applications},
  author={Folland, G.B.},
  isbn={9781118626399},
  year={2013},
  publisher={Wiley}
}

@article{haissinsky2009coarse,
  title={Coarse expanding conformal dynamics},
  author={Ha{\"\i}ssinsky, Peter and Pilgrim, Kevin M},
  journal={Ast\'{e}risque},
  number={325},
  year={2009},
}

@book{heinonen2001lectures,
  title={Lectures on Analysis on Metric Spaces},
  author={Heinonen, Juha},
  year={2001},
  publisher={Springer},
  location={New York}
}

@book{katok1995introduction,
  title={Introduction to the Modern Theory of Dynamical Systems},
  author={Katok, Anatole and Hasselblatt, Boris},
  number={54},
  year={1995},
  publisher={Cambridge Univ.\ Press},
  location={Cambridge}
}

@article{kifer1990large,
  title={Large deviations in dynamical systems and stochastic processes},
  author={Kifer, Yuri},
  journal={Transactions of the American Mathematical Society},
  shortjournal={{Trans. Amer. Math. Soc.}},
  volume={321},
  number={2},
  pages={505--524},
  year={1990}
}

@article{li2015weak,
  title={Weak expansion properties and large deviation principles for expanding {Thurston} maps},
  author={Li, Zhiqiang},
  journal={Advances in Mathematics},
  shortjournal={{Adv. Math.}},
  volume={285},
  pages={515--567},
  year={2015}
}

@book{li2017ergodic,
  title = {Ergodic Theory of Expanding Thurston Maps},
  publisher = {Atlantis Series in Dynamical Systems, Springer},
  author = {Li, Zhiqiang},
  year = {2017}
}

@article{li2018equilibrium,
  title={Equilibrium states for expanding {Thurston} maps},
  author={Li, Zhiqiang},
  journal={Communications in Mathematical Physics},
  shortjournal={{Comm.\ Math.\ Phys.}},
  volume={357},
  year={2018},
  number={2},
  pages={811--872},
  publisher={Springer}
}

@misc{li2023ground,
  title={Ground states and periodic orbits for expanding {Thurston} maps}, 
  author={Zhiqiang Li and Yiwei Zhang},
  year={2023},
  eprint={2303.00514},
  archivePrefix={arXiv},
  primaryClass={math.DS}
}

@article{lyubich1983entropy, 
  title	= {Entropy properties of rational endomorphisms of the {Riemann} sphere}, 
  volume	= {3},
  number	= {3}, 
  journal	= {Ergodic Theory and Dynamical Systems},
  shortjournal={{Ergodic Theory Dynam. Systems}}, 
  author	= {M. Ju. Lyubich}, 
  year	= {1983}, 
  pages 	= {351–385}
}

@book{przytycki2010conformal,
  title={Conformal Fractals: Ergodic Theory Methods},
  author={Przytycki, Feliks and Urba\'{n}ski, Mariusz},
  shortauthor = {Przytycki, F. and Urba\'{n}ski, M,},
  author={given=Przytycki, family=Feliks and given=Urba\'{n}ski, family=Mariusz},
  year={2010},
  publisher={Cambridge Univ.\ Press},
  location={Cambridge},
  volume={371},
}

@book{walters1982introduction,
  title={An Introduction to Ergodic Theory},
  author={Walters, Peter},
  volume={79},
  year={1982},
  publisher={Springer},
  location={New York--Berlin}
}

@article{cannon1994combinatorial,
	title = {The combinatorial {Riemann} mapping theorem},
	volume = {173},
	pages = {155--234},
	number = {2},
	journaltitle = {Acta Mathematica},
	shortjournal={{Acta Math.}},
	author = {Cannon, James W.},
	year={1994}
}

@inproceedings{bonk2006quasiconformal,
  title={Quasiconformal geometry of fractals},
  author={Bonk, Mario},
  booktitle={{Proc. Internat. Congr. Math.} (Madrid 2006)},
  volume={2},
  year={2006},
  publisher={Eur. Math. Soc., Z{\"u}rich},
  pages={1349--1373}
  //location={Z{\"u}rich}
}

@article{follmer1988large,
  title={Large deviations for the empirical field of a {Gibbs} measure},
  author={Follmer, Hans and Orey, Steven},
  journal={The Annals of Probability},
  shortjournal={{Ann. Probab.}},
  pages={961--977},
  year={1988}
}

@article{eizenberg1994large,
  title={Large deviations for {$\mathbb{Z}^{d}$}-actions},
  author={Eizenberg, Alex and Kifer, Yuri and Weiss, Benjamin},
  journal={Communications in mathematical physics},
  shortjournal={{Comm. Math. Phys.}},
  volume={164},
  pages={433--454},
  year={1994},
  publisher={Springer}
}

@article{pfister2005large,
  title={Large deviations estimates for dynamical systems without the specification property. {Application} to the $\beta$-shifts},
  author={Pfister, Charles-Edouard and Sullivan, Wayne G},
  journal={Nonlinearity},
  volume={18},
  number={1},
  pages={237--261},
  year={2005},
  publisher={IOP Publishing}
}

@article{yamamoto2009weaker,
  title={On the weaker forms of the specification property and their applications},
  author={Yamamoto, Kenichiro},
  journal={Proceedings of the American Mathematical Society},
  shortjournal={{Proc. Amer. Math. Soc.}},
  volume={137},
  number={11},
  pages={3807--3814},
  year={2009}
}

@article{iommi2015multifractal,
  title={Multifractal analysis of {Birkhoff} averages for countable {Markov} maps},
  author={Iommi, Godofredo and Jordan, Thomas},
  journal={Ergodic Theory and Dynamical Systems},
  shortjournal={{Ergodic Theory Dynam. Systems}},
  volume={35},
  number={8},
  pages={2559--2586},
  year={2015},
  publisher={Cambridge University Press}
}

@book{hatcher2002algebraic,
  title={Algebraic Topology},
  author={Hatcher, Allen},
  year={2002},
  publisher={Cambridge Univ.\ Press},
  location={Cambridge}
}

@article{bonk2024dynamical,
  title={On dynamical version of the {Kapovich--Kleiner} conjecture},
  author={Bonk, Mario and Li, Wenbo and Li, Zhiqiang},
  year={2024},
  eprint={},
  archivePrefix={arXiv},
  primaryClass={math.DS}
}

@misc{shi2023thermodynamic,
  title={Thermodynamic formalism for subsystems of expanding {Thurston} maps and large deviations asymptotics},
  author={Li, Zhiqiang and Shi, Xianghui and Zhang, Yiwei},
  year={2023},
  eprint={2312.15822},
  archivePrefix={arXiv},
  primaryClass={math.DS}
}

@misc{shi2024uniqueness,
  title={Thermodynamic formalism for subsystems of expanding {Thurston} maps {II}},
  author={Li, Zhiqiang and Shi, Xianghui},
  year={2024},
  eprint={2404.07247},
  archivePrefix={arXiv},
  primaryClass={math.DS}
}

@article{misiurewicz1973diffeomorphism,
  title={Diffeomorphism without any measure with maximal entropy},
  author={Misiurewicz, Micha{\l}},
  journal={BULLETIN DE L ACADEMIE POLONAISE DES SCIENCES-SERIE DES SCIENCES MATHEMATIQUES ASTRONOMIQUES ET PHYSIQUES},
  shortjournal={{Bull. Acad. Pol. Sci.}},
  volume={21},
  number={10},
  pages={903--910},
  year={1973}
}

@article{misiurewicz1976topological,
  title={Topological conditional entropy},
  author={Misiurewicz, Micha{\l}},
  journal={Studia Mathematica},
  shortjournal={{Studia Math.}},
  volume={55},
  number={2},
  pages={175--200},
  year={1976}
}

@book{dembo2009large,
  title={Large deviations techniques and applications},
  author={Dembo, Amir and Zeitouni, Ofer},
  volume={38},
  year={2009},
  publisher={Springer},
  location={New York}
}

@book{ellis2012entropy,
  title={Entropy, large deviations, and statistical mechanics},
  author={Ellis, Richard S},
  volume={271},
  year={2012},
  publisher={Springer},
  location={New York}
}

@book{billingsley2013convergence,
  title={Convergence of probability measures},
  author={Billingsley, Patrick},
  year={2013},
  publisher={John Wiley \& Sons},
  location={New York}
}

@book{schaefer1971locally,
  title={Topological vector spaces},
  author={Schaefer, Helmut H},
  year={1971},
  publisher={Springer},
  location={New York}
}

@book{rassoul2015course,
  title={A course on large deviations with an introduction to {Gibbs} measures},
  author={Rassoul-Agha, Firas and Sepp{\"a}l{\"a}inen, Timo},
  volume={162},
  year={2015},
  publisher={Amer. Math. Soc.},
  location={Providence, RI}
}

@article{sigmund1974dynamical,
  title={On dynamical systems with the specification property},
  author={Sigmund, Karl},
  journal={Transactions of the American Mathematical Society},
  shortjournal={{Trans. Amer. Math. Soc.}},
  volume={190},
  pages={285--299},
  year={1974}
}

@article{young1990large,
  title={Large deviations in dynamical systems},
  author={Young, Lai-Sang},
  journal={Transactions of the American Mathematical Society},
  shortjournal={{Trans. Amer. Math. Soc.}},
  volume={318},
  number={2},
  pages={525--543},
  year={1990}
}

@article{takahasi2019large,
  title={Large deviation principles for countable Markov shifts},
  author={Takahasi, Hiroki},
  journal={Transactions of the American Mathematical Society},
  shortjournal={{Trans. Amer. Math. Soc.}},
  volume={372},
  number={11},
  pages={7831--7855},
  year={2019}
}

@article{chung2019large,
  title={Large deviation principle in one-dimensional dynamics},
  author={Chung, Yong Moo and Rivera-Letelier, Juan and Takahasi, Hiroki},
  journal={Inventiones Mathematicae},
  shortjournal={{Invent. Math.}},
  volume={218},
  pages={853--888},
  year={2019},
  publisher={Springer}
}

@article{takahasi2023level,
  title={Level-2 large deviation principle for countable Markov shifts without Gibbs states},
  author={Takahasi, Hiroki},
  journal={Journal of Statistical Physics},
  shortjournal={{J. Stat. Phys.}},
  volume={190},
  number={7},
  pages={120},
  year={2023},
  publisher={Springer}
}

@article{yomdin1987volume,
  title={Volume growth and entropy},
  author={Yomdin, Yosef},
  journal={Israel Journal of Mathematics},
  shortjournal={{Israel J. Math.}},
  volume={57},
  pages={285--300},
  year={1987},
  publisher={Springer}
}

@article{newhouse1989continuity,
  title={Continuity properties of entropy},
  author={Newhouse, Sheldon E},
  journal={Annals of Mathematics},
  shortjournal = {{Ann. of Math. (2)}},
  volume={129},
  number={1},
  pages={215--235},
  year={1989},
  publisher={JSTOR}
}

@article{buzzi2014surface,
  title={surface diffeomorphisms with no maximal entropy measure},
  author={Buzzi, J{\'e}r{\^o}me},
  journal={Ergodic Theory and Dynamical Systems},
  volume={34},
  number={6},
  pages={1770--1793},
  year={2014},
  publisher={Cambridge University Press}
}

@article{iommi2022escape,
  title={Escape of entropy for countable Markov shifts},
  author={Iommi, Godofredo and Todd, Mike and Velozo, Anibal},
  journal={Advances in Mathematics},
  shortjournal={{Adv. Math.}},
  volume={405},
  pages={108507},
  year={2022},
  publisher={Elsevier}
}

@article{jenkinson2005zero,
  title={Zero temperature limits of Gibbs-equilibrium states for countable alphabet subshifts of finite type},
  author={Jenkinson, Oliver and Mauldin, R Daniel and Urba{\'n}ski, M},
  journal={Journal of Statistical Physics},
  shortjournal={{J. Stat. Phys.}},
  volume={119},
  pages={765--776},
  year={2005},
  publisher={Springer}
}

@article{buzzi2022continuity,
  title={Continuity properties of Lyapunov exponents for surface diffeomorphisms},
  author={Buzzi, J{\'e}r{\^o}me and Crovisier, Sylvain and Sarig, Omri},
  journal={Inventiones mathematicae},
  shortjournal={{Invent. Math.}},
  volume={230},
  number={2},
  pages={767--849},
  year={2022},
  publisher={Springer}
}


\end{document}